\documentclass{amsart}
\usepackage[utf8]{inputenc}
\usepackage{amsmath}
\usepackage{amsthm}
\usepackage{amsfonts}
\usepackage{amssymb}
\usepackage{mathrsfs}
\usepackage{xcolor}
\definecolor{darkgreen}{rgb}{0,0.45,0}
\definecolor{darkred}{rgb}{0.75,0,0}
\definecolor{darkblue}{rgb}{0,0,0.6}
\usepackage[colorlinks,citecolor=darkgreen,linkcolor=darkblue,urlcolor=darkred]{hyperref}
\usepackage{mathtools}
\usepackage{physics}
\usepackage{tikz}
\usepackage{tikz-cd}
\usetikzlibrary{graphs,decorations.pathmorphing,decorations.markings}

\usepackage{bbm}

\usepackage[all,cmtip]{xy}
\usepackage[shortlabels]{enumitem}

\theoremstyle{plain}
\newtheorem{thm}{Theorem}[section]
\newtheorem{lem}[thm]{Lemma}
\newtheorem{prop}[thm]{Proposition}
\newtheorem{cor}[thm]{Corollary}

\newtheorem*{thmC}{Theorem C}
\newtheorem*{reci}{The Recipe}

\theoremstyle{plain}
\newtheorem{thmA}{Theorem}

\theoremstyle{definition}
\newtheorem{conj}[thm]{Conjecture}
\newtheorem{exmp}[thm]{Example}
\newtheorem{rem}[thm]{Remark}
\newtheorem{conv}[thm]{Convention}
\newtheorem*{konj}{Conjecture}

\theoremstyle{remark}
\newtheorem*{defn}{Definition}

\newcommand{\F}{F}

\newcommand{\Jord}{\mathrm{Jord}}

\newcommand{\Sp}{\mathrm{Sp}}
\newcommand{\Ort}{\mathrm{O}}
\newcommand{\SO}{\mathrm{SO}}
\newcommand{\Mp}{\mathrm{Mp}}
\newcommand{\Irr}{{\mathrm{Irr}}}

\newcommand{\rterm}[3]{#1 \otimes S_{#2} \otimes S_{#3}}

\newcommand{\Jac}{\textnormal{Jac}}

\numberwithin{equation}{section}

\newcommand{\GL}{\operatorname{GL}}

\newcommand{\SL}{\operatorname{SL}}

\newcommand{\Hom}{\operatorname{Hom}}

\newcommand{\Ind}{\operatorname{Ind}}

\title{Theta correspondence and Arthur packets:
on the Adams conjecture}

\author[P. Baki\'{c}]{Petar Baki\'{c}}
\address{Department of Mathematics\\University of Utah\\155 S 1400 E\\Salt Lake City, UT 84112}
\email{bakic@math.utah.edu}
\author[M. Hanzer]{Marcela Hanzer}
\address{Department of Mathematics\\University of Zagreb\\
Bijeni\v{c}ka 30 \\ 10000 Zagreb, Croatia}
\email{hanmar@math.hr}

\date{}

\begin{document}

\begin{abstract}
The Adams conjecture predicts that the local theta correspondence should respect the Arthur parametrization.
In this paper, we revisit the Adams conjecture for the symplectic--even orthogonal dual pair over a nonarchimedean local field of characteristic zero. Our results provide a precise description of all situations in which the conjecture holds.
\end{abstract}

\maketitle
\thispagestyle{empty}

\setcounter{tocdepth}{1}
\tableofcontents

\normalsize

\section*{Introduction}
\label{sec_intro}The theory of theta correspondence is important in the study of automorphic forms. This paper is an attempt to understand this theory in the language of Arthur packets.\footnote{We are stealing the opening sentence from the paper which first raised this question \cite{Adams}.}

The systematic study of theta correspondence was initiated by Roger Howe in the 1970's \cite{Howe_theta_series}, building on the work of Andr\'{e} Weil \cite{Weil}. Since then, it has motivated a large body of work by numerous authors. This sustained interest stems from the fact that theta correspondence is one of the few ways to explicitly construct automorphic forms and representations. Notably, it has been used to construct many instances of Langlands functoriality.
In this paper, we study the local version of this correspondence.

To recall the basic idea, we need two ingredients. The first one is a reductive dual pair inside a symplectic group. Let $F$ be a non-archimedean local field of characteristic zero. We fix $\epsilon=\pm 1$ and let 
\begin{align*}
W_n &= \text{a }(-\epsilon)\text{-Hermitian space of even dimension }n\; \text{over } F,\\
V_m &= \text{an }\epsilon\text{-Hermitian space of even dimension }m\; \text{over } F.
\end{align*}
We let $G=G_n$ (resp.~$H=H_m$) denote the isometry group of $W_n$ (resp.~$V_m$). Tensoring the two bilinear forms, we get a symplectic form on $W\otimes V$ which is preserved by both $G$ and $H$. Moreover, $G$ and $H$ form a reductive dual pair inside $\Sp(W\otimes V)$: each one is the centralizer of the other. The second ingredient we need is the so-called Weil representation. The symplectic group has a unique non-trivial double cover, called the metaplectic group. The Weil representation $\omega$ is a representation of $\Mp(W\otimes V)$, the metaplectic group associated with $W\otimes V$. Crucially, there exists a splitting $G\times H \to \Mp(W\otimes V)$. Pulling back the Weil representation along this splitting, one obtains a representation $\omega_{W,V}$ of $G\times  H$. The goal is to analyze $\omega_{W,V}$.

For an irreducible representation $\pi$ of $G$, the maximal $\pi$-isotypic quotient of $\omega_{W,V}$ is of the form
\[
\pi \otimes \Theta(\pi),
\]
for some smooth representation $\Theta(\pi)$ of $H$. The representation $\Theta(\pi)$ is called the big theta lift of $\pi$; an early result of Kudla \cite{Kudla2} shows that it has finite length. When non-zero, $\Theta(\pi)$ has a unique irreducible quotient, denoted $\theta(\pi)$ (we set $\theta(\pi)=0$ when $\Theta(\pi)=0$). This remarkable result, known as the \emph{Howe duality conjecture}, was first formulated by Howe \cite{Howe_theta_series}, proven by Waldspurger \cite{Waldspurger_howe_duality} in odd residue characteristic and by Gan and Takeda \cite{Gan_Takeda_proof_of_Howe} in full generality. The map $\pi \mapsto \theta(\pi)$ is called the local theta correspondence.

There are two basic questions concerning this construction: a) For which $\pi$ is $\Theta(\pi)\neq 0$? and b) What is $\theta(\pi)$, when it is non-zero? An early result in this direction was proved by M\oe glin, Vign\'{e}ras and Waldspurger \cite{MVW_Howe} for cuspidal representations. Similar results were obtained by Muić \cite{Muic_Israel} for discrete series representations. These results were obtained before the local Langlands classification (LLC) was available. In modern terms, we typically ask: Given an $L$-parameter corresponding to $\pi$, can we tell whether $\theta(\pi)$ is non-zero? Furthermore, can we determine the $L$-parameter of $\theta(\pi)$?
Nowadays, complete answers to questions a) and b) are known: in \cite{Atobe_Gan} Atobe and Gan answered them for tempered representations; in \cite{nas_clanak}, we extended these results to obtain answers for general irreducible representations.

\noindent  Although these results offer complete answers, in some sense they could never tell the whole story. The reason is that theta correspondence does not respect $L$-packets: two representations from the same packet can lift to representations belonging to different $L$-packets\,---\,this much has been known from the early days of theta correspondence. In his 1989 paper \cite{Adams}, Adams proposed a solution to this problem: instead of $L$-packets, consider Arthur packets! Indeed, by enlarging the packets, one could hope to obtain a nicer formula for the correspondence. To explain this idea, we briefly recall the notion of (local) Arthur packets; we refer the reader to Section \ref{subs_Arthur} for a more detailed overview.

In his work \cite{Arthur_endoscopic}, Arthur proposed a classification of square-integrable automorphic representations of $G$. The central notion is that of an Arthur parameter. Each parameter corresponds to a (global) Arthur packet of automorphic representations. For each local place, the global parameter restricts to a local Arthur parameter, which is a map 
\[
\psi: L_F \times \SL_2(\mathbb{C}) \to \prescript{L}{}{G}
\]
subject to certain technical requirements. Here $L_F$ denotes the Weil--Deligne group of $F$, and $\prescript{L}{}{G}$ is the (complex) dual group of $G$:
\[
\prescript{L}{}{G} = 
\begin{cases}
\SO(n+1, \mathbb{C}), &\quad \text{if }G \text{ is symplectic (i.e.~if } \epsilon = 1);\\
\Ort(n, \mathbb{C}), &\quad \text{if }G \text{ is even orthogonal (i.e.~if } \epsilon = -1).
\end{cases}
\]
To each such parameter, one attaches a local Arthur packet, i.e.~a packet $\Pi_\psi$ of irreducible representations of $G$. The representations in $\Pi_\psi$ are precisely those which appear as local constituents of the corresponding global Arthur packet.
Composing $\psi$ with the standard representation of $\prescript{L}{}{G} $, one may view $\psi$ as a representation of $L_F \times \SL_2(\mathbb{C})$. 

A critical contribution to this theory was that of M\oe glin \cite{moeglin2006paquets,moeglin2006certains,moeglin2009comparaison,moeglin2010holomorphie,moeglin2011multiplicite1}, who constructed the local Arthur packets $\Pi_\psi$ and showed that they are multiplicity free. To classify the different irreducible representations in $\Pi_\psi$, M\oe glin uses additional data $(\eta, t)$ (see \S \ref{subs_Arthur}). Given a parameter $\psi$, each choice of $(\eta, t)$ yields a representation $\pi_{(\eta, t)}$. This representation is either irreducible or $0$, and the non-zero representations obtained this way constitute the local A-packet $\Pi_\psi$. The key technical question of whether $\pi_{(\eta, t)}$ is non-zero has been resolved by Xu \cite{xu2021combinatorial} and more recently by Atobe \cite{atobe2022construction}.

\noindent We now explain the Adams conjecture. Let $\pi$ be an irreducible representation of $G=G_n$. Recall that we are interested in computing the lift to $H=H_m$; we usually denote this lift by $\theta_{-\alpha}(\pi)$, where  $\alpha = m-n-\epsilon$. Suppose $\alpha >0$, and let $S_\alpha$ denote the irreducible (algebraic) $\alpha$-dimensional representation of $\text{SL}_2(\mathbb{C})$. We then have:\footnote{See \S \ref{sec_conj} for a fully precise statement; to simplify the exposition, we are not mentioning the characters needed to fix the splitting $G\times H \to \Sp(W\otimes V)$}
\begin{konj}[Adams \cite{Adams}]
Assume that $\pi$ is  parametrized by $\psi$. Then $\theta_{-\alpha}(\pi)$ is pa\-ra\-metrized by
\[
\psi_\alpha = \psi\ \oplus \ 1 \otimes S_\alpha.
\]
\end{konj}
\noindent Note that this indeed gives a homomorphism into $\prescript{L}{}{H}$; moreover, this is the simplest way to obtain a map into $\prescript{L}{}{H}$ using $\psi$ as the input! 

\noindent Adams himself verified the conjecture in \cite{Adams} for all examples of theta correspondence available at the time. However, subsequent work on the local theta correspondence led to new examples against which the conjecture could be tested. In her paper \cite{moeglinkudla}, M\oe glin revisited the work of Adams. She showed that
\begin{itemize}
    \item the above conjecture is true for large $\alpha$ (see Proposition \ref{prop_visoko});
    \item the above conjecture fails in many examples.
\end{itemize}
This failure of the Adams conjecture is the point of departure for the present paper. To be precise, given a representation $\pi$ in $\Pi_\psi$, we consider the set of those $\alpha$ for which $\theta_{-\alpha}(\pi)$ is in the expected packet:
   \[\mathcal A(\pi,\psi) = \{\alpha\geq 0,\  \alpha\equiv 1 (\text{mod } 2): \theta_{-\alpha}(\pi) \in \Pi_{\psi_\alpha}\}\]
M\oe glin (see \cite[Section 6.3]{moeglinkudla}) then asks two crucial questions:
\begin{itemize}
    \item[1)] Can we describe  $\mathcal A(\pi,\psi)$? In particular, does $\alpha\in \mathcal A(\pi,\psi)$ imply $\alpha+2\in \mathcal A(\pi,\psi)$?
    \item[2)] If so, can we find $a(\pi,\psi) := \min \mathcal A(\pi,\psi)$ explicitly?
\end{itemize}
In this paper, we answer both questions. In particular, we show that $\alpha\in \mathcal A(\pi,\psi)$ does imply $\alpha+2\in \mathcal A(\pi,\psi)$, and we develop a method for computing $a(\pi,\psi)$. 
We provide a rough overview of our results here, and leave the details for Section \ref{sec_conj}. 

Using M\oe glin's results on lifts for $\alpha \gg 0$ as a starting point, we  define a representation $\pi_\alpha$ for every odd $\alpha > 0$. By construction, there exists an odd number, denoted $d(\pi,\psi)$, such that $\pi_\alpha$ is an element of $\Pi_{\psi_\alpha}$ for $\alpha \geq d(\pi,\psi)$, and $\pi_\alpha=0$ for $\alpha < d(\pi,\psi)$.
Moreover, for $\alpha \gg 0$ we have $\pi_\alpha = \theta_{-\alpha}(\pi)$. Our first result is
\begin{thmA}[Theorem \ref{theoremA}] Let $\alpha > 1$ be odd.
Assume $\pi_\alpha = \theta_{-\alpha}(\pi)$ and $\pi_{\alpha-2} \neq 0$. Then $\theta_{-(\alpha-2)}(\pi)=\pi_{\alpha-2}$; in particular, $\theta_{-(\alpha-2)}(\pi)\neq 0$ is in $\Pi_{\psi_{\alpha-2}}$.
\end{thmA}
\noindent Theorem \ref{theoremA} allows us to start at an arbitrary large value of $\alpha$ (where we know the conjecture is true), and inductively decrease $\alpha$ to show that the Adams conjecture holds for all odd $\alpha \geq d(\pi,\psi)$.

To utilize the so-called conservation relation (see \S\ref{subs_theta}), we simultaneously look at two towers of lifts. On one of the towers (called the going-down tower) the lifts start appearing early; on the other (the going-up tower), they appear late. We prove
\begin{thmA}[Theorem \ref{theoremB}]
     On the going-up tower, $d^\text{up}(\pi,\psi)$ corresponds to the first occurrence of $\pi$:
    \[d^\text{up}(\pi,\psi) = \min\{\alpha>0: \theta_{-\alpha}^\text{up}(\pi)\neq 0\}.\]
     Consequently, the Adams conjecture is true for all non-zero lifts. Moreover, $d^\text{down}(\pi,\psi) \leq d^\text{up}(\pi,\psi)$, and equality holds if and only if $d^\text{up}(\pi,\psi) =1$.
\end{thmA}
\setcounter{thmA}{0}
\noindent Finally, we prove
\begin{thmC}[Theorem \ref{theoremC}]
The Adams conjecture never holds for $\alpha  < d(\pi,\psi)$.
\end{thmC}
Theorems A, B, and C thus resolve questions 1) and 2) posed above: we see that $\alpha\in \mathcal A(\pi,\psi)$ implies $\alpha+2\in \mathcal A(\pi,\psi)$, and moreover, $a(\pi,\psi) = d(\pi,\psi)$. 

We describe our methods. Considering M\oe glin's construction of Arthur packets, it is not surprising that we need to have good control of Jacquet modules. We thus make frequent use of Kudla's filtration (which describes the Jacquet modules of the Weil representation), but also some general results on Jacquet modules for representations of Arthur type. These are discussed in \S \ref{sec_tech}. Given that an explicit description of the theta correspondence is available in terms of the LLC \cite{nas_clanak}, one might conceivably use this to study the Adams conjecture in a roundabout manner: find the $L$-parameter of a representation $\pi \in \Pi_\psi$, compute $\theta_{-\alpha}(\pi)$, and check whether it is in $\Pi_{\psi_\alpha}$. Using a direct approach and staying within the Arthur world seemed to us the easier of the two possible approaches. Still, we make use of the results from \cite{nas_clanak} which are therefore reviewed in the Appendix. Of course, our results also rely on those of Xu: as explained above, we ultimately show that the Adams conjecture is valid if and only if $\alpha \geq d(\pi,\psi)$. But $d(\pi,\psi)$ is defined precisely by the non-vanishing of a certain representation $\pi_{\alpha}$. The algorithm developed by Xu (and refined by Atobe) is what makes this a satisfactory criterion. 

To prove that $\theta_{-\alpha}(\pi)$ is in a certain Arthur packet, it suffices to identify a candidate representation in this packet and show that it is isomorphic to $\theta_{-\alpha}(\pi)$; this is what makes Theorem A relatively easy. The proof of Theorem B uses a trick (used a few times throughout the paper) in order to compare the relevant Arthur parameter to a simpler one. It is a nice demonstration of how the results from \cite{nas_clanak} give precise qualitative information about theta lifts. Finally, Theorem C seems to us inevitably difficult: unlike Theorem A, here we prove that our theta lift \emph{cannot} be isomorphic to a representation from a given Arthur packet. Our proof is a showcase of the main technical idea of this paper, namely that theta lifts can be used to alter Arthur packets in a controlled way. This method of using theta lifts to perform ``surgery on Arthur packets" is expounded in Section \ref{subs_removal}. It is quite general, and can be used in any problem involving local Arthur packets. Aside from our main results, it is the main novelty and the main technical contribution of the present paper.

Although we work with symplectic--even orthogonal dual pairs, our methods are applicable to other dual reductive pairs, provided the local Arthur packets are defined. Indeed, our proofs rely only on M\oe glin's construction of Arthur packets and on results about theta correspondence which hold for all dual pairs of type I. Assuming Arthur's (local) parametrization of representations in terms of A-packets (thus assuming all the issues concerning endoscopy, transfers, fundamental lemmas are resolved), M\oe glin \cite{moeglin2011multiplicite1} describes the classes of groups for which her explicit construction of A-packets applies. We thus expect the same results to hold for the metaplectic--odd orthogonal dual pair, as well as for unitary dual pairs.

In recent years, there has been considerable progress in understanding local Arthur packets. We mention some relevant work. Xu \cite{xu2021combinatorial} describes an algorithm which determines whether a given representation $\pi_{(\eta,t)}$ is non-zero; this answers the crucial question in M\oe glin's construction. The work of Atobe \cite{atobe2022construction} reinterprets those results, and constructs local A-packets explicitly. Finally, Atobe \cite{Atobe2} and Hazeltine--Liu--Lo \cite{Purdue} compute the set of A-packets which contain a given irreducible representation (in particular, given an irreducible representation, these results determine whether or not it is of Arthur type). 

There are some related questions that are not addressed by the present paper. For example, one could ask for a description of theta correspondence in terms of Arthur packets when lifting to groups of smaller rank (i.e.~for $\alpha<0$). Although this question is not studied by Adams \cite{Adams}, it would be interesting to have a description of lifts, or at least a criterion for non-vanishing of lifts that does not involve translating the question to the language of $L$-parameters. We suspect the answers in this case are unlikely to be as tidy as the ones provided by the Adams conjecture. Another interesting question is the following: if the Adams conjecture fails and the lift is not in the \emph{expected} Arthur packet, could $\theta_{-\alpha}(\pi)$ still be a representation of Arthur type (just lying in some other, unexpected packet)? Finally, we mention a question related to the work of Hazeltine--Liu--Lo \cite{Purdue}. Given a representation $\pi$ of Arthur type, one could consider all the parameters $\psi$ such that $\pi \in \Pi_\psi$. Our results attach an index $d(\pi,\psi)$ to each of these parameters. Notice that, on the going-up tower, the various $d(\pi,\psi)$ are independent of $\psi$ (they depend only on the first occurrence of $\pi$). However, on the going down-tower, they might vary with $\psi$. We suspect that $d(\pi,\psi)$ will be lower if $\psi$ is “more tempered”, suggesting a connection between theta correspondence and the notion of “\emph{the} Arthur packet” discussed in \cite{Purdue}. Exploring this connection would be an interesting problem. These and similar questions fall beyond the scope of the present paper.

\noindent Adams's original motivation for stating the conjecture was explaining the theta correspondence. However, it turned out that an explicit description of the correspondence in terms of $L$-parameters was found before the conjecture was fully investigated. It is fair to ask why one would bother obtaining a less precise description (using Arthur packets) if a fully precise recipe in terms of $L$-parameters is already available. The reasons for this are twofold. The first reason is technical: as shown by the work of Atobe \cite{atobe2022construction}, going from Arthur- to $L$-parameters and back is possible, but quite involved. For most applications involving Arthur packets, it is therefore valuable to have a recipe for theta correspondence that bypasses these translation issues. The second reason is equally interesting: the nature of local Arthur packets remains relatively mysterious, as they are defined by endoscopic character identities. However, the simplicity of Adams's formula suggests a possibly deep connection between theta correspondence and Arthur packets, that one could exploit in either direction. It is our hope that this work sheds some light on representations of Arthur type.

\subsection*{Outline} In Section \ref{sec_preliminaries} we introduce the objects and the notation we use in the paper. We give an overview of the theta correspondence as well as M\oe glin's construction of local Arthur packets.
In Section \ref{sec_conj} we state the Adams conjecture and the relevant questions; we provide a detailed description of our results. Section \ref{sec_example} contains examples which further illustrate the main results. Most of the auxiliary technical results are collected in Section \ref{sec_tech}. Finally, we prove the main results\,---\,Theorems A, B, and C\,---\,in Sections \ref{sec_thmA}, \ref{sec_thmB}, and \ref{sec_thmC}, respectively. Appendix \ref{appendix} contains a summary of the results from \cite{nas_clanak} that are used in proving Theorem C.

\subsection*{Acknowledgments} We would like to express our gratitude to Wee Teck Gan for suggesting this project as well as for his continued interest and support. We would also like to thank Hiraku Atobe and Rui Chen for useful communications. This work is supported in part by the Croatian Science Foundation under the project IP-2018-01-3628.

\section{Preliminaries and Notation}
\label{sec_preliminaries}
In this section we introduce the objects and the notation used throughout the paper.
\subsection{Groups}
\label{subs_groups}
Let $F$ be a nonarchimedean local field of characteristic $0$ and let $|\cdot|$ be the absolute value on $F$ with the usual normalization. We fix $\epsilon=\pm 1$ and let 
\[
\begin{cases}
W_n = \text{a }(-\epsilon)\text{-Hermitian space of even dimension }n\; \text{over } F,\\
V_m = \text{an }\epsilon\text{-Hermitian space of even dimension }m\; \text{over } F.
\end{cases}
\]
We let $G=G(W_n)$ (resp.~$H=H(V_m)$) denote the isometry group of $W_n$ (resp.~$V_m$). Thus $G$ is a symplectic group when $\epsilon = 1$, and a (full) orthogonal group when $\epsilon=-1$.
If $W$ is a symplectic space, we shall also consider the metaplectic group $\textnormal{Mp}(W)$: this is the unique non-trivial two-fold cover of $\Sp(W)$; cf.~\cite{Kudla1}, \cite{MVW_Howe}. If $X$ is a vector space over $F$, we use $\textnormal{GL}(X)$ to denote the general linear group of $X$. All the groups defined here are totally disconnected locally compact topological groups.

\subsection{Witt towers}
\label{subs_witt}
Every $\epsilon$-Hermitian space $V_m$ has a Witt decomposition
\[
V_m = V_{m_0} + V_{r,r} \quad (m = m_0 + 2r),
\]
where $V_{m_0}$ is anisotropic and $V_{r,r}$ is split (i.e.~a sum of $r$ hyperbolic planes). The space $V_{m_0}$ is unique up to isomorphism, and so is the number $r \geqslant 0$ (called the Witt index of $V_m$). The collection of spaces
\[
\mathcal{V} = \{V_{m_0} + V_{r,r} \colon r \geqslant 0\}
\]
is called a Witt tower.

\subsection{Parabolic subgroups}
\label{subs_parabolics}
Let $V_m$ be an $\epsilon$-Hermitian space of Witt index $r$. We fix a set of standard parabolic subgroups $Q_t$, $t=1,\dotsc,r$, of $H(V_m)$ like in \cite[\S 2.3]{nas_clanak}. The parabolic subgroup $Q_t$ has a Levi decomposition $Q_t=M_tN_t$ with Levi factor $M_t$ isomorphic to $\GL_t(F) \times H(V_{m-2t})$. By further partitioning $t$ we get the remaining standard parabolic subgroups. The standard maximal parabolic subgroups of $G(W_n)$ are denoted $P_t$.

\subsection{Representations}
\label{subs_reps}

Let $G = G(W_n)$ be one of the groups introduced in \S\ref{subs_groups}. We work in the category of smooth complex representations of $G$. The set of equivalence classes of irreducible representations of $G$ will be denoted by $\Irr(G)$. We denote the contragredient of $\pi$ by $\pi^\vee$.

For each parabolic subgroup $P=MN$ of $G$ we have the (normalized) induction and Jacquet functors, denoted by $\Ind_P^G$ and $r_P$. These are related by the standard Frobenius reciprocity
\[
\Hom_G(\pi, \Ind_P^G(\pi')) \cong \Hom_M(r_P(\pi), \pi')
\]
and by the second (Bernstein) form of Frobenius reciprocity,
\[
\Hom_G(\Ind_P^G(\pi'), \pi) \cong \Hom_M(\pi', r_{\overline{P}}(\pi)).
\]
Here $\overline{P} = M\overline{N}$ is the parabolic subgroup opposite to $P$. In the rest of \S \ref{subs_reps} we set up the notation and establish some basic facts about  induced representations and Jacquet functors.

\subsubsection{Notation for induction}
\label{subsubs_notation_ind}
Let $P=MN$ be the standard parabolic subgroup of $\GL_t(F)$ with Levi factor $M = \textnormal{GL}_{t_1}(F) \times \dotsm \times \textnormal{GL}_{t_k}(F)$ ($t = t_1 +\dotsb+t_k$), and let $\tau_i$ be a representation of $\textnormal{GL}_{t_i}(F)$ for $i=1,\dotsc,k$. We use 
\[
\tau_1 \times \dotsm \times \tau_k
\]
to denote $\Ind_P^{\textnormal{GL}_{t}}(\tau_1 \otimes \dotsm \otimes \tau_k)$. Similarly, let $P=MN$ be the standard parabolic subgroup of $G(W_n)$ with Levi factor $M = \textnormal{GL}_{t_1}(F) \times \dotsm \times \textnormal{GL}_{t_k}(F) \times G(W_{n-2t})$, and let $\pi_0$ a representation of $G(W_{n-2t})$. In this case we use
\[
\tau_1 \times \dotsm \times \tau_k \rtimes \pi_0
\]
to denote $\Ind_P^G(\tau_1 \otimes \dotsm \otimes \tau_k \otimes \pi_0)$.

\subsubsection{Segments}
\label{subsubs_segments}
Let $\rho$ be an irreducible unitary supercuspidal representation of $\GL_k(F)$ and let $x,y$ be real numbers such that $y-x$ is a non-negative integer. Any tuple of representations of the form
\[
(\rho|\cdot|^x,\rho|\cdot|^{x+1},\dotsc, \rho|\cdot|^{y})
\]
is called a segment. We denote such a segment by $[x,{y}]_\rho$; when $\rho$ is clear from the context, we further simplify the notation and use $[x,y]$. The representation
\[
\rho|\cdot|^x \times \rho|\cdot|^{x+1} \times \dotsb \times  \rho|\cdot|^{y}
\]
has a unique irreducible quotient, and a unique irreducible subrepresentation. We denote the quotient by $\delta([x, {y}]_\rho)$, and the subrepresentation by $\zeta([x, {y}]_\rho)$. When $\rho$ is trivial these are usually shortened to $\delta(x,y)$ and $\zeta(x,y)$, respectively.
When $y=x-1$, these are to be interpreted as denoting the trivial representation of the trivial group.

Recall that the segments $[x,y]_\rho$ and $[x',y']_{\rho'}$ are said to be linked if neither contains the other, and their union is itself a segment. Similarly, they are said to be juxtaposed if their intersection is empty, but their union is a segment. We use the following well known result throughout the paper.
\begin{lem} Let $[x, {y}]_\rho$ and $[x',y']_{\rho'}$ be two segments.
\begin{enumerate}[(i)]
    \item The representation $\delta([x, {y}]_\rho) \times \delta([x',y']_{\rho'})$ reduces if and only if the segments $[x, {y}]_\rho$ and $[x',y']_{\rho'}$ are linked. The same holds if we replace $\delta$ by $\zeta$.
    \item The representation $\delta([x, {y}]_\rho) \times \zeta([x',y']_{\rho'})$ reduces if and only if the segments $[x,y]$ and $[x',y']$ are juxtaposed.
\end{enumerate}
\end{lem}
\begin{proof}
Part (i) is shown in \cite{Zelevinsky_GL}, and (ii) is Lemma I.6.3 of \cite{Moeglin1989}. See also Section 3.1 of \cite{nas_clanak} for a more detailed analysis of the situation in (ii). 
\end{proof}

\noindent This construction extends to generalized segments. Fix $\zeta = \pm 1$. A generalized segment (of dimensions $m\times n$) is a matrix
\[
\begin{bmatrix}
x_{11} & \dotsb & x_{1n}\\
\vdots & \ddots & \vdots\\
x_{m1} & \dotsb & x_{mn}
\end{bmatrix}_\rho
\]
with $x_{i+1,j} = x_{ij} - \zeta$ and $x_{i,j+1} = x_{ij} + \zeta$, for all $i,j$. Let $\sigma_i$ be the unique irreducible subrepresentation of $\rho|\cdot|^{x_{i1}} \times \dotsb \times \rho|\cdot|^{x_{in}}$, for $i=1,\dotsc, m$. The representation attached to the above generalized segment is defined as the unique irreducible subrepresentation of $\sigma_1\times \dotsc \times \sigma_m$.
We abuse notation and simply use the above matrix to denote the corresponding representation. As before, when $\rho$ is clear from the context, we suppress it from the notation. We note that the matrix $[x_{ij}]_\rho$ and its transpose $[x_{ij}]_\rho^t$ correspond to the same representation; if not otherwise stated, we will assume that all the segments are written with decreasing rows and increasing columns.



\subsubsection{Notation for Jacquet functors}
\label{subsubs_Jacquet}

The construction of local Arthur packets involves repeated use of various Jacquet functors. To simplify the exposition, we introduce a convenient way to denote the Jacquet modules in question. Variations of this notation have become standard in the literature on local Arthur packets. Given $\pi \in \Irr (G(W_n))$ and a standard maximal parabolic $P$ with Levi factor $\GL_t(F) \times G(W_{n-2t})$, we may write the semisimplification of $r_P(\pi)$ as a direct sum
\[
\bigoplus_i\ \tau_i \otimes \sigma_i
\]
with $\tau_i \in \Irr (\GL_t(F))$ and $\sigma_i \in \Irr (G(W_{n-2t}))$. Fix $\rho$, an irreducible unitary supercuspidal representation of $\GL_t(F)$. For any real number $x$ we define
\[
\Jac^\rho_x (\pi) := \bigoplus_{\tau_i = \rho|\cdot|^x} \sigma_i.
\]
We extend this notation as follows. For a tuple $(x_1,\dotsc,x_k)$ of real numbers, we set
\[
\Jac^\rho_{x_1,\dotsc,x_k} := \Jac^\rho_{x_k} \circ \Jac^\rho_{x_{k-1}} \circ \dotsb \circ \Jac^\rho_{x_1}.
\]
Finally, let $[B,A]_\rho$ be a segment (cf.~\S\ref{subsubs_segments}) and $T>0$ an integer. For $\zeta = \pm 1$ we set
\[
\Jac^\rho_{\zeta[B+T,A+T]\to \zeta[ B, A]} :=\Jac^\rho_{\zeta (B+1),\dotsc,\zeta (A+1)} \circ  \dotsb \circ \Jac^\rho_{\zeta (B+T-1),\dotsc,\zeta (A+T-1)}\circ\Jac^\rho_{\zeta (B+T),\dotsc,\zeta (A+T)}.
\]
See also Remark \ref{rem_zgusnjavanje} for additional notation related to Jacquet modules. When $\rho$ is implied from the context, we omit it and write $\Jac_x$, $\Jac_{x_1,\dotsc,x_k}$, $\Jac_{\zeta[B+T, A+T]\to \zeta [ B, A]}$.

\subsubsection{The Langlands classification}
\label{subsubs_Langlands}
We briefly recall the Langlands classification. Let $\delta_i \in \textnormal{GL}_{t_i}(F), i = 1,\dotsc, k$ be irreducible discrete series (unitarizable) representations, and let $\tau$ be an irreducible tempered representation of $G(W_{n-2t})$, where $t=t_1+\dotsb+t_k$. Any representation of the form
\[
\nu^{s_k}\delta_k \times \dotsb \times \nu^{s_1}\delta_1 \rtimes \tau,
\]
where $s_k \geqslant \dotsb \geqslant s_1 > 0$ (and where $\nu$ denotes the character $\lvert\det\rvert$ of the corresponding general linear group) is called a standard representation (or a standard module). It possesses a unique irreducible quotient (the so-called Langlands quotient), denoted by $L(\nu^{s_k}\delta_k, \dotsc, \allowbreak \nu^{s_1}\delta_1; \allowbreak \tau)$. Conversely, every irreducible representation of $G(W_n)$ can be represented as the Langlands quotient of a unique standard representation. In this way, one obtains a complete description of $\Irr(G(W_n))$.

\subsection{The theta correspondence}
\label{subs_theta}
We recall the basic facts about the local theta correspondence.

We begin by fixing, once and for all, a non-trivial additive character $\psi$ of $F$. The space $W_n\otimes V_m$ has a natural symplectic structure, and the choice of $\psi$ determines a Weil representation of $\Mp(W_n\otimes V_m)$. We attach a pair of characters $(\chi_W,\chi_V)$ to the spaces $W_n$ and $V_m$ as in Section 3.2 of \cite{Gan_Ichino_formal_degrees}; note that the characters depend only on the Witt tower, and not on the dimension of the space. These characters are needed to define the splitting of the metaplectic cover:
\[
G(W_n) \times H(V_m) \rightarrow \Mp(W_n\otimes V_m).
\]
Pulling back the Weil representation, we obtain the Weil representation of the dual pair $G(W_n) \times H(V_m)$. We denote this representation by $\omega_{m,n}$; since $\chi_W, \chi_V$ and $\psi$ are fixed throughout the paper, we omit them from the notation. In fact, if the dimensions are known (or irrelevant), we further simplify the notation and use just $\omega$.

For any $\pi \in \Irr(G(W_n))$, the maximal $\pi$-isotypic quotient of $\omega_{m,n}$ is of the form
\[
\pi \otimes \Theta(\pi,V_m)
\]
for a certain smooth representation $\Theta(\pi,V_m)$ of $H(V_m)$, called the full theta lift of $\pi$. When the target Witt tower is fixed, we will denote it by $\Theta_\alpha(\pi)$, where $\alpha = n + \epsilon - m$ (recall that $\epsilon$ is defined in \S\ref{subs_groups}). Note that $\alpha$ is an odd integer.

A classic result of Kudla \cite{Kudla2} shows that $\Theta(\pi,V_m)$ is either zero, or an admissible representation of finite length. The following theorem establishes the theta correspondence:
\begin{thm}[Howe duality]
\label{thm_Howe}

If $\Theta(\pi,V_m)$ is non-zero, it possesses a unique irreducible quotient, denoted by $\theta(\pi,V_m)$.
\end{thm}
\noindent Originally conjectured by Howe in \cite{Howe_theta_series}, this was first proven by Waldspurger \cite{Waldspurger_howe_duality} when the residual characteristic of $F$ is different from $2$, and by Gan and Takeda \cite{Gan_Takeda_proof_of_Howe} in general.
The representation $\theta(\pi,V_m)$ is called the (small) theta lift of $\pi$; like the full lift, we will also denote it by $\theta_\alpha(\pi)$. The resulting bijection $\pi \leftrightarrow \theta(\pi)$ between representations appearing as quotients of $\omega$ is called the (local) theta correspondence. Additionally, it will be convenient to set $\theta(\pi,V_m)=0$ when $\Theta(\pi,V_m)=0$; we often use this convention.

A standard approach to the theory of theta correspondence involves considering towers of lifts (see Propositions 4.1 and 4.3 of \cite{Kudla1}):
\begin{prop}[Tower property]
\label{prop_towers}
Let $\pi$ be an irreducible representation of $G(W_n)$. Fix a Witt tower $\mathcal{V} = (V_m)$.
\begin{enumerate}[(i)]
\item If $\Theta(\pi,V_m) \neq 0$, then $\Theta(\pi,V_{m+2r}) \neq 0$ for all $r\geqslant 0$.
\item For $m$ large enough, we have $\Theta(\pi,V_m) \neq 0$.
\end{enumerate}
\end{prop}
The above proposition allows us to define, for any Witt tower $\mathcal{V} = (V_m)$,
\[
m_\mathcal{V}(\pi) := \min\{m \geqslant 0: \Theta(\pi,V_m) \neq 0\}.
\]
This number\,---\,also denoted $m(\pi)$ when the choice of $\mathcal{V}$ is implicit\,---\,is called the first occurrence index of $\pi$. Note the slight abuse of terminology: we are using the term “index” even though $m(\pi)$ is the dimension.

A refinement of the tower property is the so-called conservation relation. When $\epsilon = 1$, the Witt towers of quadratic spaces we consider can be appropriately organized into pairs, with the towers comprising a pair denoted $\mathcal{V}^+$ and $\mathcal{V}^-$; a complete list of pairs of towers can be found in \cite[Chapter V]{Kudla1}. Thus, instead of observing just one target tower, we simultaneously look at two of them. This way, each $\pi \in \Irr(G(W_n))$ gives us two first occurrence indices, $m^+(\pi)$ and $m^-(\pi)$.

If $\epsilon = -1$, there is only one tower of $\epsilon$-hermitian (i.e.~symplectic) spaces. In this case $W_n$ is a quadratic space,  and we proceed as follows: since $G(W_n)$ is now equal to $O(W_n)$, any $\pi \in \Irr(G(W_n))$ is naturally paired with its twist, $\det \otimes \pi$. This allows us to define
\begin{align*}
m^+(\pi) &= \min\{m(\pi), m(\det\otimes\pi)\},\\
m^-(\pi) &= \max\{m(\pi), m(\det\otimes\pi)\}.
\end{align*}
Thus, regardless of whether $\epsilon = 1$ or $-1$, we may set
\[
m^{\text{down}}(\pi) =  \min\{m^+(\pi), m^-(\pi)\}, \quad m^{\text{up}}(\pi) =  \max\{m^+(\pi), m^-(\pi)\}.
\]
Note that when $V_m$ is a symplectic space, we have $m^{\text{down}}(\pi) = m^+(\pi) $ and $ m^{\text{up}}(\pi) = m^-(\pi)$. The Conservation relation (first conjectured by Kudla and Rallis in \cite{Kudla_Rallis_progress}, and ultimately proven by Sun and Zhu in \cite{Sun_Zhu_conservation}) states that
\[
m^{\text{up}}(\pi) + m^{\text{down}}(\pi) = 2n + 2\epsilon + 2.
\]
The tower in which $m(\pi) = m^{\text{down}}(\pi)$ (resp.~$m^{\text{up}}(\pi)$) is called the going-down (resp.~going-up) tower.
\begin{rem}
\label{rem_updown}
This labeling of towers is slightly imprecise. Indeed, the conservation relation implies 
\[
m^{\text{down}}(\pi) \leq n+\epsilon+1\quad \text{and}\quad m^{\text{up}}(\pi) \geq n+\epsilon+1.
\]
However, it may well happen that $m^{\text{down}}(\pi)= m^{\text{up}}(\pi) = n+\epsilon+1$, in which case the “going-up” and “going-down” designations are ambiguous. This ambiguity can be resolved; see e.g.~Theorem 4.1 (2) in \cite{Atobe_Gan}. However, we do not have to worry about this: the questions we study in this paper simplify dramatically when $\pi$ satisfies the above equality; in this case we can describe the results without specifying the target tower.
\end{rem}

\begin{conv}
\label{conv_dettwist}
A convenient feature of our notation is that $\theta_\alpha(\theta_{-\alpha}(\pi)) = \pi$, whenever $\theta_{-\alpha}(\pi)\neq 0$. In fact, this is \emph{almost} true: if $\pi$ is a representation of an orthogonal group, then (depending on the choice of tower), computing the lift $\theta_{-\alpha}(\pi)$ might actually mean computing the lift of $\det \otimes \pi$. Therefore, in this case, we might get $\theta_\alpha(\theta_{-\alpha}(\pi)) =\det \otimes \pi$.

To resolve this, we adopt the following (highly useful) convention: if $\theta_{-\alpha}(\pi)=\pi'$, we always write $\theta_{\alpha}(\pi')=\pi$, and interpret $\theta_\alpha$ to mean ``apply the theta lift and, if necessary, twist by det to ensure $\theta_\alpha(\theta_{-\alpha}(\pi)) = \pi$. Together with viewing the lifts of $\pi$ and $\det\otimes \pi$ as lifts to two separate towers, this convention enables one to have a uniform approach when dealing with theta lifts of classical groups.
\end{conv}

\subsection{Arthur packets}
\label{subs_Arthur}
We now recall the notion of a local A-packet.

We let $W_\F$ denote the Weil group of $\F$; then $L_\F = W_\F \times \SL_2(\mathbb{C})$ is the Weil--Deligne group. Let $G = G(W)$ be one of the groups introduced in \S\ref{subs_groups} and let $G^\circ$ denote the identity component of $G$:
\[
G^\circ = \begin{cases}
\Sp(W), \quad\text{for }\epsilon = 1;\\
\SO(W), \quad\text{for }\epsilon = -1.
\end{cases}
\]
Furthermore, when $\epsilon = -1$ (so that $G=\Ort(W)$), we let $\sigma_0$ denote the outer automorphism of $G^\circ$ given by conjugation by a fixed element $\varepsilon \in \Ort(W)\backslash\SO(W)$; to allow for uniform notation we set $\sigma_0 = \text{id}$ when $\varepsilon = 1$. We let $\Sigma_0$ denote the group generated by $\sigma_0$.

Attached to $G^\circ$ we have
\begin{align*}
    \widehat{G^\circ} &= \text{the complex dual group of } G^\circ,\\
    \prescript{L}{}{G^\circ} &=\text{the Langlands dual group of } G^\circ.
\end{align*}
An Arthur parameter for $G^\circ$ is a $\widehat{G^\circ}$-conjugacy class of admissible homomorphisms
\[
\psi: L_\F \times \SL_2(\mathbb{C}) \to \prescript{L}{}{G^\circ}
\]
such that the image of $W_\F$ is bounded. We abuse the notation and use $\psi$ to denote both the representative and its conjugacy class. We let $\Psi(G^\circ)$ denote the set of Arthur parameters. In \cite{Arthur_endoscopic}, Arthur attaches to each $\psi \in \Psi(G^\circ)$ a multiset $\Pi_\psi(G^\circ)$ of elements in $\Irr(G^\circ)/\Sigma_0$. We call $\Pi_\psi(G^\circ)$ the Arthur packet (or A-packet, for short) attached to $\psi$.

Although one does not have a systematic definition of the Langlands dual group for a disconnected reductive group, it was observed some time ago that an appropriate object in the case of $G=O(W)$ is $\widehat{G}=O(2n,\mathbb{C}).$
A detailed discussion on that topic can be found in \cite{Atobe_Gan_O(2n)}. In fact, Arthur \cite{Arthur_endoscopic} gives a parametrization of representations of  quasi-split $O(W)$ using Arthur packets, and from it, derives a parametrization of representations of $SO(W)$ (up to an action of $\Sigma_0$). For the purposes of the present paper we work with the full orthogonal group, so we view a local Arthur parameter as an admissible homomorphism
\[\psi: L_F \times \SL_2(\mathbb{C})\to \widehat{G},\]
(cf.~\S 6 of \cite{Atobe_Gan_O(2n)}). Analogously we define $\Psi(G)$ and we call $\Pi_{\psi}(G)$ the Arthur packet (or A-packet, for short) attached to $\psi.$ Although the Arthur classification has not been fully worked out for non-quasi-split orthogonal groups, we do consider them here. This is justified by the fact that the main results we use\,---\,M\oe glin's construction of packets and her results on the theta correspondence \cite[Theorem 5.1]{moeglinkudla}\,---\,take non-quasi-split groups into account. We refer the reader to the discussions in \cite[\S 1]{moeglinkudla} and \cite[Introduction]{moeglin2011multiplicite1}. To summarize, our results in cases which involve non-quasi-split orthogonal groups should be viewed as slightly speculative, although we expect them to take on precisely the form presented here.

We thus use $\Pi_\psi(G)$ to denote an Arthur packet attached to $\psi$ regardless of whether $G$ is symplectic or orthogonal. When $G$ is fixed and there is no ambiguity, we will often write simply $\Pi_\psi$ instead of $\Pi_\psi(G)$. Arthur packets for $G$ have been constructed by M\oe glin \cite{moeglin2006paquets}, \cite{moeglin2006certains}, \cite{moeglin2009comparaison}, \cite{moeglin2010holomorphie}, who also showed that the packets are in fact multiplicity-free \cite{moeglin2011multiplicite1}. We provide a rough outline of this construction.

Let $\psi \in \Psi(G)$. Composing $\psi$ with the standard representation of $\widehat{G}$, we can view it as a representation of $W_\F \times \SL_2(\mathbb{C}) \times \SL_2(\mathbb{C})$. Thus $\psi$ can be decomposed as
\[
\psi = \bigoplus_i m_i(\rterm{\rho_i}{a_i}{b_i}).
\]
Here $\rho_i$ is an equivalence class of irreducible unitary representations of $W_\F$; each such $\rho_i$ can be identified (via the local Langlands correspondence) with an irreducible unitary cuspidal representation of $\GL_{\dim \rho_i}(\F)$. Furthermore, $S_{n}$ denotes the irreducible algebraic $n$-dimensional representation of $\SL_2(\mathbb{C})$, and $m_i$ is the multiplicity of the corresponding irreducible summand. We call the triple $(\rho,a,b)$ a Jordan block of $\psi$ if the corresponding representation $\rterm{\rho}{a}{b}$ occurs in $\psi$. The multiset of all Jordan blocks will be denoted by $\Jord(\psi)$. Similarly, we often want to focus on a particular $\rho$, so we set $\Jord_\rho(\psi) = \{(a,b): (\rho,a,b) \in \Jord(\psi)\}$.

The first step of the construction of A-packets reduces the problem to the case of good parity. We say that a Jordan block $(\rho,a,b) \in \Jord(\psi)$ is of good parity if  $\rterm{\rho}{a}{b}$ factors through a group of the same type (symplectic/orthogonal) as $\widehat{G}$ (in particular, this implies that $\rho$ is self-dual). Otherwise, we say that $(\rho,a,b) \in \Jord(\psi)$ is of bad parity. Following M\oe glin, we denote the sub(multi)set of elements of good (resp.~bad) parity by $\Jord(\psi_\text{bp})$ (resp.~$\Jord(\psi_\text{mp})$). The parameter $\psi$ can then be written as
\[
\psi = \psi_\text{mp} + \psi_\text{bp} + \psi_\text{mp}^\vee.
\]
Here $\psi_\text{bp}$ corresponds to $\Jord(\psi_\text{bp})$ in an obvious manner, while the rest of the terms (which correspond to blocks in $\Jord(\psi_\text{mp})$) can be grouped as $\psi_\text{mp} +  \psi_\text{mp}^\vee$, with $\psi_\text{mp}^\vee$ denoting the dual of $\psi_\text{mp}$.
There is a bijection between $\Pi_\psi$ and $\Pi_{\psi_\text{bp}}$: $\psi_\text{mp}$ determines an irreducible representation $\pi_\text{mp}$ of a general linear group; for any representation $\pi_0 \in \Pi_{\psi_\text{bp}}$, the representation $\pi = \pi_{\text{mp}}\rtimes \pi_0$ is irreducible, and is contained in $\Pi_{\psi}$. Conversely, any $\pi \in \Pi_{\psi}$ is of the form $\pi_{\text{mp}}\rtimes \pi_0$ for some $\pi_0 \in \Pi_{\psi_\text{bp}}$. Applying Kudla's filtration (Lemma \ref{lem_Kudla}) to the map $\pi_{\text{mp}}\rtimes \pi_0 \twoheadrightarrow \pi$, one shows that the Adams conjecture (Conjecture \ref{slutnja}) holds for $\pi$ if and only if it holds for $\pi_0$. This allows us to focus on parameters of good parity: from now on, we assume $\psi = \psi_\text{bp}$.

Now let $\psi = \psi_\text{bp}$ be a parameter of good parity. M\oe glin's construction of $\Pi_\psi$ involves the following data:
\begin{itemize}[label=\textemdash]
    \item an admissible order on $\Jord(\psi)$;
    \item a function $t: \Jord(\psi) \to \mathbb{Z}_{\geq 0}$ such that $t(\rho,a,b) \in [0,\min(a,b)/2]$ for every $(\rho,a,b) \in \Jord(\psi)$;
    \item a function $\eta: \Jord(\psi) \to \{\pm 1\}$.
\end{itemize}
Loosely speaking, the function $\eta$ can be viewed as a character of the so-called component group $\mathcal{S}_\psi$ attached to $\psi$. We allow $t$ and $\eta$ to attain different values on different copies of the same Jordan block $(\rho,a,b)$. 
The functions $t$ and $\eta$ satisfy additional requirements:  Set $\epsilon_{t,\eta}(\rho,a,b)=\eta(\rho,a,b)^{\min(a,b)}(-1)^{\lfloor\min(a,b)/2\rfloor+t(\rho,a,b)}$. We then require
\begin{equation}
\label{eq_det_condition}
\prod_{(\rho,a,b)\in\Jord(\psi)} \epsilon_{t,\eta}(\rho,a,b) = \epsilon_G.
\end{equation}
Here $\epsilon_G=1$ if $\psi$ is of odd dimension; when $\psi$ is of even dimension, $\epsilon_G$ is the Hasse invariant of the orthogonal group $G$.
\begin{rem}
\label{rem_parni}
When $t(\rho,a,b)=\min(a,b)/2$, the sign $\eta(\rho,a,b)$ does not carry additional information: changing it does not change the parametrized representation. See \cite[pp.\ 5--6]{xu2021combinatorial}, where this ambiguity is resolved using an equivalence relation, and \cite{moeglinkudla} where $\eta(\rho,a,b)$ is set to $1$ whenever $t(\rho,a,b)=\min(a,b)/2$.
\end{rem}

To explain the admissibility condition for orders on $\Jord(\psi)$, we introduce the following notation, used throughout the paper. Given $(\rho,a,b)\in\Jord(\psi)$, we let
\[
A = \frac{a+b}{2}-1,\quad B = \frac{|a-b|}{2},\quad \zeta = \text{sgn}(a-b)\in\{\pm 1\} 
\]
and we write $(\rho,A,B,\zeta)$ instead of $(\rho,a,b)$. When $a-b=0$, we set $\zeta = 1$.
A total order $>$ on $\Jord(\psi)$ is said to be admissible if for all $(\rho,A,B,\zeta)$ and  $(\rho,A',B',\zeta') \in \Jord(\psi)$ we have
\[
A > A' \text{ and } B' > B \text{ and } \zeta= \zeta' \quad \text{implies}\quad (\rho,A,B,\zeta) > (\rho,A',B',\zeta').
\]
Recall that a parameter $\psi$ is said to have discrete diagonal restriction (DDR, for short) if the segments $[B, A]_\rho$ and $[{B'},{A'}]_{\rho}$ are disjoint for any two blocks $(\rho,A,B,\zeta), (\rho,A',B',\zeta') \in \Jord(\psi)$. Each choice of $(t,\eta)$ (subject to conditions described above) then corresponds to an irreducible representation in $\Pi_\psi$. DDR parameters play an important role in this paper. Although they are much simpler than general parameters of good parity, constructing the corresponding packets is non-trivial. We do not describe this construction here; instead, we refer the reader to Proposition 3.3 of \cite{moeglinkudla}. We point out that there is more than one way to parametrize the representations inside $\Pi_\psi$; see Remark \ref{rem_parametrization} below. When talking about DDR parameters, we often use the notation $(\rho,A,B,\zeta)$ interchangeably with the segment notation $[B,A]_\rho$.

One may reduce the construction of $\Pi_\psi$ for a general (good parity) parameter $\psi$ to the DDR case, as follows:
Let $\Jord(\psi) = \{(\rho_i,A_i,B_i,\zeta_i): i=1,\dotsc,r\};$ here we assume that the Jordan blocks are enumerated with respect to some admissible order. We consider a new parameter $\psi_{\gg}$ (of a larger group) with $\Jord(\psi_{\gg}) = \{(\rho_i,A_i+T_i,B_i+T_i,\zeta_i): i=1,\dotsc,r\}$, where $T_i$, $i=1,\dotsc,r$ are positive integers such that $\psi_{\gg}$ is DDR. In this situation, we say that $\psi_\gg$ dominates $\psi$. For each choice of $(t,\eta)$, we get a representation $\pi_{\gg,t,\eta}$ in $\Pi_{\psi_\gg}$. Let
\[
\pi_{t,\eta} = \Jac^{\rho_r}_{\zeta_r[A_r+T_r, B_r+T_r]\to \zeta_r[A_r,B_r]} \circ \dotsb \circ \Jac^{\rho_1}_{\zeta_1[A_1+T_1, B_1+T_1]\to \zeta_1[A_1, B_1]} (\pi_{\gg,t,\eta}).
\]
Then $\pi_{t,\eta}$ is either $0$, or is an irreducible element of the packet $\Pi_\psi$; the packet $\Pi_\psi$ consists of all non-zero representations obtained in this way. 

Let us describe an intermediate stage in this process; this will be useful later. Fix an equivalence class of $\rho$ of irreducible unitary representations of $W_F$. We may then choose an admissible order on $\Jord(\psi)$ such that $(\rho',A',B',\zeta') < (\rho,A,B, \zeta)$ whenever $\rho' \neq \rho)$. Thus, if $\Jord(\psi) = \{(\rho_i,A_i,B_i,\zeta_i): i=1,\dotsc,r\}$ (as above), we assume that there exists an index $j$ such that $\rho_i \neq \rho$ for $i\leq j$ and $\rho_i = \rho$ for $i> j$. Consider the parameter $\psi_{\gg}^\rho$ which is obtained from $\psi$ by replacing $(\rho,A_i,B_i,\zeta_i)$ with $(\rho,A_i+T_i,B_i+T_i,\zeta_i)$ for each $i > j$. In particular, the segments $[B, A]_\rho$ and $[{B'},{A'}]_{\rho}$ are disjoint for any two blocks $(A,B,\zeta), (A',B',\zeta') \in \Jord_\rho(\psi)$. Such a parameter is said to be $\rho$-DDR, and we refer to $\psi_{\gg}^\rho$ as the $\rho$-DDR parameter which dominates $\psi$. Again, for any choice of $(t,\eta)$, we get a representation $\pi_{\gg,t,\eta}^\rho$ in $\Pi_{\psi_\gg^\rho}$. This is an intermediate step between $\pi_{t,\eta}$ and $\pi_{\gg,t,\eta}$:
\[
\pi_{\gg,t,\eta}^\rho = \Jac^{\rho_j}_{\zeta_j[A_j+T_j, B_j+T_j]\to \zeta_j[A_j,B_j]} \circ \dotsb \circ \Jac^{\rho_1}_{\zeta_1[A_1+T_1, B_1+T_1]\to \zeta_1[A_1, B_1]} (\pi_{\gg,t,\eta}).
\]
Note that we are only applying the first $j$ Jacquet functors (instead of all $r$, as in the construction of $\pi_{t,\eta}$). It follows that $\pi_{t,\eta}\neq 0$ implies $\pi_{\gg,t,\eta}^\rho  \neq 0$. We will return to this construction in \S \ref{subs_Arthurtype}. 

\begin{rem}
\label{rem_T}
We use the above notation throughout the paper. Any time we use $\psi_{\gg}$ it is implied that the numbers $T_i$ are defined; we will freely refer to them whenever we deal with $\psi_{\gg}$.
Moreover, this notation will be used functorially with respect to decorations: for instance, the DDR parameter which dominates $\psi'$ will be denoted by $\psi'_{\gg}$, etc.
\end{rem}

Although this reduction to the DDR case allows a simple description, it is far from being easy to work with. For instance, it is not a priori clear whether the representation $\pi_{t,\eta}$ is non-zero for a given $(\eta,t)$. This question is of course crucial for our considerations in this paper, and it turns out to be highly non-trivial. In \cite{xu2021combinatorial}, Xu gives an algorithm to determine whether $\pi_{t,\eta}$ is non-zero. The recent work of Atobe \cite{atobe2022construction} also contains a criterion that answers this question.

\begin{rem}
\label{rem_parametrization}
As noted above, there is more than one way to parametrize representations inside $\Pi_\psi$. In her work, M\oe glin uses a parametrization which differs from the one originally used by Arthur. The comparison between the two has been conducted by M\oe glin herself; more recently, this has been addressed by Xu \cite{xu2017moeglin}. In this paper, we use the parametrization introduced by M\oe glin. However, there is another choice to be made, related (roughly) to the parametrization of supercuspidal representations. Here we use the parametrization of supercuspidals provided by Arthur. This combination of M\oe glin's parametrization with Arthur's “initial conditions” is precisely the one studied by Xu in \cite{xu2017moeglin}.
\end{rem}

In the rest of the paper, we often have to specify the data $(\psi,t,\eta)$. We will do so by listing all the irreducible summands of $\psi$ with the corresponding values of $\eta$ and $t$ written on top of them:
\[
\psi = \bigoplus_{i=1}^r\overset{t_i,\eta_i}{\rho_i\otimes S_{a_i}\otimes S_{b_i}} \quad \text{or} \quad \psi = \bigoplus_{i=1}^r\overset{t=t_i,\ \eta_i}{\rho_i\otimes S_{a_i}\otimes S_{b_i}}.
\]
We usually omit those $\eta_i$'s or $t_i$'s that are not relevant to the discussion.


\section{The Adams conjecture}
\label{sec_conj}
In his 1989 paper \cite{Adams}, Adams proposed the following:
\begin{conj}
\label{slutnja}
Suppose $\pi$ is an irreducible representation of $G$ contained in the Arthur packet attached to the parameter $\psi$. Suppose that $\alpha > 0$ is an odd integer such that $\Theta_{-\alpha}(\pi)\neq 0$. Then, $\theta_{-\alpha}(\pi)$ is contained in the A-packet parametrized by
\[
\psi_\alpha = (\chi_W\chi_V^{-1}\otimes\psi)\ \oplus\ \chi_W\otimes S_1 \otimes S_\alpha.
\]
\end{conj}
When $\psi$ and $\alpha$ are fixed, we often refer to $\Pi_{\psi_\alpha}$ as the A-A (Adams--Arthur) packet.
Adams himself verified this conjecture for all the examples of theta correspondence available at the time. However, subsequent work on the local theta correspondence led to new findings related to the above conjecture. In particular, in her paper \cite{moeglinkudla}, M\oe glin revisited the work of Adams. She showed that
\begin{itemize}[label=\textemdash]
    \item Conjecture \ref{slutnja} is true for large $\alpha$ (see Proposition \ref{prop_visoko} for a more precise statement);
    \item Conjecture \ref{slutnja} fails in many examples.
\end{itemize}
In view of these findings, our goal in this paper is to investigate the extent to which Conjecture \ref{slutnja} holds.
To be precise, we are interested in the following questions suggested by M\oe glin (see \cite[Section 6.3]{moeglinkudla}):
\begin{itemize}
    \item[1)] Given a representation $\pi$ in $\Pi_\psi$, consider the set
    \[\mathcal A(\pi,\psi) = \{\alpha > 0,\  \alpha\equiv 1 (\text{mod } 2): \theta_{-\alpha}(\pi) \in \Pi_{\psi_\alpha}\}.\]
    Is it true that $\alpha\in \mathcal A(\pi,\psi)$ implies $\alpha+2\in \mathcal A(\pi,\psi)$?
    \item[2)] If so, can we find $a(\pi,\psi) := \min \mathcal A(\pi,\psi)$ explicitly?
\end{itemize}

\bigskip

This paper provides complete answers to both of these questions. We briefly explain our approach and results here. We start with the following result of M\oe glin:


\begin{prop}[\cite{moeglinkudla}, Theorem 5.1]
\label{prop_visoko}
Let $\pi$ be parametrized by $(\psi,t,\eta)$, with
\[
\psi = \bigoplus_{i=1}^r{\rho_i\otimes S_{a_i}\otimes S_{b_i}}.
\]
Let $\alpha\gg0$. Then $\theta_{-\alpha}(\pi)$ is parametrized by $(\psi_\alpha,t',\allowbreak \eta')$, where $t'(\chi_W\chi_V^{-1}\rho_i,a_i, b_i) = t(\rho_i , a_i, b_i)$ for $i=1,\dotsc,r$, and $\eta'$ is defined as follows.

For $i = 1, \dotsc,r$
\[
\eta'(\chi_W\chi_V^{-1}\rho_i ,a_i, b_i) = 
\begin{cases}
-\eta(\rho_i,a_i,b_i), &\quad \text{if } \rho_i \cong \chi_V\\  
\eta(\rho_i,a_i,b_i),  &\quad \text{if }  \rho_i \ncong \chi_V.   
\end{cases}
\]
If $\pi$ is a representation of the symplectic group, then $\eta'(\chi_W \otimes S_1 \otimes S_\alpha)$ is determined by the target tower of orthogonal groups. If $\pi$ is a representation of the orthogonal group, then $\eta'(\chi_W \otimes S_1 \otimes S_\alpha)$ is uniquely determined by Equation \eqref{eq_det_condition}.
\end{prop}
\noindent Here we assume the order on $\Jord(\psi_\alpha)$ is inherited from the order on $\Jord(\psi)$:
\[
(\chi_W\chi_V^{-1}\rho_i,a_i,b_i) > (\chi_W\chi_V^{-1}\rho_j,a_j,b_j) \iff (\rho_i,a_i,b_i) > (\rho_j,a_j,b_j),
\]
with the added block being the largest:
\[
(\chi_W,1,\alpha) > (\chi_W\chi_V^{-1}\rho_i,a_i,b_i) \quad \text{for all } i.
\]

\begin{rem}
\label{rem_ort_lifts}
If $\pi$ is a representation of the orthogonal group, recall that twisting by $\det$ allows us to still consider lifts to “two target towers”, as explained in \S \ref{subs_theta}. If $\pi$ is parametrized by $(\psi,t,\eta)$, then $\pi \otimes \det$ is parametrized by $(\psi,t,\tilde{\eta})$ where
\[
\tilde{\eta}(\rho_i, a_i, b_i) =
\begin{cases}
-\eta(\rho_i,a_i,b_i), &\quad \text{if } \dim(\rho_i \otimes S_{|a_i-b_i|+1}) \text{ is odd};\\  
\eta(\rho_i,a_i,b_i), &\quad \text{if } \dim(\rho_i \otimes S_{|a_i-b_i|+1}) \text{ is even}. 
\end{cases}
\]
Note that the transformation $\eta \mapsto \tilde{\eta}$ does not change the right-hand side of \eqref{eq_det_condition}. Therefore, if we want to consider the lift of $\pi$ to “the other tower”, we first replace $\eta$ by $\tilde{\eta}$, and then perform the transformation described in Proposition \ref{prop_visoko}. For tempered representations, this follows from e.g.\ \cite[Desideratum 3.9 (9)]{Atobe_Gan_O(2n)}. The claim then follows for all representations in Arthur packets using M\oe glin's construction (from elementary to non-DDR representations of good parity) and the fact that 
\[(\pi_1\times \pi_2\times \cdots \times \pi_k\rtimes \tau)\otimes \det\cong \pi_1\times \pi_2\times \cdots \times \pi_k\rtimes (\tau\otimes \det),\]
where $\pi_i,\;i=1,2,\ldots,k$ and $\tau$ are irreducible representations.
\end{rem}
\begin{cor}
\label{cor_removal}
Let $\pi'$ be parametrized by $(\psi',t',\eta')$ with
\[
\psi' = \bigoplus_{i=1}^r{\chi_W\chi_V^{-1}\rho_i\otimes S_{a_i}\otimes S_{b_i}}\ \oplus\ \chi_W \otimes S_1 \otimes S_\alpha.
\]
Assume that $\alpha \gg 0$. Then there exists a representation $\pi$ parametrized by $(\psi,\eta,t)$ with
\[
\psi = \bigoplus_{i=1}^r{\rho_i\otimes S_{a_i}\otimes S_{b_i}}
\]
such that, for a suitable choice of tower, $\theta_{-\alpha}(\pi) = \pi'$ or 
$\theta_{-\alpha}(\pi) = \pi'\otimes \det$.

Here $\eta$ and $t$ are related to $\eta', t'$ as in Proposition \ref{prop_visoko}, except possibly when $\pi'$ is a representation of an orthogonal group; in that case, we may have to replace $\eta'$ by $\tilde{\eta}'$ (as in Remark \ref{rem_ort_lifts}) before applying the transformation $\eta' \to \eta$.
\end{cor}
\begin{proof}
This is simply Proposition 2.2 in reverse. The only thing one has to ensure is that $\eta$ satisfies the condition in Equation \eqref{eq_det_condition}. If $\pi'$ is the representation of a symplectic group, there are no complications with $\eta$; the right-hand side of Eq.\ \eqref{eq_det_condition} determines the tower of orthogonal groups on which we can find $\pi$. However, if $\pi'$ is a representation of an orthogonal group, one has to keep in mind Remark 2.3: if $\eta$ does not fulfill the condition of Eq.\ \eqref{eq_det_condition}, that just means we are lifting $\pi'$ to the wrong tower. To be precise, in that case, we have
\[\prod_{(\rho,a,b)\in\Jord(\psi)} \epsilon_{t',-\eta'}(\rho,a,b) = -1.\]
On the other hand, $\pi'\otimes \det$ is parametrized by $(\psi',t',\tilde{\eta'})$ and then we examine
\[\prod_{(\rho,a,b)\in\Jord(\psi)} \epsilon_{t',-\tilde{\eta'}}(\rho,a,b).\]
The factors that  have changed the sign $\epsilon_{t',-\tilde{\eta'}}(\rho,a,b)=-\epsilon_{t',-\eta'}(\rho,a,b)$ are, by the definition of $\epsilon_{t',-\tilde{\eta'}}(\rho,a,b)$, precisely the ones for which the product $(\dim \rho)ab$ is odd. Since $\psi$ parameterizes a representation of a symplectic group and thus factors through the odd orthogonal group, there is an odd number of such factors. All in all, this means that the signs are changed so that the product in Eq.\ \eqref{eq_det_condition} is equal to 1; $(\psi,t',-\tilde{\eta'})$ parameterizes a representation $\pi$ of a symplectic group, and  $\theta_{-\alpha}(\pi)=\pi'\otimes \det.$
\end{proof}

 Proposition \ref{prop_visoko} gives us the necessary starting point and allows us to identify a candidate parameter for each lift. First, for each odd $\alpha \gg 0$, we define a candidate representation $\pi_\alpha$ in $\Pi_{\psi_\alpha}$ using the formula from the above proposition. Now, starting from $\pi_\alpha$ for any $\alpha\gg0$, we define representations $\pi_{\alpha-2}$, $\pi_{\alpha-4}$, \dots, $\pi_1$ inductively, using the following recipe: 
\begin{reci}
\label{recipe}
Let $\alpha > 1$ be odd. Suppose that $\pi_\alpha$ is parametrized by $(\psi_\alpha, \eta,t)$. If $\pi_\alpha =0$, we set $\pi_{\beta}=0$ for all $\beta < \alpha$. Otherwise, we define $\pi_{\alpha-2}$ to be the representation (possibly zero) parametrized by $\psi_{\alpha-2}$ and the functions $\eta, t$ modified as follows:

To go from $\psi_\alpha$ to $\psi_{\alpha-2}$, we simply shift the added block from $\chi_W\otimes S_1 \otimes S_\alpha$ to $\chi_W\otimes S_1 \otimes S_{\alpha-2}$. By default, the order on $\Jord(\psi_\alpha)$ is the same as the order on $\Jord(\psi_{\alpha-2})$, and the functions $\eta$ and $t$ do not change. The exception is the case when $\Jord_{\chi_V}(\psi)$ contains a block $(A,B,\zeta)$ with $B=\frac{\alpha-1}{2}$. In this case, going from $\psi_\alpha$ to $\psi_{\alpha-2}$, we change the order so that the added block becomes smaller than $(A,B,\zeta)$, and use the formulas described in Lemma \ref{lem_xu} to update $\eta$ and $t$. If there are multiple blocks with $B=\frac{\alpha-1}{2}$, we change the order and $\eta,t$ for each such block.
\end{reci}
\begin{exmp}
\label{ex_recipe}
We let $\pi$ be the Steinberg representation of $\Sp_6$. Then $\pi$ is given by the Arthur parameter
\[
\psi = \overset{+}{1\otimes S_7 \otimes S_1}.
\]
Recall that the sign above a given summand indicates the corresponding value of $\eta$. For $\alpha \geq 7$ the Recipe produces
\[
\psi_\alpha = \overset{-}{1 \otimes S_7 \otimes S_1} \, \oplus \, \overset{\epsilon}{1 \otimes S_1 \otimes S_\alpha}.
\]
The choice of $\epsilon \in \{\pm 1\}$ corresponds to the choice of target tower of orthogonal groups. For $\alpha \in \{1,3,5\}$, the two summands need to switch places, so we get (using Lemma \ref{lem_xu})
\[
\psi_\alpha =  \overset{-\epsilon}{1 \otimes S_1 \otimes S_\alpha} \, \oplus \, \overset{+}{1 \otimes S_7 \otimes S_1}.
\]
\end{exmp}
With the Recipe in mind, we define
\[
 d(\pi,\psi) = \min\{\alpha > 0,\  \alpha\equiv 1 (\text{mod } 2): \pi_\alpha \neq 0\}. 
\]
The Recipe thus defines a candidate representation $\pi_\alpha$ in $\Pi_{\psi_\alpha}$ for each odd $\alpha \geq d(\pi,\psi)$. We know that this representation has the property $\pi_\alpha = \theta_{-\alpha}(\pi)\neq 0$ for $\alpha \gg 0$. Our first result is
\begin{thm}[Theorem A]
\label{theoremA} Let $\alpha > 1$ be odd.
Assume $\pi_\alpha = \theta_{-\alpha}(\pi)$ and $\pi_{\alpha-2} \neq 0$. Then $\theta_{-(\alpha-2)}(\pi)=\pi_{\alpha-2}$; in particular, $\theta_{-(\alpha-2)}(\pi)\neq 0$ is in the A-A packet.
\end{thm}
Thus Theorem \ref{theoremA} shows that the Adams conjecture holds for all $\alpha\geq d(\pi,\psi)$. As explained in \S \ref{subs_theta}, it is fruitful to look at two towers of lifts simultaneously. Therefore, we consider $d^\text{up}(\pi,\psi)$ and $d^\text{down}(\pi,\psi)$, corresponding to the going-up and the going-down tower, respectively. We prove
\begin{thm}[Theorem B]
\label{theoremB}
     On the going-up tower, $d^\text{up}(\pi,\psi)$ corresponds to the first occurrence of $\pi$:
    \[d^\text{up}(\pi,\psi) = \min\{\alpha>0: \theta_{-\alpha}^\text{up}(\pi)\neq 0\}.\]
     Consequently, the Adams conjecture is true for all non-zero lifts. Moreover, $d^\text{down}(\pi,\psi) \leq d^\text{up}(\pi,\psi)$, and equality holds if and only if $d^\text{up}(\pi,\psi) =1$.
\end{thm}
\noindent Notice that Theorem \ref{theoremB} explains why we do not have to worry about the situation $m^\text{up}(\pi)=m^\text{down}(\pi)=n+\epsilon+1$ addressed in Remark \ref{rem_updown}: In this case, $d^\text{up}(\pi,\psi) =1$, so the inequality in Theorem \ref{theoremB} implies $d^\text{down}(\pi,\psi) =1$ as well. Thus $d(\pi,\psi)=1$ on both towers, so Conjecture \ref{slutnja} holds for all $\alpha>0$.

\noindent Finally, we answer question 1) posed above:

\begin{thm}[Theorem C]
\label{theoremC}
Conjecture \ref{slutnja} is always false for $\alpha < d(\pi,\psi)$.
\end{thm}
\noindent Thus Theorems \ref{theoremA}, \ref{theoremB}, and \ref{theoremC} answer both questions 1) and 2) posed above: we have $\mathcal A(\pi,\psi) =  \{\alpha\geq d(\pi,\psi),\  \alpha\equiv 1 (\text{mod } 2)\}$; in particular, $\alpha \in \mathcal A(\pi,\psi)$ implies $\alpha+2 \in \mathcal A(\pi,\psi)$. Furthermore, $a(\pi,\psi)=d(\pi,\psi)$. 

On the going-up tower, finding $d(\pi,\psi)$ is equivalent to finding the first occurrence of $\pi$. The question of non-vanishing has been completely resolved in \cite{nas_clanak}, where the answer is stated in terms of the L-parameter of the representation. In this paper, $\pi$ is given as a member of an Arthur packet; the Recipe enables us to identify the going-up tower and determine the first occurrence index in terms of the Arthur packet data. 

On the going-down tower, finding $d(\pi,\psi)$ boils down to establishing the (non-)vanishing of the candidate representation $\pi_\alpha$. This criterion is satisfactory thanks to the recent work of Xu \cite{xu2021combinatorial} and Atobe \cite{atobe2022construction}.

\begin{rem} A word on the signs. The parametrization of representations inside an A-packet\,---\,and thus, the statement of Proposition 2.2\,---\,depends on the choice of a Whittaker datum. For tempered A-parameters we fix the Whittaker datum as in \cite{Atobe_Gan}, Remark B.2. In particular, this gives us functions $t$ and $\eta$ for cuspidal representations; from that, M\oe glin's construction gives us all the remaining of the representations. The fact that our transformation $\eta \mapsto \eta'$ in Proposition \ref{prop_visoko} differs from M\oe glin's in \cite[Theorem 5.1]{moeglinkudla} is precisely due to the difference in parametrization of supercuspidals, as explained in Remark \ref{rem_parametrization}.

Note that Theorem \ref{theoremA} refines the statement of Conjecture \ref{slutnja}, regardless of which pa\-ra\-me\-tri\-zation one uses: we do not only get the packet $\psi_\alpha$, but also the functions $t'$, and $\eta'$ that correspond to $\theta_{-\alpha}(\pi)$. Recall that the function $\eta$ (resp.\ $\eta'$) represents a character of the component group $\mathcal{S}_\psi$ (resp.\ $\mathcal{S}_{\psi_\alpha}$). 
If one uses the so-called Whittaker normalization (as explained in Sections 1 and 4 of \cite{xu2017moeglin}), then $\eta$ is obtained by pulling back $\eta'$ along the natural map $\mathcal{S}_\psi  \to \mathcal{S}_{\psi_\alpha}$. This is called the \emph{refined Adams conjecture}, and can be verified using the transition formulas in Theorem 1.1 of \cite{xu2021combinatorial}.
\end{rem}

\section{Examples}
\label{sec_example}
Before proving the main theorems, we demonstrate the results on a few examples. Note that the Arthur classification is still conjectural in the non-quasi-split case (so in particular, for orthogonal groups of trivial discriminant and Hasse invariant equal to $-1$); however, to simplify our examples, we momentarily ignore this fact and work with $\chi_V=\chi_W=1$.
\begin{exmp}[Example \ref{ex_recipe} continued]
Recall that the Steinberg representation $\pi$ of $\Sp_6$ is given by
\[
\psi = \overset{+}{1\otimes S_7 \otimes S_1}.
\]
According to Theorem \ref{theoremA}, the lift $\theta_{-\alpha}(\pi)$ is given by
\[
\psi_\alpha = \overset{-}{1 \otimes S_7 \otimes S_1} \, \oplus \, \overset{\epsilon}{1 \otimes S_1 \otimes S_\alpha}
\]
for $\alpha \geq 7$. The choice of $\epsilon \in \{\pm 1\}$ corresponds to the choice of target tower of orthogonal groups. In this case, $d(\pi,\psi)=1$ on both towers; in particular, $\Theta_{-1}(\pi)\neq 0$ on both towers. The lifts for $\alpha \in \{1,3,5\}$ are given by
\[
\psi_\alpha =  \overset{-\epsilon}{1 \otimes S_1 \otimes S_\alpha} \, \oplus \, \overset{+}{1 \otimes S_7 \otimes S_1}.
\]
\end{exmp}
\begin{exmp}
\label{ex_two}
We let $\pi$ be the trivial representation of $\Sp_6$. Then $\pi$ is given by the Arthur parameter
\[
\psi = \overset{+}{1 \otimes S_1 \otimes S_7}.
\]
For $\alpha > 7$ the lift $\theta_{-\alpha}(\pi)$ is given by
\[
\psi_\alpha = \overset{-}{1 \otimes S_1 \otimes S_7} \, \oplus \, \overset{\epsilon}{1 \otimes S_1 \otimes S_\alpha}.
\]
Again, the choice of $\epsilon \in \{\pm 1\}$ corresponds to the choice of target tower of orthogonal groups.

If $\epsilon = 1$, then $\pi_7 = 0$. This shows that $d(\pi,\psi)=9$ on the corresponding tower. In other words, this is the going-up tower, and $\theta_{-9}(\pi)$ is the first occurrence of $\pi$.

If $\epsilon = -1$, then $d(\pi,\psi)=1$. The lifts for $\alpha \in \{1,3,5,7\}$ on this tower are given by
\[
\psi_\alpha =  \overset{-}{1 \otimes S_1 \otimes S_\alpha} \, \oplus \, \overset{-}{1 \otimes S_1 \otimes S_7}.
\]
\end{exmp}
\begin{exmp}
Let $\pi$ be the representation of $\Sp_{10}$ given by the Arthur parameter
\[
\psi = \overset{-}{1 \otimes S_1 \otimes S_1} \, \oplus \, \overset{+}{1 \otimes S_3 \otimes S_1} \, \oplus \, \overset{-}{1 \otimes S_1 \otimes S_7}.
\]
Then $\pi$ is the Langlands quotient of $\nu^3\rtimes \sigma$, where $\sigma$ is the supercuspidal representation of $\Sp_8$ parametrized by
\[
\overset{-}{1 \otimes S_1 \otimes S_1} \, \oplus \, \overset{+}{1 \otimes S_3 \otimes S_1} \, \oplus \, \overset{-}{1 \otimes S_5 \otimes S_1}.
\]
For $\alpha > 7$ the lift $\theta_{-\alpha}(\pi)$ is given by
\[
\psi_\alpha = \overset{+}{1 \otimes S_1 \otimes S_1} \, \oplus \, \overset{-}{1 \otimes S_3 \otimes S_1} \, \oplus \, \overset{+}{1 \otimes S_1 \otimes S_7} \, \oplus \, \overset{\epsilon}{1 \otimes S_1 \otimes S_\alpha}.
\]
If $\epsilon = -1$, then $\pi_7 = 0$. Indeed, $\pi_9=\theta_{-9}(\pi)$ is the Langlands quotient of $\delta(3,4) \rtimes \theta_{-7}(\sigma)$; note that $\theta_{-7}(\sigma)$ is supercuspidal because it is the first lift of a supercuspidal representation $\sigma$. Thus $\pi_7 = \Jac_{-4}(\pi_9) = 0$, as one verifies with a simple Jacquet module computation (see \S \ref{subs_compJac}). In other words, $\epsilon = -1 $ corresponds to the going-up tower, and $\theta_{-9}(\pi)$ is the first occurrence on this tower.

If $\epsilon = 1$, then $\pi_7$ and $\pi_5$ are non-zero, given by
\[
\psi_\alpha = \overset{+}{1 \otimes S_1 \otimes S_1} \, \oplus \, \overset{-}{1 \otimes S_3 \otimes S_1} \, \oplus \, \overset{+}{1 \otimes S_1 \otimes S_\alpha} \, \oplus \, \overset{+}{1 \otimes S_1 \otimes S_7}
\]
for $\alpha \in \{5,7\}$. However, we have $\pi_3 = 0$. Indeed, $\pi_5 = \theta_{-5}(\pi)$ on the going-down tower. This means that $\pi_5$ is the Langlands quotient of 
\[
\nu^3 \times \nu^2 \times \nu^1 \rtimes \theta_{-1}(\sigma),
\]
where $\theta_{-1}(\sigma)$ is tempered. It follows (using the results of \S \ref{subs_compJac}) that $\pi_3 = \Jac_{-2}(\pi_5) = 0$, even though $\Theta_{-3}(\pi) \neq 0$.
\end{exmp}

\begin{exmp}
\label{ex_four}
Let $V_{10}$ be the $10$-dimensional quadratic space of discriminant $1$ and Hasse invariant $-1$. Let $\pi$ be the representation of $\Ort(V_{10})$ parametrized by
\[
\psi =  \overset{-}{1 \otimes S_1 \otimes S_3} \, \oplus \, \overset{+}{1 \otimes S_1 \otimes S_7}
\]
Then $\pi$ is the Langlands quotient of $\nu^3 \times  \delta(1,2) \rtimes \sigma$, where $\sigma$ is the supercuspidal representation of $\Ort(V_4)$ parametrized by
\[
\overset{-}{1 \otimes S_1 \otimes S_1} \, \oplus \, \overset{+}{1 \otimes S_3 \otimes S_1}.
\]
Note that $\pi \otimes \det$ is parametrized by
\[
\psi =  \overset{+}{1 \otimes S_1 \otimes S_3} \, \oplus \, \overset{-}{1 \otimes S_1 \otimes S_7}.
\]
We consider the lifts of both $\pi$ and $\pi \otimes \det$ to the symplectic tower; we abuse the terminology in the standard way and refer to $\theta(\pi)$ and $\theta(\pi\otimes \det)$ as the lifts of $\pi$ to two different towers.

For $\alpha > 7$ the lifts $\theta_{-\alpha}$ are given by
\[
\psi_\alpha = \overset{\eta}{1 \otimes S_1 \otimes S_3} \, \oplus \, \overset{-\eta}{1 \otimes S_1 \otimes S_7} \, \oplus \, \overset{-}{1 \otimes S_1 \otimes S_\alpha}
\]
where $\eta = 1$ (resp.~$\eta = -1$) corresponds to the lift of $\pi$ (resp.~$\pi \otimes \det$).

\bigskip

If $\eta = -1$, then $\pi_7 = 0$. Indeed, $\pi_9=\theta_{-9}(\pi\otimes \det)$ is the Langlands quotient of $\delta(3,4) \times \delta(1,3) \rtimes \theta_{-5}(\sigma \otimes \det)$; note that $\theta_{-5}(\sigma\otimes \det)$ is supercuspidal because it is the first lift of a supercuspidal representation $\sigma \otimes \det$. Thus $\pi_7 = \Jac_{-4}(\pi_9) = 0$. In other words, $\pi \otimes \det$ is the going-up tower, and $\theta_{-9}(\pi \otimes \det)$ is its first occurrence.

If $\eta = 1$, then $\pi_7$ and $\pi_5$ are non-zero, given by
\[
\psi_\alpha = \overset{+}{1 \otimes S_1 \otimes S_3} \, \oplus \, \overset{-}{1 \otimes S_1 \otimes S_\alpha} \, \oplus \, \overset{-}{1 \otimes S_1 \otimes S_7}
\]
for $\alpha \in \{5,7\}$. However, we have $\pi_3 = 0$. Indeed, $\pi_5 = \theta_{-5}(\pi)$ is the Langlands quotient of 
\[
\nu^3 \times \nu^2 \times \delta(1,2) \times \nu^1 \rtimes \theta_{-1}(\sigma),
\]
where $\theta_{-1}(\sigma)$ is tempered. Thus $\pi_3 = \Jac_{-2}(\pi_5) = 0$, even though $\Theta_{-3}(\pi) \neq 0$.

\end{exmp}

\section{Technical results}
\label{sec_tech}
This section contains certain auxiliary results we use throughout the paper. These are mostly (variations of) known results, but we list them here in order to keep the paper self-contained. We advise the reader to merely skim this section on initial reading. 

\subsection{M\oe glin's construction}
\label{subs_Arthurtype}
Recall that the representations parametrized by a general (non-DDR) good parity parameter are constructed by passing to a DDR parameter and then shifting back the blocks by taking Jacquet modules. We give a more detailed description of this process. For us, it will suffice to focus on a single cuspidal representation $\rho$, which we now fix. Let $\psi$ be a (non-DDR) good parity parameter with $\Jord_\rho(\psi) =  \{(A_i,B_i,\zeta_i): i=1,\dotsc r\}$. Now let $\psi_\gg^\rho$ be the $\rho$-DDR parameter which dominates $\psi$ (see \S \ref{subs_Arthur}). Let $\pi$ be an irreducible representation in $\Pi_\psi$, and let $\pi_{\gg}^\rho$ denote the corresponding representation in $\Pi_{\psi_\gg^\rho}$. We wish to describe intermediate steps in going from $\pi_{\gg}^\rho$ to $\pi$:

 For each $i$, denote by $\pi_{(i)}$ the $i$-th intermediate step in going from $\pi_{\gg}^\rho$ to $\pi$: the lowest $i$ blocks in $\Jord_\rho$ have been returned to the original position they had in $\psi$, while the higher blocks remain in the position they have in $\psi_{\gg}^\rho$. In particular, $\pi_{(0)}=\pi_{\gg}^\rho$ and $\pi_{(r)}=\pi$. We let $L_i$ be the representation which corresponds to the generalized segment defined by the data $[B_i,A_i]$, $T_i$, and $\zeta_i$, namely
\[
\begin{bmatrix}
\zeta_i(B_i+T_i) & \dotsb & \zeta_i(A_i+T_i)\\
\vdots & \ddots & \vdots\\
\zeta_i(B_i+1) & \dotsb & \zeta_i(A_i+1)
\end{bmatrix}.
\] 
We have
\begin{lem}[2.8 in \cite{moeglin2010holomorphie}]
\label{lem_L}\mbox{} 
\begin{enumerate}[(i)]
    \item $\pi_{(i-1)}$ is the unique irreducible subrepresentation of $L_{i} \rtimes \pi_{(i)}$. Moreover,
    \[\Jac_{\zeta_i[B_i+T_i, A_i+T_i]\to \zeta_i[ B_i, A_i]}(L_{i} \rtimes \pi_{(i)})
    \] is irreducible, and thus equal to $\pi_{(i)}$.
    \item  Consequently, $\pi_{\gg}^\rho$ is the unique irreducible subrepresentation of $L_1 \times L_2 \times \dotsb \times L_r \rtimes \pi$. Moreover,
    \[\Jac_{L_1,\dotsc,L_r}(L_1 \times L_2 \times \dotsb \times L_r \rtimes \pi)
    \]
    is irreducible, and thus equal to $\pi$.
    Here $\Jac_{L_1,\dotsc,L_r}$ denotes 
    \[
    \Jac_{\zeta_r[A_r+T_r, B_r+T_r]\to \zeta_r[A_r,B_r]} \circ \dotsb \circ \Jac_{\zeta_1[A_1+T_1, B_1+T_1]\to \zeta_1[A_1, B_1]}.
    \]
\end{enumerate}
\end{lem}

\begin{rem}
\label{rem_zgusnjavanje}
Like in Remark \ref{rem_T}, we use the above notation throughout the paper: We will freely refer to the representations $L_i$ and use the notation $\Jac_{L_1,\dotsc,L_r}$ anytime we deal with $\pi_{\gg}^\rho$.
\end{rem}
The following result is a special case of Proposition 3.3 (2) in \cite{xu2021combinatorial} (see also Proposition 3.2. of \cite{Moeglin_discrets} for DDR representations).
{\begin{lem}
\label{lem_Jac=0}
Let $\alpha \in \mathbb{R}$ be such that $\frac{\alpha-1}{2} \neq \zeta_iB_i$, for all $i=1,\dotsc r$. Then $\Jac_{ \frac{\alpha-1}{2}}(\pi) = 0$.
\end{lem}}



\subsection{Change of order formulas}
\label{subs_BX}
A key technical result in Xu's work \cite{xu2021combinatorial} describes the so-called change of order formulas. Let $\pi$ be parametrized by $(\psi,\eta,t)$. Assume that the blocks $(A,B,\zeta)$ and $(A',B',\zeta')$ in $\Jord_\rho(\psi)$ are adjacent under some fixed order $<$. Now consider the order $<'$ obtained by swapping $(A,B,\zeta)$ and $(A',B',\zeta')$ and keeping all other blocks in their positions. If $<'$ is admissible, one would like to know $(\eta', t')$ in order for $\pi$ to be parametrized by $(\psi,\eta',t')$ using the order $<'$ on $\Jord_\rho$. Section 6.1 of \cite{xu2021combinatorial} contains formulas for $(\eta',t')$ in this situation. The formulas in question are somewhat complicated, so we only record them in the special case that we need, namely, when one of the segments is a singleton, contained in the other segment.
\begin{lem}
\label{lem_xu}
Let $\psi$ be a parameter containing the blocks $(A,B,\zeta)$ and $(x,x,-)$ for $x \in [B,A]$. Let $>$ be an order on $\Jord(\psi)$ under which the singleton segment is the immediate successor of $(A,B,\zeta)$. Using this order, let $\pi$ be parametrized by $(\psi,\eta,t)$. Let $>'$ be the order obtained from $>$ by swapping the two segments in question. Let $(\psi,\eta',t')$ be the parameter for $\pi$ with respect to $>'$. Assuming $\zeta=-1$, we have:
\begin{itemize}
\item If $\eta(A,B,-) = (-1)^{A-B}\eta(x,x,-)$ and $t(A,B,-) < \lfloor(A-B+1)/2\rfloor$, then
\begin{align*}
t'(A,B,-) &= t(A,B,-)+1;\\
\eta'(x,x,-) &= (-1)^{A-B}\eta(x,x,-);\\
\eta'(A,B,-) &=-\eta(A,B,-).
\end{align*}
\item If $\eta(A,B,-) = (-1)^{A-B}\eta(x,x,-)$ and $t(A,B,-) = \lfloor(A-B+1)/2\rfloor$, then
\begin{align*}
t'(A,B,-) &= \begin{cases}t(A,B,-), &\text{ if $A-B$ is even};\\
t(A,B,-)-1, &\text{ if $A-B$ is odd};
\end{cases}\\
\eta'(x,x,-) &= (-1)^{A-B}\eta(x,x,-);\\
\eta'(A,B,-) &=\eta(A,B,-). 
\end{align*}
\item If $\eta(A,B,-) \neq (-1)^{A-B}\eta(x,x,-)$, then
\begin{align*}
t'(A,B,-) &= t(A,B,-)-1;\\
\eta'(x,x,-) &= (-1)^{A-B}\eta(x,x,-);\\
\eta'(A,B,-) &=-\eta(A,B,-). 
\end{align*}
\end{itemize}
If $\zeta=+1$, then $t' = t$, $\eta'(x,x,-) = (-1)^{A-B+1}\eta(x,x,-)$, and $\eta'(A,B,+) = -\eta(A,B,+)$.
\end{lem}

\subsection{Computing Jacquet modules}
\label{subs_compJac}
Considering the way Arthur packets are constructed, it is not surprising that we often need to compute various Jacquet modules. A very useful tool here is Tadić's formula \cite{Tadic_structure}. For a representation $\pi$ of $G_n$ we let $\mu^*(\pi)$ denote the sum of (the semi-simplifications) of Jacquet modules taken with respect to the maximal standard parabolic subgroups and the whole group itself. We let $m^*$ denote the analogous construction for general linear groups. Then
\begin{equation}
\label{eq_tadic}
\mu^*(\delta \rtimes \pi) = M^*(\delta) \rtimes \mu^*(\pi).
\end{equation}
In place of a precise (but lengthy) description of $M^*$, we offer a rough (but brief) outline:
\begin{equation}
\label{eq_tadic_2}
M^* = (m \otimes 1) \circ (\sim \otimes\ m^*) \circ s \circ m^*,
\end{equation}
where
\begin{align*}
    m &: x \otimes y \mapsto x \times y\\
    s &: x\otimes y \mapsto y \otimes x\\
    \sim\ &= \text{ contragredient}\\
    1 &= \text{ identity mapping}.
\end{align*}
We refer the reader to Theorem 5.4 of \cite{Tadic_structure} for further details. Although these results were originally proven for symplectic and special odd orthogonal groups, they can be extended to full even orthogonal groups (see \cite{MR1739616}), even in the non-quasi-split case (see \cite[\S 16]{MR1896238}). We will need the following special case of the above formula:
\begin{equation}
    \label{Tadic_zeta}
M^*(\zeta(a,b))=  \sum_{i=a-1}^b\sum_{j=i}^b 
\zeta(-b,-(j+1))
\times \zeta(a,i) \otimes \zeta(i+1,j).
\end{equation}
Finally, we record a result on Jacquet modules of generalized segment representations:
\begin{lem}
\label{lem_KretLapid}
Let $\sigma$ be an irreducible representation attached to the generalized segment
\[
\begin{bmatrix}
x_{11} & \dotsb & x_{1n}\\
\vdots & \ddots & \vdots\\
x_{m1} & \dotsb & x_{mn}
\end{bmatrix}.
\]
We assume (without loss of generality) that the rows of the generalized segment are decreasing. Let $[x,y]$ be a segment. Then $m^*(\sigma)$ contains a term of the form $\zeta(x,y) \otimes \xi$ (for some irreducible $\xi$) if and only if $x = x_{11},\ y=x_{k1} $ for some $k\leq m$.
\end{lem}
\begin{proof}
This is a direct consequence of the main theorem in \cite{Kret_Lapid}
\end{proof}

\subsection{On standard modules}
\label{subs_stdmod}
Recall that any $\pi \in \Irr(G)$ is the unique irreducible quotient of a (unique) standard module $\nu^{s_k}\delta_k \times \dotsb \times \nu^{s_1}\delta_1 \rtimes \tau$.
We may use this quotient form of the Langlands classification interchangeably with the subrepresentation form, thanks to the so-called MVW involution:
\begin{lem}[Lemma 2.2 in \cite{Atobe_Gan}]
\label{lemma:MVWinv}
Let $\tau_i \in \Irr(\textnormal{GL}_{t_i}(F))$, $i=1,\dotsc,k$ and $\pi_0 \in \Irr(G(W_{n_0}))$. Let $P$ be a standard parabolic subgroup of $G(W_n)$ ($n = n_0 + 2\sum t_i$) with Levi component equal to $\textnormal{GL}_{t_1}(F) \times \dotsm \times \textnormal{GL}_{t_k}(F) \times G(W_{n_0})$. Then, for any $\pi \in \Irr(G(W_{n}))$, the following statements are equivalent:
\begin{enumerate}[(i)]
\item $\pi \hookrightarrow \tau_1 \times \dotsm \times \tau_k \rtimes \pi_0$; 
\item $\tau_1^\vee \times \dotsm \times \tau_k^\vee \rtimes \pi_0 \twoheadrightarrow \pi$.
\end{enumerate}
\end{lem}
Because each $\nu^{s_i}\delta_i$ is given by a segment, we often view the ($\GL$ part of the) standard module simply as a collection of segments. 
In a few places in this paper we refer to the \emph{standard module of $\pi$ above} (resp.~\emph{below}) $z$ for an irreducible representation $\pi$ and some fixed $z>0$. Let us explain this.
In \S\ref{subsubs_Langlands} the segments were sorted by their midpoints; we point out that we can also sort them lexicographically with respect to their endpoints. 
If we collect all the segments $[x,y]$ with $x \geq z$ (resp.\,$x<z$), we obtain a standard representation of a general linear group; let $\Sigma_{\geq z}$ (resp.~$\Sigma_{< z}$) denote the Langlands quotient of this representation. Then we can write any irreducible representation $\pi$ as the unique irreducible quotient of
\[
\Sigma_{\geq z} \times \Sigma_{< z} \rtimes \tau.
\]
We will need:
\begin{lem}
\label{lem_standardni_mod_iznad}
Let $z > 0$. 
Then
$\Jac_{-z}(\pi) \neq 0$ if and only if $m^*(\Sigma_{\geq z})$ contains a subquotient of the form $\xi \otimes \nu^{z}$.
\end{lem}
\begin{proof}
This follows from Tadić's formula \eqref{eq_tadic} and \eqref{eq_tadic_2}. First, assume $\Jac_{-z}(\pi) \neq 0$. Then $\Jac_{-z}(\Sigma_{\geq z} \times \Sigma_{< z} \rtimes \tau) \neq 0$. Applying formula \eqref{eq_tadic_2} to $\Sigma_{< z}$, and using the fact $\tau$ is tempered, we see that $\mu^*(\Sigma_{< z} \rtimes \tau)$ cannot contain a subquotient of the form $\nu^{-z} \otimes  \dotsb$. Thus $\Jac_{-z}(\Sigma_{\geq z} \times \Sigma_{< z} \rtimes \tau) \neq 0$ implies $M^*(\Sigma_{\geq z})$ contains a subquotient of this form. Now we use formula \eqref{eq_tadic_2}: $M^*=(m \otimes 1) \circ (\sim \otimes\ m^*) \circ s \circ m^*$ shows the subquotients of $M^*(\Sigma_{\geq z})$ are of the form 
\[
\tilde{B}\times A_1 \otimes A_2,
\]
where $A\otimes B \leq m^*(\Sigma_{\geq z})$ and $A_1 \otimes A_2 \leq m^*(A)$. But $\widetilde{B}\times A_1 = \nu^{-z}$ if and only if $A_1 = 1$ (trivial) and $B = \nu^z$. This shows $A\otimes \nu^z \leq m^*(\Sigma_{\geq z})$, which we needed to prove.

Conversely, assume $\xi \otimes \nu^{z} \leq m^*(\Sigma_{\geq z})$. This implies that the Jacquet module of $\Sigma_{\geq z}$ (with respect to the appropriate parabolic subgroup) contains a quotient of this form. But $r_P(\Sigma_{\geq z}) \to \xi \otimes \nu^{z} $
implies, using Frobenius, that $\Sigma_{\geq z} \hookrightarrow \xi \times \nu^{z}$.
Equivalently, we have $\nu^{z} \times \xi \twoheadrightarrow \Sigma_{\geq z}$, which in turns shows
\[
\nu^{z} \times \xi \times \Sigma_{< z} \rtimes \tau \twoheadrightarrow \pi.
\]
This implies $\Jac_{-z}(\pi) \neq 0$, which we needed to show.
\end{proof}

\subsection{Kudla's filtration}
\label{subs_Kudla}
One of the basic tools we use is Kudla's filtration, which describes the Jacquet modules (with respect to maximal parabolics) of the Weil representation. The result was originally proved as Theorem 2.8 in Kudla's paper \cite{Kudla2}; we state it here using the notation of the present paper:
\begin{prop}
\label{prop_Kudla}
The Jacquet module $R_{P_k}(\omega_{m,n})$ has a $\GL_k\times G(W_{n-2k}) \times H(V_m)$-equivariant filtration
\[
R_{P_k}(\omega_{m, n}) = R^0 \supset R^1 \supset \dotsb \supset R^k \supset R^{k+1} = 0
\]
in which the successive quotients $J^a = R^a/R^{a+1}$ are given by
\[
J^a = \text{Ind}_{P_{k-a,a}\times G_{n-2k}\times Q_a}^{\GL_k\times G_{n-2k}\times H_m}\left({\chi_V|\emph{\text{det}}_{\GL_{k-a}}|^{\lambda_{k-a}}\otimes \Sigma_a \otimes \omega_{m-2a, n-2k}}\right).
\]
Here
\begin{itemize}
\item $\lambda_{k-a} = (m-n+k-a-\epsilon)/2$;
\item $P_{k-a,a} =$ standard parabolic subgroup of $\GL_k$ with Levi factor $\GL_{k-a}\times \GL_a$;
\item $\Sigma_a = C_c^\infty(\GL_a)$, the space of locally constant compactly supported functions on $\GL_a$. The action of $\GL_a \times \GL_a$ on $\Sigma_a$ is given by
\[
[(g,h)\cdot f](x) = \chi_V(\det(g))\chi_W(\det(h))f(g^{-1}\cdot x\cdot h).
\]
\end{itemize}
If $m-2a$ is less than the dimension of the first (anisotropic) space in $\mathcal{V}$, we put $R^a=J^a=0$.
\end{prop}
Let us derive some corollaries. The following lemma is useful:
\begin{lem}
\label{lem_Kudla}
Let $\pi\in \Irr(G_n), \pi_0 \in \Irr(G_{n-2k})$. Let $l = n - m + \epsilon$ and let $\sigma$ be a generalized segment representation (\S \ref{subsubs_segments}), written with decreasing rows. Assume that
\[
\tag{K}
\text{the first column of the generalized segment does not contain }\frac{l-1}{2}.
\]
Then
$
\chi_V\sigma \rtimes \pi_0 \twoheadrightarrow \pi 
$
implies
$
\chi_W\sigma \rtimes \Theta_l(\pi_0) \twoheadrightarrow  \Theta_l(\pi)
$.
\end{lem}
\begin{proof}
This is a straightforward application of Kudla's filtration. The case of a single segment is proven in (e.g.) Corollary 5.3 of \cite{Atobe_Gan}; the proof for generalized segments is the same, and we sketch it here. Instead of $\chi_V\sigma \rtimes \pi_0 \twoheadrightarrow \pi$, we write $ \pi \hookrightarrow \chi_V\sigma^\vee \rtimes \pi_0$. Thus
\begin{align*}
\Theta_l(\pi)^\vee \cong \Hom_{G_n}(\omega_{m,n},\pi)_\infty
&\hookrightarrow \Hom_{G_n}(\omega_{m,n}, \chi_V\sigma^\vee \rtimes \pi_0)_\infty\\
&\cong \Hom_{\GL_k\times G_{n-2k}}(r_{P_k}(\omega_{m,n}), \chi_V\sigma^\vee \otimes \pi_0)_\infty,
\end{align*}
using Frobenius reciprocity. Here, and elsewhere, the subscript $\infty$ indicates the subspace comprised of all the smooth vectors.

We now use Kudla's filtration; the result will follow if we can show that $\chi_V\sigma^\vee \otimes \pi_0$ cannot appear as a quotient of $J^a$ for $a<k$. Indeed, assuming this, we can continue the above computation to get
\begin{align*}
\Theta_l(\pi)^\vee &\hookrightarrow \Hom_{\GL_k\times G_{n-2k}}(J^k, \chi_V\sigma^\vee \otimes \pi_0)_\infty\\
&\cong \Hom_{\GL_k\times G_{n-2k}}(\Sigma_k \otimes \omega_{n-2k,m-2k}, \chi_V\sigma^\vee \otimes \pi_0)_\infty\\
&\cong \left(\chi_W\sigma \rtimes \Theta_{l}(\pi_0)\right)^\vee,
\end{align*}
which is what we want. It remains to show that $\Hom_{\GL_k\times G_{n-2k}}(J^a, \chi_V\sigma^\vee \otimes \pi_0) = 0$ for $a < k$. Bernstein's Frobenius reciprocity shows that this hom-space is isomorphic to
\[
 \Hom_{\GL_{k-a}\times \GL_a\times G_{n-2k}}({\chi_V|\text{det}_{\GL_{k-a}}|^{\lambda_{k-a}}\otimes \Sigma_a \otimes \omega_{m-2a, n-2k}}, R_{\overline{P}_{k-a,a}}(\chi_V\sigma^\vee) \otimes \pi_0).
\]
We show that this space is trivial by proving that $R_{\overline{P}_{k-a,a}}(\chi_V\sigma^\vee) = R_{{P}_{k-a,a}}(\chi_V\sigma)^\vee$ cannot have a subquotient of the form $\chi_V|\text{det}_{\GL_{k-a}}|^{\lambda_{k-a}} \otimes \xi$ (for a representation $\xi$ of $\GL_{a}$). Indeed, $|\text{det}_{\GL_{k-a}}|^{\lambda_{k-a}} = \zeta(\frac{1-l}{2},\frac{-l-1}{2}+k-a)$, so $R_{{P}_{k-a,a}}(\chi_V\sigma)$ would have to contain $\zeta(\frac{1+l}{2}+a-k,\frac{l-1}{2})$ for the above hom-space to be non-zero. But Lemma \ref{lem_KretLapid} shows that this is impossible if $\frac{l-1}{2}$ does not appear in the first column of $\sigma$. Thus $\Hom_{\GL_k\times G_{n-2k}}(J^a, \chi_V\sigma^\vee \otimes \pi_0) = 0$. 
\end{proof}
We often use the above lemma to show (non-)vanishing:
\begin{lem}
\label{lem_Kudla_nonvanishing}
Assume $\chi_V\sigma \rtimes \pi_0 \twoheadrightarrow \pi$, where $\sigma$ is a generalized segment representation.
\begin{enumerate}[(i)]
\item If $\sigma$ satisfies condition (K) with $l=\alpha$, then $\Theta_\alpha(\pi_0)=0$ implies $\Theta_\alpha(\pi)=0$.
\item If $\sigma$ satisfies condition (K) with $l = -\alpha$, then $\Theta_\alpha(\pi)=0$ implies $\Theta_\alpha(\pi_0)=0$.
\end{enumerate}
\end{lem}
\begin{proof}
\begin{enumerate}[(i)]
\item Since $\sigma$ satisfies (K), we may apply Lemma \ref{lem_Kudla} to get $\chi_W\sigma \rtimes \Theta_\alpha(\pi_0) \twoheadrightarrow \Theta_\alpha(\pi)$. Now $\Theta_\alpha(\pi_0)=0$ clearly implies $\Theta_\alpha(\pi)=0$.
\item Assume $\Theta_\alpha(\pi_0) \neq 0$. Then the Conservation relation shows $\Theta_{-\alpha}'(\pi_0)= 0$, where $\Theta'$ denotes the lift to the other tower. Since $\sigma$ satisfies (K) with $l = -\alpha$, we may use Lemma \ref{lem_Kudla} to compute $\Theta_{-\alpha}'$; we get $\chi_W\sigma \rtimes \Theta_{-\alpha}'(\pi_0) \twoheadrightarrow \Theta_{-\alpha}'(\pi)$. Now $\Theta_{-\alpha}'(\pi_0)= 0$ implies $\Theta_{-\alpha}'(\pi)= 0$.
But then the Conservation relation shows $\Theta_{\alpha}(\pi) \neq  0$.\qedhere
\end{enumerate}
\end{proof}
Lemma \ref{lem_Kudla} can often be refined to give information about the small theta lifts:
\begin{lem}
\label{lem_Kudla_small}
Assume $\chi_V\sigma \rtimes \pi_0 \twoheadrightarrow \pi$ like in Lemma \ref{lem_Kudla}. Assume $\sigma$ satisfies condition (K) for $l= \alpha$. Then we may deduce
\[
\chi_W\sigma \rtimes \theta_\alpha(\pi_0) \twoheadrightarrow \theta_\alpha(\pi)
\]
in either of the following two situations:
\begin{enumerate}[(i)]
    \item $r_{\overline{P}}(\theta_\alpha(\pi))$ has only one irreducible subquotient on which $\GL_k(F)$ acts by $\chi_W\sigma$; or
    \item $r_{\overline{P}}(\pi)$ has only one irreducible subquotient on which $\GL_k(F)$ acts by $\chi_V\sigma$, and additionally, $\sigma$ satisfies (K) for $l=-\alpha$.
\end{enumerate}
Here $P$ is the appropriately chosen maximal standard parabolic subgroup, and $\overline{P}$ denotes its opposite. 
\end{lem}
\begin{proof}
\begin{enumerate}[(i)]
\item Lemma \ref{lem_Kudla} implies $\chi_W\sigma \rtimes \Theta_\alpha(\pi_0) \twoheadrightarrow \Theta_\alpha(\pi)\twoheadrightarrow \theta_\alpha(\pi)$. Bernstein's Frobenius reciprocity then shows there is a non-zero map $\chi_W\sigma \otimes \Theta_\alpha(\pi_0) \to r_{\overline{P}}(\theta_\alpha(\pi))$. Since $\Theta_\alpha(\pi_0)$ has a unique irreducible quotient, this implies that $\chi_W\sigma \otimes \theta_\alpha(\pi_0)$ is an irreducible subquotient of $ r_{\overline{P}}(\theta_\alpha(\pi))$. Assuming (i), we find that it is a subrepresentation, and using Frobenius again, the claim follows.

\item Lemma \ref{lem_Kudla} shows $\chi_W\sigma \rtimes \Theta_\alpha(\pi_0) \twoheadrightarrow \Theta_\alpha(\pi)\twoheadrightarrow \theta_\alpha(\pi)$, which means that there is an irreducible subquotient $\pi_0'$ of $\Theta_\alpha(\pi_0)$ such that $\chi_W\sigma \rtimes \pi_0' \twoheadrightarrow \theta_\alpha(\pi)$. Because of the additional assumption in (ii), we can now use the lemma in reverse (to compute $\Theta_{-\alpha}$); we get $\chi_V\sigma \rtimes \Theta_{-\alpha}(\pi_0') \twoheadrightarrow \pi$, i.e.~a non-zero map $\chi_V\sigma \otimes \Theta_{-\alpha}(\pi_0') \to r_{\overline{P}}(\pi)$. Since $\Theta_{-\alpha}(\pi_0')$ has a unique irreducible quotient, it follows that there is a non-zero map $\chi_V\sigma \otimes \theta_{-\alpha}(\pi_0') \to r_{\overline{P}}(\pi)$. But if $r_{\overline{P}}(\pi)$ only has one subquotient on which $\GL_k(F)$ acts by $\chi_V\sigma$, this implies $\theta_{-\alpha}(\pi_0') = \pi_0$.
Therefore $\pi_0' = \theta_{\alpha}(\pi_0)$, so $\chi_W\sigma \rtimes \theta_{\alpha}(\pi_0) \twoheadrightarrow \theta_\alpha(\pi)$. Note that we are using Convention \ref{conv_dettwist} here.\qedhere
\end{enumerate}
\end{proof}
We can now use the above results to show that theta lifts are well-behaved with respect to stretching the parameters. To that end, we recall the setting of Lemma \ref{lem_L}. Let $\pi \in \Pi_\psi$ be a representation of Arthur type parametrized by $\psi$, and let $\Jord_{\chi_V}(\psi) = \{(A_i,B_i,\zeta_i): i=1,\dotsc,r\}$. Let $\psi_{\gg}^{\chi_V}$ be a $\chi_V$-DDR parameter which dominates $\psi$, and let $\pi_{\gg}^{\chi_V}$ be the corresponding representation, so that
\[
\pi_{\gg}^{\chi_V} \hookrightarrow \chi_VL_1\times \dotsb\times \chi_VL_r \rtimes \pi.
\]
We also have $\pi_{(i)}$ for each $i\in \{0,\dotsc,r\}$: this is the $i$-th intermediate step in going from $\pi_{\gg}^{\chi_V}$ to $\pi$. For $i < r$, the representation $\pi_{(i)}$ is the unique irreducible subrepresentation of $\chi_VL_{i+1} \rtimes \pi_{(i+1)}$, and
\[
\pi_{(i)} \hookrightarrow \chi_VL_{i+1}\times \dotsb\times \chi_VL_r \rtimes \pi.
\]
\begin{prop}
\label{prop_wwtawwtaKudla}
Fix $k \in \{0,\dotsc,r-1\}$ and consider
\[
\pi_{(k)} \hookrightarrow \chi_VL_{k+1}\times \dotsb\times \chi_VL_r \rtimes \pi
\]
as above. Let $\alpha$ be an odd integer. We make the following assumption about the blocks $(A_i,B_i,\zeta_i)$, $i>k$ that are being shifted to go from $\psi$ to $\psi_{(k)}$:
\begin{itemize}
\item if $\zeta_i = 1$, then $\frac{|\alpha|-1}{2} < B_i$; and
\item if $\zeta_i = -1$, then $\frac{|\alpha|-1}{2} \leq A_i$.
\end{itemize}
Then 
\begin{enumerate}[a)]
\item $\Theta_\alpha(\pi_{(k)})\neq 0$ if and only if $\Theta_\alpha(\pi) \neq 0$.
\item If $B_i > 0$ for all $i>k$, then  $\pi_{(k)}$ and $\pi$ share the same going-up tower (and thus also the same going-down tower).
\item Finally, 
\[
\theta_\alpha(\pi_{(k)}) \hookrightarrow \chi_WL_{k+1}\times \dotsb\times \chi_WL_r \rtimes \theta_\alpha(\pi).
\]
\end{enumerate}
\end{prop}
\begin{proof}
Recall that $L_i$ is the generalized segment representation which corresponds to shifting the block $(A_i+T_i,B_i+T_i,\zeta_i)$ back to $(A_i,B_i,\zeta_i)$. Instead of $\pi_{(i-1)}\hookrightarrow \chi_VL_i\rtimes \pi_{(i)}$, we may write
\[
\chi_VL_i^\vee\rtimes \pi_{(i)} \twoheadrightarrow \pi_{(i-1)}.
\]
The point of the two bullets above is that they guarantee that $L_i^\vee$ satisfies condition (K) with both $l = \alpha$ and $l = -\alpha$. Indeed, suppose that $\zeta_i = 1$. Then all the entries in $L_i^\vee$ are negative, so condition (K) clearly holds for $l = |\alpha|$. On the other hand, the first column of $L_i^\vee$ is $-(A_i+1), \dotsc, -(B_i+1)$. By the assumption in the first bullet, $\frac{-|\alpha|-1}{2} > -(B_i+1)$, so $\frac{-|\alpha|-1}{2}$ does not appear in the first column of $L_i^\vee$. When $\zeta_i = -1$, all the entries in $L_i^\vee$ are positive, so condition (K) clearly holds for $l = -|\alpha|$. On the other hand, the first column of $L_i^\vee$ is $A_i+1, \dotsc, A_i+T_i$. Therefore, the condition in the second bullet ensures that $\frac{|\alpha|-1}{2}$ does not appear in the first column. We now observe:
\begin{enumerate}[a)]
\item Because (K) is satisfied for both $l = \pm \alpha$, we may use both (i) and (ii) of Lemma \ref{lem_Kudla_nonvanishing}. We get $\Theta_\alpha(\pi_{(i-1)}) \neq 0 \iff \Theta_\alpha(\pi_{(i)}) \neq 0$.
\item For any representation $\pi$, the going-up tower is the one on which $\Theta_1(\pi) = 0$. Taking $\alpha = 1$ in part a), we see that both bullets are satisfied when $B_i > 0$. Thus $\Theta_1(\pi_{(i-1)}) \neq 0 \iff \Theta_1(\pi_{(i)}) \neq 0$. In other words, both $\pi_{(i-1)}$ and $\pi_{(i)}$ have the same going-up tower.
\item Because (K) holds for $l=\pm \alpha$, and because of the uniqueness of the appropriate Jacquet modules (see Lemma \ref{lem_L}) we may apply case (ii) of Lemma \ref{lem_Kudla_small} to get $
\chi_WL_i^\vee\rtimes \theta_\alpha(\pi_{(i)}) \twoheadrightarrow \theta_\alpha(\pi_{(i-1)})
$, or equivalently,
\[
\theta_\alpha(\pi_{(i-1)}) \hookrightarrow \chi_WL_i \rtimes \theta_\alpha(\pi_{(i)}).
\]
\end{enumerate}
Repeating this reasoning for each $i=k+1,\dotsc, r$, we get the desired result.
\end{proof}


Finally, we will often use the following (cf.~Lemma 5.1 of \cite{muic2004howe}):
\begin{lem}
\label{lem_descent}
Let $\pi \in \Irr(G_n)$. Assume  $\Theta_{-\alpha}(\pi)\neq 0$ for some $\alpha>0$. If $\Jac_{{\frac{1-\alpha}{2}}}^{\chi_V}(\pi)=0$, but $\Jac_{{\frac{1-\alpha}{2}}}^{\chi_W}(\theta_{-\alpha}(\pi)) \allowbreak\neq 0$, then $\Theta_{2-\alpha}(\pi)\neq 0$. Moreover, if $\Jac_{{\frac{1-\alpha}{2}}}^{\chi_W}(\theta_{-\alpha}(\pi))$ is irreducible, then $\theta_{2-\alpha}(\pi)=\Jac_{{\frac{1-\alpha}{2}}}^{\chi_W}(\theta_{-\alpha}(\pi))$.
\end{lem}
\begin{proof}
The non-vanishing of $\Jac_{{\frac{1-\alpha}{2}}}^{\chi_W}(\theta_{-\alpha}(\pi))$ implies that there exists an irreducible representation $\xi$ such that $\theta_{-\alpha}(\pi)\hookrightarrow \chi_W|\cdot|^{\frac{1-\alpha}{2}}\rtimes \xi$.
We then have 
\begin{align*}
\pi^\vee &\hookrightarrow \Hom_{G_n}(\omega_{m,n}, \theta_{-\alpha}(\pi))_\infty\\
&\hookrightarrow \Hom_{G_n}(\omega_{m,n}, \chi_W|\cdot|^{\frac{1-\alpha}{2}}\rtimes \xi)_\infty\\
&\cong \Hom_{\GL_1\times G_{n-2}}(r_P(\omega_{m,n}), \chi_W|\cdot|^{\frac{1-\alpha}{2}}\otimes \xi)_\infty,
\end{align*}
where $P$ is the appropriate maximal standard parabolic subgroup. This proves that the last $\Hom$-space is non-zero. Now Kudla's filtration shows that this implies one of the following:
\begin{itemize}
    \item $\Jac_{{\frac{1-\alpha}{2}}}^{\chi_V}(\pi)\neq 0$: or
    \item $\Theta_{\alpha-2}(\xi) \twoheadrightarrow \pi$, so $\theta_{\alpha-2}(\xi) \twoheadrightarrow \pi$ (in particular, $\theta_{2-\alpha}(\pi) = \xi \neq 0$).
\end{itemize}
Since we are assuming the first option is not true, the second must hold.
If $\Jac_{{\frac{1-\alpha}{2}}}^{\chi_W}(\theta_{-\alpha}(\pi))$ is irreducible (and thus equal to $\xi$), this shows $\theta_{2-\alpha}(\pi)\cong \xi =\Jac_{{\frac{1-\alpha}{2}}}^{\chi_W}(\theta_{-\alpha}(\pi))$.
\end{proof}

\begin{rem}
\label{rem_chi}
Kudla's filtration shows that all the action in theta correspondence occurs along a single cuspidal line: From now on, in all the constructions involving segments (e.g.~\S \ref{subs_Arthurtype}), we have either $\rho=\chi_V$ or $\rho=\chi_W$. In the remainder of the paper, we omit $\chi_V$ and $\chi_W$ from the notation whenever it is possible to do so without causing confusion. In particular, we write $\Jac_{x}$ and $L \rtimes \pi$ instead of $\Jac_{x}^{\chi_V}$ and $\chi_VL \rtimes \pi$. 
\end{rem}

%

\section{Theorem A}
\label{sec_thmA}
Let $\pi$ be parametrized by
\[
\psi = \bigoplus_{i=1}^r\overset{t_i,\eta_i}{\chi_V\otimes S_{a_i}\otimes S_{b_i}}\,\oplus\, \bigoplus_{\rho\neq \chi_V}\overset{t,\eta}{\rho\otimes S_a\otimes S_b}.
\]
Proposition \ref{prop_visoko} shows that $\theta_{-\alpha}(\pi)$ is in the A-A packet for $\alpha\gg0$, with parameter
\[
\psi_\alpha = \bigoplus_{i=1}^r\overset{t_i,-\eta_i}{\chi_W \otimes S_{a_i}\otimes S_{b_i}}\oplus \overset{\epsilon}{\chi_W\otimes  S_1\otimes S_{\alpha}}  \,\oplus \, \bigoplus_{\rho\neq \chi_V}\overset{t,\eta}{\chi_W\chi_V^{-1}\rho\otimes S_a\otimes S_b}.
\]
Note that $\epsilon$ determines the target tower. 

From now on, we omit the $\bigoplus_{\rho\neq \chi_V}$ part from the parameter of $\pi$ since it remains unchanged in\,---\,and does not affect\,---\,the theta correspondence. We also adhere to notational conventions explained in Remark \ref{rem_chi}. Accordingly, we speak of DDR-representations even though we really mean $\rho$-DDR for $\rho=\chi_V$ or $\chi_W$.

We now prove Theorem \ref{theoremA}. Assuming $\pi_\alpha=\theta_{-\alpha}(\pi)$ and $\pi_{\alpha-2}\neq 0$, we need to prove that $\theta_{-(\alpha-2)}(\pi)=\pi_{\alpha-2}$.
We first prove this in a special case:

\begin{lem}
\label{lem_A_babycase}
Assume $\zeta_jB_j \neq -\frac{\alpha-1}{2}$ for $j=1,\dotsc, r$. Suppose $\pi_\alpha = \theta_{-\alpha}(\pi)$ for some $\alpha > 2$. If $\pi_{\alpha-2}\neq 0$ then $\theta_{-(\alpha-2)}(\pi)=\pi_{\alpha-2}$. In particular, $\theta_{-(\alpha-2)}(\pi)\neq 0$.
\end{lem}
\begin{proof}
Let $(\psi_{\alpha})_{\gg}$ be a DDR parameter which dominates $\psi_{\alpha}$; denote the corresponding representation by $(\pi_{\alpha})_{\gg}$. The precise order we use on $\Jord(\psi_\alpha)$ is non-important, but we require the following: there exists an index $i$ such that for all $j\le i$ the blocks $(A_j,B_j,\zeta_j)$ satisfy $B_j\le \frac{\alpha-3}{2}$, and for $j>i$ we have $B_j\ge \frac{\alpha -1}{2}.$
Let $\beta$ denote the position to which $\chi_W\otimes  S_1\otimes S_{\alpha}$ gets taken when $\psi_\alpha$ is stretched to $(\psi_{\alpha})_{\gg}$. Then, 
\[
(\pi_{\alpha})_{\gg}\hookrightarrow L_1\times \cdots \times L_i\times \zeta(-\frac{\beta-1}{2},-\frac{\alpha+1}{2})\times \nu^{-\frac{\alpha-1}{2}}\times L_{i+1}\times \cdots \times L_r\times \pi_{\alpha-2}.\]
The above is a typical example of how we stretch parameters. To go from $\pi_{\alpha-2}$ to $(\pi_{\alpha})_{\gg}$, we first move the blocks $j=1,\dotsc,i$, then the added singleton block (first from $\alpha-2$ to $\alpha$, and then from $\alpha$ to $\beta$), followed by the blocks $j=i+1,\dots,r$.
Our assumption that $\zeta_jB_j \neq -\frac{\alpha-1}{2}$ ensures that $\nu^{-\frac{\alpha-1}{2}}$ can switch places with $L_{i+1},\ldots,L_r.$  After applying $\Jac_{L_1,\ldots,L_i,-\frac{\beta-1}{2},\dotsc,-\frac{\alpha+1}{2},L_{i+1},\ldots,L_r}$ we get
\[\pi_{\alpha}\hookrightarrow \nu^{-\frac{\alpha-1}{2}}\rtimes \pi_{\alpha-2}.\]
Since $\pi_{\alpha}= \theta_{-\alpha}(\pi)$, this shows $\Jac_{-\frac{\alpha-1}{2}}(\theta_{-\alpha}(\pi))$ is irreducible and isomorphic to $\pi_{\alpha-2}$. Moreover, $\Jac_{-\frac{\alpha-1}{2}}(\pi)\allowbreak =0$ by Lemma \ref{lem_Jac=0}. The result now follows from Lemma \ref{lem_descent}.
\end{proof}
We can now deduce Theorem A for general $\pi$ from Lemma \ref{lem_A_babycase}.
\begin{proof}[Proof of Theorem \ref{theoremA}]
First, we fix an admissible order $>$ on $\Jord(\psi)=\{(A_i,B_i,\zeta_i): i=1,\dotsc,r\}$. The precise nature of this order is non-important, but we require that there exist an index $k\in \{1\dotsc,r\}$ such that $j>k$ if and only if 
\begin{itemize}
    \item $B_j > \frac{\alpha-1}{2}$ or
    \item $B_j = \frac{\alpha-1}{2}$ and $\zeta_j = -1$.
\end{itemize}
 We extend the order $>$ to $\psi_{\alpha}= \psi \oplus 1\otimes S_1 \otimes S_\alpha$ by inserting $1\otimes S_1 \otimes S_\alpha$ between $(A_k,B_k,\zeta_k)$ and $(A_{k+1},B_{k+1},\zeta_{k+1})$.
Now let $\pi_{>}$ and $(\pi_{\alpha})_{>}$ denote the representations obtained from $\pi$ and $\pi_{\alpha}$ by shifting up all the blocks $(A_i,B_i,\zeta_i)$ for $i > k$.
Note that $\pi_{>}$ then satisfies the conditions of Lemma \ref{lem_A_babycase}. Furthermore, note that the blocks we are shifting satisfy the bullet points of Proposition \ref{prop_wwtawwtaKudla}. To summarize, we work with the following representations:
\begin{itemize}
\item $\pi$ and its stretched version, $\pi_>$;
\item $\pi_\alpha$ (resp.\,$(\pi_{\alpha})_{>}$) obtained by inserting $1\otimes S_1 \otimes S_\alpha$ into the parameter $\psi$ (resp.\,$\psi_{>}$);
\item $\pi_{\alpha-2}$ and $(\pi_{\alpha-2})_{>}$, obtained from $\pi_{\alpha}$ and $(\pi_{\alpha})_{>}$ by replacing $1\otimes S_1 \otimes S_\alpha$ with $1\otimes S_1 \otimes S_{\alpha-2}$.
\end{itemize}
We break up the proof into three steps:
\bigskip

\noindent \textbf{Step 1:} $\pi_{\alpha}=\theta_{-\alpha}(\pi)$ implies $(\pi_{\alpha})_{>}=\theta_{-\alpha}(\pi_>)$. Indeed, by construction, 
\[
(\pi_{\alpha})_{>} \hookrightarrow L_{k+1} \times \dotsb \times L_r\rtimes \pi_{\alpha}.
\]
Since $\theta_\alpha(\pi_\alpha)=\pi\neq 0$, Proposition \ref{prop_wwtawwtaKudla} a) shows $\theta_\alpha((\pi_{\alpha})_{>})\neq 0$. Furthermore, part c) of the same proposition shows
\[
\theta_\alpha((\pi_{\alpha})_{>}) \hookrightarrow L_{k+1} \times \dotsb \times L_r\rtimes \pi.
\]
Since $\pi_{>}$ is precisely the unique irreducible subrepresentation of the representation on the right, we conclude $\theta_\alpha((\pi_{\alpha})_{>}) = \pi_{>}$. Equivalently, $\theta_{-\alpha}(\pi_{>}) = (\pi_{\alpha})_{>}$. 

\bigskip

\noindent \textbf{Step 2:} $(\pi_{\alpha-2})_{>} \neq 0 \iff \pi_{\alpha-2} \neq 0$.

\medskip

\noindent By construction, $\pi_{\alpha-2} = \Jac_{L_{k+1}, \dotsc, L_r}((\pi_{\alpha-2})_{>})$. Thus, $(\pi_{\alpha-2})_{>} = 0$ implies $\pi_{\alpha-2} = 0$.

Also by construction, $(\pi_{\alpha-2})_{>} =\Jac_{-\frac{\alpha-1}{2}}((\pi_{\alpha})_{>})$. Thus, if $(\pi_{\alpha-2})_{>}\neq 0$, we get 
\[
(\pi_{\alpha})_{>} \hookrightarrow \nu^{-\frac{\alpha-1}{2}} \rtimes (\pi_{\alpha-2})_{>}.
\]
Since we know that $\pi_{\alpha} = \Jac_{L_{k+1}, \dotsc, L_r}((\pi_{\alpha})_{>})  \neq 0$ (and none of the $L_i$'s contain $\pm \frac{\alpha-1}{2}$), the above embedding forces $\Jac_{L_{k+1}, \dotsc, L_r}((\pi_{\alpha-2})_{>})  \neq 0$. But this means precisely $\pi_{\alpha-2}\neq 0$.

\bigskip

\noindent \textbf{Step 3:} $\theta_{2-\alpha}(\pi)=\pi_{\alpha-2}$. By construction, $\pi_>$ satisfies the conditions of Lemma \ref{lem_A_babycase}. Step 1 shows $(\pi_{\alpha})_{>}=\theta_{-\alpha}(\pi_>)$, and Step 2 shows $(\pi_{\alpha-2})_{>} \neq 0$ (because we are assuming $\pi_{\alpha-2}\neq 0$ in Theorem \ref{theoremA}). All of this combined shows that we can apply Lemma \ref{lem_A_babycase} to $\pi_>$: we get that $\theta_{2-\alpha}(\pi_>)$ is non-zero and isomorphic to $ (\pi_{\alpha-2})_{>}:=\Jac_{-\frac{\alpha-1}{2}}((\pi_{\alpha})_{>})$. We now have
\[
\theta_{2-\alpha}(\pi_>) = (\pi_{\alpha-2})_{>} \hookrightarrow L_{k+1} \times \dotsb \times L_r\rtimes \pi_{\alpha-2},
\]
and it remains to compute $\theta_{\alpha-2}$. We apply Proposition \ref{prop_wwtawwtaKudla} again (the bullet points are still satisfied when we replace $\alpha$ by $\alpha-2$) to get
\[
\pi_> \hookrightarrow L_{k+1} \times \dotsb \times L_r\rtimes \theta_{\alpha-2}(\pi_{\alpha-2}).
\]
Note that we are using Convention \ref{conv_dettwist} when applying $\theta_{\alpha-2}$ here.
Since $\Jac_{L_{k+1}, \dotsc, L_r}(\pi_>) = \pi$, we conclude $\theta_{\alpha-2}(\pi_{\alpha-2}) = \pi$. Equivalently, $\pi_{\alpha-2}=\theta_{2-\alpha}(\pi)$, which we needed to prove.
\end{proof}

Theorem \ref{theoremA} suggests the following question: assuming $\pi_\alpha \neq 0$, when is $\pi_{\alpha-2} \neq 0$? As mentioned in \S \ref{sec_conj}, the question of non-vanishing has been resolved by the work of Xu \cite{xu2021combinatorial} and Atobe \cite{atobe2022construction}: given any $\pi$, one can determine the (non-)vanishing of $\pi_{\alpha-2}$ using the algorithms they developed. However, we do not need to go through the entire algorithm every time we change $\alpha$ to $\alpha-2$. It is often easier to deduce the non-vanishing of $\pi_{\alpha-2}$ by comparing $\psi_{\alpha-2}$ to $\psi_\alpha$; after all, the two parameters are very similar. Indeed, there are certain simple situations\,---\,that is, positions of the added term $\chi_W\otimes S_{1} \otimes S_{\alpha}$ relative to the other terms in $\psi_\alpha$\,---\,in which we always know that $\pi_\alpha \neq 0$ implies $\pi_{\alpha-2} \neq 0$. These will be useful to keep in mind going forward:
\begin{lem}
\label{lem_where_can_it_stop}
Let $\pi$ be parametrized by 
\[
\psi = \bigoplus_{i=1}^r\overset{t_i,\eta_i}{\chi_V\otimes S_{a_i} \otimes S_{b_i}} 
\]
where we have indexed the terms according to the order on $\Jord({\psi})$. Fixing $j\in \{1,\dotsc,r\}$, assume that we have $A_i \ll B_j$ for $i < j$ ($0 \ll B_1$ if $j=1$) and $B_i \gg A_j$ for $i > j$.

Now suppose $\pi_\alpha\neq 0$, where $B_j \leq \frac{\alpha-1}{2} < B_{j+1}$.
Then $\pi_{\alpha-2} \neq 0$ if
\begin{itemize}
\item[(i)] $\frac{\alpha-1}{2} >  A_j+1 $; or
\item[(ii)] $\frac{\alpha-1}{2} \leq A_j$.
\end{itemize}
Finally, if $\frac{\alpha-1}{2}= A_j+1$, then 
\begin{itemize}
\item[(iii)] $
\pi_{\alpha-2} =  0\quad \text{if and only if}\quad a_j < b_j,\ t_j=0, \text{ and }\eta \neq (-1)^{A_j-B_j}\eta_j$.
\end{itemize}
\end{lem}
\begin{proof}
First, we reduce the proof to the case where $\psi$ is DDR. We let $(\psi_\alpha)_\gg$ (resp.\ $(\psi_{\alpha-2})_\gg$) be a DDR parameter which dominates $\psi_\alpha$ (resp.\ $\psi_{\alpha-2}$).
Because we are assuming that $[B_j,A_j]$ is far away from all the other segments, we may construct this DDR parameter without moving the term $\chi_W\otimes S_{a_j} \otimes S_{b_j}$ (i.e.\ with $T_j=0$). For the same reason, we may assume that in this construction $1 \otimes S_1 \otimes S_\alpha$ gets replaced by $1 \otimes S_1 \otimes S_\beta$ where $\beta$ satisfies $\frac{\beta-1}{2} < B_{i}$ for $i>j$. Under these assumptions, we have
\[
\pi_\alpha = \Jac_{L_{j+1},\dotsc,L_r}\Jac_{-\frac{\beta-1}{2},\dotsc,-\frac{\alpha+1}{2}}\Jac_{L_1,\dotsc,L_{j-1}}(\pi_\alpha)_\gg.
\]
Moreover, $\Jac_{-\frac{\beta-1}{2},\dotsc,-\frac{\alpha+1}{2}}$ commutes with $\Jac_{L_i}$ for $i\neq j$. We may therefore write the above equality as
\[
\pi_\alpha = \Jac_{L_1,\dotsc,L_{j-1},L_{j+1},\dotsc,L_r}(\pi_\alpha)_>,
\]
where $(\pi_\alpha)_>$ denotes the (non-zero) representation $\Jac_{-\frac{\beta-1}{2},\dotsc,-\frac{\alpha+1}{2}}(\pi_\alpha)_\gg$. We let $(\psi_{\alpha})_>$ denote the corresponding parameter. One constructs $(\psi_{\alpha-2})_>$ and $(\pi_{\alpha-2})_>$ analogously, by computing $\Jac_{-\frac{\beta-1}{2},\dotsc,-\frac{\alpha+1}{2},-\frac{\alpha-1}{2}}$ instead of $\Jac_{-\frac{\beta-1}{2},\dotsc,-\frac{\alpha+1}{2}}$. In particular, we have $(\pi_{\alpha-2})_> = \Jac_{-\frac{\alpha-1}{2}}(\pi_{\alpha})_>$ (possibly zero).

But once more, $\Jac_{-\frac{\alpha-1}{2}}$ commutes with $\Jac_{L_i}$ for all $i\neq j$. This implies
\[
\Jac_{-\frac{\alpha-1}{2}}\pi_\alpha \neq 0 \Rightarrow \Jac_{-\frac{\alpha-1}{2}}(\pi_\alpha)_>\neq 0.
\]
The reverse implication follows from $(\pi_\alpha)_> \hookrightarrow L_1 \times \dotsb \times L_{j-1} \times L_{j+1} \times \dotsb \times L_{r} \rtimes \pi_\alpha$ combined with the fact that none of the $L_i$'s contain $\pm \frac{\alpha-1}{2}$. Indeed, all the entries in the generalized segment $L_i$ are either smaller (by absolute value) than $\frac{\alpha-1}{2}$ (if $i<j$) or larger, if $i>j$.

Thus $\pi_{\alpha-2}\neq 0$ if and only if $(\pi_{\alpha-2})_>\neq 0$. With this, we are ready to prove the proposition.

In case (i), $(\psi_{\alpha-2})_>$ is DDR, so $(\pi_{\alpha-2})_>\neq 0$ is automatic. In cases (ii) and (iii), the non-vanishing of $(\pi_{\alpha-2})_>$ comes down to checking the compatibility conditions of \cite[Proposition 5.2]{xu2021combinatorial}; one needs to verify the following:
\begin{itemize}
\item[(ii)] If $\chi_W\otimes S_{1} \otimes S_{\alpha}$ and $\chi_W\otimes S_{a_j} \otimes S_{b_j}$ satisfy the compatibility conditions of Proposition 5.2 [ibid.], then the same is true of $\chi_W\otimes S_{1} \otimes S_{\alpha-2}$ and $\chi_W\otimes S_{a_j} \otimes S_{b_j}$.
\item[(iii)] Here $\alpha=a_j+b_j+1$. The only situation in which $\chi_W\otimes S_{1} \otimes S_{a_j+b_j-1}$ and $\chi_W\otimes S_{a_j} \otimes S_{b_j}$ do not satisfy these compatibility conditions is if $a_j < b_j$, $t_j=0$, and $\eta \neq (-1)^{A_j-B_j}\eta_j$.
\end{itemize}
We omit the details.
\end{proof}
\begin{rem}
\label{rem_descent}\begin{enumerate}[a)]
\item The above lemma justifies the following heuristic. Loosely speaking, to determine $d(\pi,\psi)$, one starts by forming $\psi_\alpha$ for $\alpha \gg 0$, and proceeds to lower the added summand through the parameter. One can always lower the added summand $\chi_W\otimes S_{1} \otimes S_{\alpha}$ through “empty space” (case (i) above). Furthermore, the added summand can skip over any “isolated” summands it encounters; the only exception is described by case (iii) above.
\item In particular, case (iii) provides some intuition as to why the Recipe might stop working: the added summand cannot go past a block which has $\zeta = -1$ and $t=0$ if the signs $\eta$ are incompatible. This observation makes it easy to determine $d(\pi,\psi)$ in certain situations (see Examples \ref{ex_two} and \ref{ex_four}).
\end{enumerate}

\end{rem}
We close this section by noting another consequence of Xu's algorithm. In the following lemma, we write $\chi$ and $\chi'$ to signify that we might be dealing with representations of two groups of different type (orthogonal or symplectic):
\begin{lem}
\label{lem_high_blocks_dont_matter}
Suppose $\pi$ and $\pi'$ are two representations parametrized by
\begin{align*}
\psi &= \overbrace{\bigoplus_{i=1}^k\overset{t_i,\eta_i}{\chi\otimes S_{a_i} \otimes S_{b_i}}}^{\text{low blocks}} \,\oplus\, \overbrace{\bigoplus_{i=k+1}^r\overset{t_i,\eta_i}{\chi\otimes S_{a_i} \otimes S_{b_i}}}^{\text{high blocks}} \quad \text{and}\\
\psi'&= \underbrace{\bigoplus_{i=1}^k\overset{t_i,\eta_i'}{\chi' \otimes S_{a_i} \otimes S_{b_i}}}_{\text{low blocks}} \,\oplus\, \underbrace{\bigoplus_{i=k+1}^{r'}\overset{t_i',\eta_i'}{\chi'\otimes S_{a_i'} \otimes S_{b_i'}}}_{\text{high blocks}}
\end{align*}
whose low blocks coincide. Moreover, let $\eta_i' = \eta_i$ for $i=1,\dotsc,k$. We assume all the high blocks (indexed by $i > k$) are far away from the low blocks and from each other.

Suppose that there exists a level $\alpha$ such that $\frac{\alpha-1}{2} < \textnormal{min}\{B_{k+1}, B_{k+1}'\}$ and $d(\pi,\psi), d(\pi',\psi') \leq \alpha$. If the value $\eta(\chi_W\otimes S_1 \otimes S_\alpha)$ is the same in $\psi_\alpha$ and $\psi_\alpha'$, then $d(\pi,\psi) = d(\pi',\psi')$. The same conclusion holds if we assume $\eta_i' = -\eta_i$ for $i=1,\dotsc,k$ and require that $\psi_\alpha$ and $\psi_\alpha'$ have different values of $\eta(\chi_W\otimes S_1 \otimes S_\alpha)$.
\end{lem}
\begin{proof}
We are assuming that $d(\pi,\psi)\leq  \alpha$ and $d(\pi',\psi') \leq \alpha$: therefore, we can apply the Recipe to obtain non-zero representations $\pi_\alpha$ and $\pi_\alpha'$. The lower parts of $\psi_\alpha$ and $\psi_\alpha'$ will still look the same, and we are assuming that both parameters have the same value of $\eta$ on the added block. To find $d(\pi,\psi)$ and $d(\pi',\psi')$ we keep lowering the added block (using the Recipe) until we get parameters that correspond to the zero representation. The starting observation in Xu's algorithm is that the non-vanishing of a representation is independent of the far-away blocks in the corresponding parameter. Since the high blocks are irrelevant, and the low blocks are the same in $\psi_\alpha$ and $\psi'_\alpha$, we conclude $d(\pi,\psi)=d(\pi',\psi')$. 

For the last claim of the Lemma, we note that the outcome of Xu's algorithm remains unchanged if all the signs in the low blocks are flipped.
\end{proof}
\section{Theorem B}
\label{sec_thmB}
We recall the setting of Theorem \ref{theoremB}. We are simultaneously lifting $\pi$ to two different towers, $\mathcal V^+$ and $\mathcal V^-$; we denote the corresponding lifts by $\theta^\pm(\pi)$. For $\alpha \gg 0$, the lifts $\theta_{-\alpha}^\pm(\pi)$ are given by parameters
\[
\psi_{\alpha}^{\pm} = \chi_W\chi_V^{-1}\psi \oplus \overset{\epsilon}{\chi_W \otimes S_1 \otimes S_\alpha}.
\]
To obtain the lower lifts, we descend using the Recipe. At some point, the Recipe will give a parameter $\psi_\alpha^+$ such that $\pi_{\alpha}^+ =0 $; recall that then  $d^{+}(\pi,\psi)=\alpha+2$; $d^{-}(\pi,\psi)$ is defined in an analogous way (see \S \ref{sec_conj}).

It is possible that $\pi_\alpha^\pm \neq 0$ for all $\alpha > 0$. (If that is the case, we set $d^{\pm}(\pi,\psi) =1$.) From the point of view of this paper, this is uninteresting, albeit nice: it implies (using Theorem \ref{theoremA}) that $\theta_{-\alpha}(\pi)$ is non-zero and contained in the A-A packet for all $\alpha>0$. Because of this, we assume that at least one of the indices $d^{\pm}(\pi,\psi)$ is greater than $1$. 

Without loss of generality, assume $d^-(\pi,\psi)\geq d^+(\pi,\psi)$. To simplify notation, let $\alpha = d^-(\pi,\psi)$. Thus the assumption is $\pi_\alpha^\pm \neq 0$, but $\pi_{\alpha-2}^- = 0$ (recall $\alpha \geq 3$). Proving Theorem \ref{theoremB} then amounts to proving the following claims:
\begin{enumerate}[a)]
    \item $\theta_{2-\alpha}^-(\pi)=0$. (This shows that $\mathcal{V}^-$ is the going-up tower and that $\theta_{-\alpha}^-(\pi)$ is the first occurrence of $\pi$ on $\mathcal{V}^-$.)
    \item $\pi_{\alpha-2}^+ \neq 0$.
\end{enumerate}
To prove this, we use the same reduction we used in the proof of Theorem \ref{theoremA}. Namely, we let $(\pi_{\alpha}^\pm)_>$ (resp.~$\pi_>$) be the expanded version of $\pi_{\alpha}^\pm$ (resp.~$\pi$) constructed there. Now Step 1 from the same proof shows:
\begin{itemize}
\item $\theta_{-\alpha}^\pm(\pi_>)=(\pi_{\alpha}^\pm)_> \neq 0$.
\end{itemize}
\begin{rem}
Thus, we may also denote this representation by $(\pi_>)^\pm_\alpha$: Step 1 shows that the representation $(\pi_{\alpha}^\pm)_>$ obtained by stretching the parameter $\psi_\alpha$ is the same as the representation $(\pi_>)^\pm_\alpha$ obtained by applying the Recipe to $\pi_>$. 
\end{rem}
Then, because $\pi_{\alpha-2}^-=0$, Step 2 shows
\begin{itemize}
\item $(\pi_{\alpha-2}^-)_> =0$ \ \ (in view of the Remark above, this is the same as saying  $(\pi_>)^-_{\alpha-2} =0$).
\end{itemize}
Step 2 also shows that we can prove b) $\pi_{\alpha-2}^+ \neq 0$ by proving the following Lemma:
\begin{lem}
\label{lem_Bb}
$(\pi_{\alpha-2}^+)_>\neq 0$.
\end{lem}
\begin{proof}
The idea is to use Lemma \ref{lem_high_blocks_dont_matter}. 
By construction, for any block $(A,B,\zeta) \in \Jord(\psi_{\alpha,>}^\pm)$ with $B > \frac{\alpha-1}{2}$ we have $B \gg \frac{\alpha-1}{2}$ (in other words, the blocks above $\alpha$ are far away). By Lemma \ref{lem_high_blocks_dont_matter}, the question of non-vanishing of $(\pi_{\alpha-2}^\pm)_>$ is, roughly speaking, independent of these far-away blocks. We use this observation to alter the blocks of $\psi_>$:
if $(A,B,\zeta)$ has $B \gg \frac{\alpha-1}{2}$, replace it by $(A,B,+)$ (we may keep the same $\eta$ and $t$ on this block). Call the resulting parameter $\psi'_>$ and let $\pi'_>$ be the corresponding representation. Now, Remark \ref{rem_descent} shows that we can apply the Recipe to get $(\pi'_>)_\alpha^\pm \neq 0$ (the added summand can always skip over blocks which have $\zeta = +1$). 

By construction, $\psi_>'$ contains no blocks $(A,B,\zeta)$ with $B > \frac{\alpha-1}{2}$ and $\zeta = -1$. By Lemma \ref{lem_Jac=0}, this implies $\Jac_{-x}(\pi'_>)_{\alpha}^\pm \allowbreak= 0$ for any $x > \frac{\alpha-1}{2}$. Now Lemma \ref{lem_standardni_mod_iznad} shows that there are no segments $[x,y]$ in the standard module of $(\pi'_>)_{\alpha}^\pm$ with $x > \frac{\alpha-1}{2}$.

Furthermore, we have already shown $(\pi_>)_{\alpha-2}^- = 0$ (see the second bullet point above). By Lemma \ref{lem_high_blocks_dont_matter}, this implies that $(\pi'_>)_{\alpha-2}^\epsilon= 0$, where the sign $\epsilon \in \{\pm 1\}$ is chosen so that on this tower, $\eta(\chi_W\otimes S_1 \otimes S_\alpha)$ is the same in $(\psi_>')_\alpha^\epsilon$ and $(\psi_>)_\alpha^-$.

But $(\pi'_>)_{\alpha-2}^\epsilon = \Jac_{-\frac{\alpha-1}{2}}(\pi'_>)_{\alpha}^\epsilon$, so we conclude that this Jacquet module vanishes. In addition to the above claim about the standard module, this shows that there are no segments $[x,y]$ in the standard module of $(\pi'_>)_{\alpha}^\epsilon$ with $x = \frac{\alpha-1}{2}$. Specifically, the singleton $\nu^{\frac{\alpha-1}{2}}$ does not appear. By Corollary \ref{cor_first_lift} b), this means that $(\pi'_>)_{\alpha}^\epsilon$ is the first lift of $\pi'$ on $\mathcal{V}^\epsilon$. In particular, $\mathcal{V}^\epsilon$ is the going-up tower for $\pi'_>$.

But then $\mathcal{V}^{-\epsilon}$ is the going-down tower for $\pi'_>$. By Theorem \ref{downlifts} the standard module of $(\pi'_>)_{\alpha}^{-\epsilon} = \theta_{-\alpha}^{-\epsilon}(\pi'_>) $ contains the singleton segment corresponding to $\nu^{\frac{\alpha-1}{2}}$. Since there are no segments $[x,y]$ with $x>\frac{\alpha -1}{2}$, this implies $\Jac_{-\frac{\alpha-1}{2}}(\pi'_>)_{\alpha}^{-\epsilon} \neq 0$. Therefore $(\pi'_>)_{\alpha-2}^{-\epsilon} \neq 0$. By Lemma \ref{lem_high_blocks_dont_matter} again (comparing the lifts of $\pi_>$ on $\mathcal{V}^+$ and $\pi_>'$ on $\mathcal{V}^{-\epsilon}$), this implies $(\pi_>)_{\alpha-2}^{+} \neq 0$, which we needed to show.
\end{proof}
This takes care of b); it remains to prove a). We prove the following Lemma.
\begin{lem}
\label{lem_Ba}
$\theta_{2-\alpha}^-(\pi_>)=0$.
\end{lem}
\begin{proof}
Lemma \ref{lem_Bb} shows that we have the following situation:
\[
(\pi_>)^\pm_\alpha \neq 0, \quad (\pi_>)^-_{\alpha-2} = 0,\quad \text{and} \quad (\pi_>)^+_{\alpha-2} \neq 0.
\]
Let $\xi$ be an irreducible representation and let $z > 0$. It follows from Lemma \ref{lem_standardni_mod_iznad} that the question of whether or not $\Jac_{-z}(\xi)$ vanishes is determined solely by the standard module of $\xi$ above $z$ (denoted $\Sigma_{\geq z}$ in \S \ref{subs_stdmod}). If neither $\theta_{-\alpha}^+(\pi_>)=(\pi_>)^+_{\alpha}$ nor  $\theta_{-\alpha}^-(\pi_>)=(\pi_>)^-_{\alpha}$ are the first lifts on their respective towers, then Corollary \ref{cor_first_lift} c) shows that their standard modules are identical above $\frac{\alpha-1}{2}$.
This would imply that $\Jac_{-\frac{\alpha-1}{2}}(\pi_>)^+_{\alpha}$ and $\Jac_{-\frac{\alpha-1}{2}}(\pi_>)^-_{\alpha}$ are either both zero, or both non-zero. But we have  $\Jac_{-\frac{\alpha-1}{2}}(\pi_>)^-_{\alpha} = (\pi_>)^-_{\alpha-2}=0$ and $\Jac_{-\frac{\alpha-1}{2}}(\pi_>)^+_{\alpha} = (\pi_>)^+_{\alpha-2}\neq 0$. This shows that one of the lifts $(\pi_>)^\pm_\alpha$ is the first occurrence on the respective tower. Since we know that $\theta_{2-\alpha}^+(\pi_>)\neq 0$, it follows that $\theta_{-\alpha}^-(\pi_>)$ is the first occurrence of $\pi_>$ on $\mathcal{V}^-$, which we needed to show.
\end{proof}

We have thus shown $\theta_{2-\alpha}^-(\pi_>)=0$, and it remains to show that this implies $\theta_{2-\alpha}^-(\pi)=0$.
Recall the construction of $\pi_>$ from \S \ref{sec_thmA}. We have
\[
\pi_{>} \hookrightarrow L_{k+1} \times \dotsb \times L_r\rtimes \pi
\]
where none of the (segments corresponding to) $L_i$'s contain $-\frac{\alpha-3}{2}$. Now Lemma \ref{lem_Kudla_nonvanishing} (ii) shows that $\theta_{2-\alpha}^-(\pi_>)=0$ implies $\theta_{2-\alpha}^-(\pi)=0$. This concludes the proof of Theorem \ref{theoremB}.

\section{Theorem C}
\label{sec_thmC}
Let $\pi$ be an irreducible representation pa\-ra\-me\-trized by
\[
\psi = \bigoplus_{i=1}^r\overset{t_i,\eta_i}{\chi_V\otimes S_{a_i} \otimes S_{b_i}}.
\]
In this section we prove Theorem \ref{theoremC}: for $\alpha < d^{\text{down}}(\pi,\psi)$, $\theta_{-\alpha}(\pi)$ is never in the A-A packet. The proof is long, but the main idea is simple. To avoid  obscuring it by technical details, we begin by outlining our strategy:
\subsection{Proof outline}
\label{subs_outline}
Consider a special case: suppose $a_r = 1$, so that $\pi$ is parametrized by 
\[
\psi = \bigoplus_{i=1}^{r-1}\chi_V\otimes S_{a_i} \otimes S_{b_i} \,\oplus \, \chi_V\otimes S_{1} \otimes S_{b_r}.
\]
Assume that $\chi_V\otimes S_{1} \otimes S_{b_r}$ is above all the other summands: $b_r > a_i+b_i-1$, for all $i < r$. Under these assumptions, Corollary \ref{cor_removal} shows that $\pi$ (or $\pi \otimes \det$) is the lift $\theta_{-b_r}(\pi')$ of a representation $\pi'$ parametrized by 
\[
\psi' = \bigoplus_{i=1}^{r-1}\chi_W\otimes S_{a_i} \otimes S_{b_i}.
\]
In other words, we are able to \emph{remove the highest term} of $\psi$ by computing a theta lift: $\pi' = \theta_{b_r}(\pi)$. This is the key idea in our proof.

\bigskip

\noindent We explain the strategy. Assume that Theorem \ref{theoremC} fails for $\pi$: there exists an $\alpha < d^{\text{down}}(\pi,\psi)$ such that $\theta_{-\alpha}(\pi)$ is in the A-A packet. Note that $d^{\text{up}}(\pi,\psi) =  b_r+2$, so $d^{\text{down}}(\pi,\psi) \leq  b_r$ (see Remark \ref{rem_descent}). We have $\pi = \theta_{-b_r}(\pi')$, and we distinguish two cases, depending on whether this is a going-up or a going-down lift:
\begin{itemize}
\item[1.] $\pi = \theta_{-b_r}^{\text{down}}(\pi')$. We know that
\[
\theta_{-\alpha}(\pi) = \theta_{-\alpha}(\theta_{-b_r}(\pi')) 
\]
is in the A-A packet, parametrized by
\[
\psi_\alpha = \bigoplus_{i=1}^{r-1}\chi_W\otimes S_{a_i} \otimes S_{b_i} \,\oplus \, \chi_W\otimes S_{1} \otimes S_{\alpha} \,\oplus \, \chi_W\otimes S_{1} \otimes S_{b_r}.
\]
The highest summand in $\psi_\alpha$ is still $\chi_W\otimes S_{1} \otimes S_{b_r}$, and we use the same idea as before: we can remove it by computing $\theta_{b_r}$. The key observation is given in Corollary \ref{downlifts_commute}, which says that the going-down lifts commute: $  \theta_{-\alpha}(\theta_{-b_r}(\pi')) = \theta_{-b_r}(\theta_{-\alpha}(\pi')) $. 
Applying $\theta_{b_r}$ to this representation (and using Convention \ref{conv_dettwist}), we conclude that the representation
\[
\theta_{b_r}( \theta_{-b_r}(\theta_{-\alpha}(\pi'))) = \theta_{-\alpha}(\pi')
\]
is parametrized by
\[
\psi'_\alpha = \bigoplus_{i=1}^{r-1}\chi_V\otimes S_{a_i} \otimes S_{b_i} \,\oplus \, \chi_V\otimes S_{1} \otimes S_{\alpha}.
\]
But this is precisely the A-A packet for $\theta_{-\alpha}(\pi')$! 
Furthermore, it is not hard to show that in this case, $d^{\text{down}}(\pi,\psi) = d^{\text{down}}(\pi',\psi')$. In other words, assuming that Theorem \ref{theoremC} fails for $\pi$, we conclude that it must also fail for $\pi'$.

This allows us to proceed inductively, and simplify the parameter. Ultimately, one arrives at the following case, in which it is easy to obtain a contradiction.

\item[2.] $\pi = \theta_{-b_r}^{\text{up}}(\pi')$. Again, we are assuming that $\theta_{-\alpha}^{\text{down}}(\pi)= \theta_{-\alpha}^{\text{down}}(\theta_{-b_r}^{\text{up}}(\pi'))$ is in the A-A packet, parametrized by
\[
\psi_\alpha = \bigoplus_{i=1}^{r-1}\chi_W\otimes S_{a_i} \otimes S_{b_i} \,\oplus \, \chi_W\otimes S_{1} \otimes S_{\alpha} \,\oplus \, \chi_W\otimes S_{1} \otimes S_{b_r}.
\]
Once more, because the highest summand is $\chi_W\otimes S_{1} \otimes S_{b_r}$ we can compute the theta lift $\theta_{b_r}$ of this representation. In particular,
\[
\theta_{b_r}(\theta_{-\alpha}^{\text{down}}(\theta_{-b_r}^{\text{up}}(\pi'))) \neq 0.
\]
However, Corollary \ref{key} shows that this is impossible, so we reach a contradiction.
\end{itemize}

\noindent In the rest of Section \S \ref{sec_thmC}, we provide all the technical details needed to turn the above strategy into a proof. The main difficulty is reducing the proof to the case where the highest term of the parameter is $\chi_V\otimes S_{1} \otimes S_{b_r}$, as assumed above. We do this in two stages:
\begin{itemize}
\item First, we show that we may assume $b_r > a_r$ in the highest summand. This is explained in \S \ref{subs_stretch}.
\item Then, we explain how to replace this summand with one that has $a_r=1$. This is more difficult. We split the proof into two subsections: \S \ref{subs_removal} (where we set up the main idea) and \S \ref{subs_replacement} (where we provide the technical details).
\end{itemize}
Finally, in \S \ref{subs_removing}, we finish the proof; this resembles the strategy outlined above.

\bigskip
\noindent Before going into the remainder of Section \ref{sec_thmC}, we advise the reader to skim over Appendix \ref{appendix}, where we have summarized the results of \cite{nas_clanak} that are used in this proof.
To simplify the notation, we let $d=d^{\text{down}}(\pi,\psi)$ throughout \S \ref{sec_thmC}. We also assume $d > 1$, because Theorem \ref{theoremC} is trivial otherwise.

\subsection{Stretching the parameter}
\label{subs_stretch}
We begin by taking care of the first bullet of the above outline.
\begin{lem}
\label{lem_D_reduction1}
It suffices to prove Theorem \ref{theoremC} for $\psi$ such that
\begin{itemize}
    \item for all $(A,B,\zeta) \in \Jord(\psi)$
    \[
   \zeta=-1 \text{ and } 2A+1\geq d  \quad \Rightarrow \quad 2B+1 \gg d;
    \]
    \item the blocks $(A,B,-)$ with $2B+1 > d$ are far apart from each other.
\end{itemize}
In plain terms, we may assume that all the $\zeta=-1$ segments above $d$ are far above $d$, and far away from each other. 
\end{lem}
\begin{rem}
\label{rem_spacing}
The above notion of “far away” can be made more precise, cf.\,\S 2 in \cite{xu2021combinatorial}.
\end{rem}
\begin{proof}[Proof of Lemma \ref{lem_D_reduction1}]
Suppose that $\theta_{-\alpha}(\pi)$ is isomorphic to a representation $\pi^{\alpha}$ in the A-A packet for some $\alpha < d$. We use the notation $\pi^{\alpha}$ (as opposed to $\pi_\alpha$) because we want to emphasize that we are not assuming that this is the representation suggested by the Recipe.  To simplify notation, let $(A_0,B_0,\zeta_0)$ denote the added block $(\frac{\alpha-1}{2},\frac{\alpha-1}{2},-)$. We fix an admissible order $>$ on $\Jord(\psi_\alpha)=\{(A_i,B_i,\zeta_i): i=0,\dotsc,r\}$. The only requirement is that there exist an index $k\in \{0,\dotsc,r\}$ such that $i>k$ if and only if
\begin{itemize}
    \item $\zeta_i=-1$ and $2A_i+1 \geq d$.
\end{itemize}
This order restricts to an order on $\Jord(\psi)$. We now form parameters $\psi_{>}$ and $(\psi_{\alpha})_>$ by shifting up all the blocks $(A_i,B_i,\zeta_i)$ for $i>k$. (If we shift high enough, the blocks of $\Jord(\psi_>)$ will satisfy the conditions of the Lemma.) Let $\pi_{>}$ and $(\pi^{\alpha})_>$ denote the corresponding representations. We then have
\begin{align}
    \pi_{>} &\hookrightarrow L_{k+1}\times \dotsb \times L_r \rtimes \pi \label{eq_D_prva}\\
   (\pi^{\alpha})_> &\hookrightarrow L_{k+1}\times \dotsb \times L_r \rtimes \pi^{\alpha}. \label{eq_D_druga}
\end{align}
By Proposition \ref{prop_wwtawwtaKudla} a), $\theta_{-\alpha}(\pi_>)\neq 0$. By part c) of the same proposition, we get
\[
\theta_{-\alpha}(\pi_{>}) \hookrightarrow L_{k+1}\times \dotsb \times L_r \rtimes \theta_{-\alpha}(\pi) \cong L_{k+1}\times \dotsb \times L_r \rtimes \pi^{\alpha}.
\]
But \eqref{eq_D_druga} shows that the right-hand side has a unique irreducible representation, isomorphic to $(\pi^{\alpha})_>$ (see Lemma \ref{lem_L}, (ii)). 
Thus $(\pi^{\alpha})_> = \theta_{-\alpha}(\pi_{>})$.

In short, we have shown that $\theta_{-\alpha}(\pi)=\pi^{\alpha}$ implies $\theta_{-\alpha}(\pi_>)=(\pi^{\alpha})_>$; in particular, $\theta_{-\alpha}(\pi_>)$ is in the A-A packet. It remains to show  the following (we postpone the proof):
\begin{lem}
\label{lem_D_reduction1aux}
$d^{\text{down}}(\pi_>,\psi_>)=d$.
\end{lem}
\noindent The key point is that $\pi_{>}$ then satisfies the assumption from the statement of Lemma \ref{lem_D_reduction1}.

Now assume that we have proven Theorem \ref{theoremC} for $\pi_>$. Suppose that $\theta_{-\alpha}(\pi)=\pi^{\alpha}$ is in the A-A packet for some $\alpha < d$. The above discussion shows that this would imply $\theta_{-\alpha}(\pi_>)=(\pi^{\alpha})_>$. But this is impossible if Theorem \ref{theoremC} is valid for $\pi_>$. We thus get Theorem \ref{theoremC} for $\pi$ as well.
\end{proof}
\begin{proof}[Proof of Lemma \ref{lem_D_reduction1aux}]
Recall that $d^{\text{down}}(\pi_>,\psi_>)$ is defined by the Recipe in \S \ref{sec_conj}: we know $(\pi_>)_{\alpha}\neq 0$ for $\alpha \gg 0$, and we descend through the parameter ($\alpha \mapsto \alpha-2$) until we get $(\pi_>)_{\alpha} = 0$. 

We begin by noticing that $(\pi_>)_{\alpha} = 0$ cannot happen until after $\alpha$ has become lower than all the blocks above $d$ which have $\zeta = -1$ (i.e.\ the blocks discussed in Lemma \ref{lem_D_reduction1}). Indeed, since the blocks of $\psi_{>}$ above $d$ are far apart, the only way the Recipe could result in $(\pi_>)_{\alpha} = 0$ is if the block $1 \otimes S_1 \otimes S_\alpha$ were to encounter a block $(A+T, B+T, -)\in \Jord(\psi_>)$ with $t=0$ and incompatible $\eta$ (see Remark \ref{rem_descent}). In this case, $(\pi_>)_{\alpha}=0$ for $\alpha = 2(A+T)+1$. But the same block\,---\,shifted back\,---\,appears in the original parameter $\psi$: we have $(A,B,-)\in \Jord(\psi)$. This would imply that $\pi_\alpha = 0$ for $\alpha = 2A+1$. However, this is impossible, since (by construction) $2A+1\geq d$; recall that  $\pi_\alpha \neq 0$ whenever $\alpha \geq d$.

The above discussion shows that we can lower $1 \otimes S_1 \otimes S_\alpha$ until it becomes the lowest block (with respect to the order on $\Jord(\psi_>)$) with $\zeta = -1$ above $d$. It now follows that we can lower it all the way to $d$ to get $(\pi_>)_{\alpha}\neq 0$ for all $\alpha \geq d$. Indeed, if $1 \otimes S_1 \otimes S_\alpha$ is the lowest block with $\zeta=-1$ above $d$, then we can view $(\psi_>)_{\alpha}$ as one of the intermediate parameters between $\psi_{d}$ and the DDR parameter which dominates $\psi_{d}$ (see \S \ref{subs_Arthurtype}; of course, we really mean $\rho$-DDR representation with $\rho = \chi_V$ or $\chi_W$). Thus, by M\oe glin's construction, $\pi_{d}$ is obtained by computing the relevant Jacquet modules of $(\pi_>)_{\alpha}$; this shifts the blocks of $(\psi_>)_\alpha$ back to their original places in $\psi_d$. Therefore $\pi_{d}\neq 0$ means
\[
0 \neq \pi_d = \Jac_{L_{k+1},\dotsc,L_r}\Jac_{-\frac{\alpha-1}{2}, -\frac{\alpha-3}{2},\dots,-\frac{d+1}{2}}((\pi_>)_{\alpha}).
\]
In particular, 
\[
(\pi_>)_{d}=\Jac_{-\frac{\alpha-1}{2}, -\frac{\alpha-3}{2},\dots,-\frac{d+1}{2}}((\pi_>)_{\alpha})\neq 0.
\]
This shows that $(\pi_>)_{d} \neq 0$, so we conclude $d^{\text{down}}(\pi_>,\psi_>) \leq d$.

We now prove the reverse inequality: $d^{\text{down}}(\pi_>,\psi_>) \geq d$. Assume the contrary, i.e.~that $d^{\text{down}}(\pi_>,\psi_>) < d$. We prove that this implies $\pi_{d-2}\neq 0$, thus leading to a contradiction\,---\,recall that $\pi_{d-2}=0$ by definition of $d=d(\pi,\psi)$. Assuming $d^{\text{down}}(\pi_>,\psi_>) < d$, and in particular, $(\pi_>)_{d-2}\neq 0$, Theorem \ref{theoremA} shows that $(\pi_>)_{d-2} = \theta_{2-d}(\pi_>)$. Thus, taking $\beta \gg 0$, we get (see Corollary \ref{higher_lift})
\begin{equation}
\label{eq_D_treca}
\theta_{-\beta}(\pi_>) \hookrightarrow \zeta(-\frac{\beta-1}{2}, -\frac{d-1}{2}) \rtimes (\pi_>)_{d-2}.
\end{equation}
On the other hand, applying Lemma \ref{lem_Kudla_small} to \eqref{eq_D_prva}, we get
\[
\theta_{-\beta}(\pi_{>}) \hookrightarrow L_{k+1}\times \dotsb \times L_r \rtimes \theta_{-\beta}(\pi),
\]
which shows that $\Jac_{L_{k+1},\dotsc,L_r}(\theta_{-\beta}(\pi_{>}))\neq 0$. Combining this with \eqref{eq_D_treca} we get
\[
\Jac_{L_{k+1},\dotsc,L_r}\zeta(-\frac{\beta-1}{2}, -\frac{d-1}{2}) \rtimes (\pi_>)_{d-2}\neq 0.
\]
A simple Jacquet module computation using formula \eqref{Tadic_zeta} for $M^*$ shows that this is possible only if $\Jac_{L_{k+1},\dotsc,L_r}(\pi_>)_{d-2} \neq 0$. But this means precisely $\pi_{d-2} \neq 0$. 
\end{proof}

The results proved above allow us to assume that we are in the situation described in Lemma \ref{lem_D_reduction1}: from this point on, we assume that the order on $\Jord(\psi)$ looks like
\begin{equation}
\label{eq_parameter}
\underbrace{\begin{pmatrix}
\text{blocks with}\\
A < \frac{d-1}{2}
\end{pmatrix}}_{\text{low blocks}} \quad < \quad 
\underbrace{\begin{pmatrix}
\text{blocks with}\\
\zeta=+1 \text{ and } A \geq \frac{d-1}{2}
\end{pmatrix}}_{\text{middle blocks}} \quad \ll \quad 
\underbrace{\begin{pmatrix}
\text{blocks with}\\
\zeta=-1 \text{ and } B \gg \frac{d-1}{2}
\end{pmatrix}}_{\text{high blocks}}.
\end{equation}
Furthermore, we assume that all the high blocks are far away from each other and from the middle blocks. This takes care of the first bullet in our plan. We will use the term “high blocks” throughout the rest of Section \ref{sec_thmC} to describe this part of the parameter. Before moving on, we record a lemma that will be useful later.
\begin{lem}
\label{lem_obstacle}
At least one of the high blocks has $t=0$.
\end{lem}
\begin{proof}
Recall that we are assuming $d>1$.
Suppose none of the high blocks in $\psi$ have $t=0$. By Remark \ref{rem_descent}, this means that there are no obstacles in applying the Recipe to lifts $\theta_{-\alpha}(\pi)$: on either tower, we may descend until $1 \otimes S_1 \otimes S_\alpha$ has become smaller than all the high blocks. Now we use the same trick we used in the proof of Theorem \ref{theoremB}: Lemma \ref{lem_high_blocks_dont_matter} shows that $d^{\text{up/down}}(\pi,\psi)$ does not depend on the high blocks, so we can once again replace them with blocks which have $\zeta = +1$. Call the resulting parameter $\psi'$ and let $\pi'$ be the corresponding representation. Thus, by construction, there are no blocks $(A,B,\zeta) \in \Jord(\psi')$ with $\zeta = -1$ and $B \geq \frac{d-1}{2}$.

By Lemma \ref{lem_high_blocks_dont_matter},  $d^{\text{up/down}}(\pi,\psi) = d^{\text{up/down}}(\pi',\psi')$. In particular, $d^{\text{down}}(\pi',\psi') = d$. But since $\psi'$ has no blocks $(A,B,-)$ with $B \geq \frac{d-1}{2}$, it follows from Lemma \ref{lem_Jac=0} and \ref{lem_standardni_mod_iznad} that the standard module of $\pi'$ contains no segments $[x,y]$ with $x \geq \frac{d-1}{2}$. Thus, by Theorem \ref{downlifts}, the standard module of $\theta_{-d}^{\text{down}}(\pi')=\pi'_d$ has only one such segment: the singleton segment $[\frac{d-1}{2},\frac{d-1}{2}]$. It follows that there is an irreducible representation $\xi$ such that $\nu^\frac{d-1}{2} \rtimes \xi \twoheadrightarrow \pi'_d$. But this implies $\pi_{d-2}'=\Jac_{-\frac{d-1}{2}}(\pi_{d}')\neq 0$, which is impossible by definition of $d=d(\pi',\psi')$.
\end{proof}

\begin{rem}
\label{rem_tower} The highest block with $t=0$ carries important information. Recall Remark \ref{rem_descent}: on one of the towers, the Recipe will stop working when the added block encounters the highest block with $\zeta =-1$ and $t=0$; this happens on the tower where the signs $\eta$ on these two blocks are incompatible. (On the other tower, we will be able to descend further.) By Theorem \ref{theoremB}, this level also indicates the first occurrence on the going-up tower. Therefore, under the assumptions of Lemma \ref{lem_D_reduction1}, the going-up tower is determined by the sign $\eta$ on the highest summand with $t=0$.
\end{rem}
\subsection{Surgery on Arthur packets}
\label{subs_removal}

Let $\chi_V\otimes S_{a_r} \otimes S_{b_r}$ be the highest summand in $\psi$ and, as before, let $(A_r, B_r, \zeta_r)$ denote the corresponding block in $\Jord(\psi)$. In outlining the proof (\S \ref{subs_outline}), we assumed that the highest summand was of the form $ \chi_V\otimes S_{1} \otimes S_{b_r}$; the corresponding segment was a singleton with $A_r = B_r = \frac{b_r-1}{2}$. This made it easy to remove it: $\pi'$ was obtained as $\theta_{b_r}(\pi)$. In general, however, we might have $A_r > B_r$, which complicates things. Reducing the proof to the case $A_r = B_r$ is the main technical issue of our approach. Let us outline the plan.

If $t(A_r,B_r,\zeta_r) = 0$, then a different parameter for $\pi$ may be obtained by replacing the big segment $[B_r,A_r]$ with the corresponding sequence of singleton segments. In other words, we replace $\chi_V \otimes S_{a_r} \otimes S_{b_r}$ with 
\[
\chi_V \otimes S_{1} \otimes S_{b_r-a_r+1}, \quad \chi_V \otimes S_{1} \otimes S_{b_r-a_r+3},\quad  \dots \quad \chi_V \otimes S_{1} \otimes S_{a_r+b_r-1}.
\]
The signs on these summands will alternate. Now, we are once again in a situation where the highest term is a singleton. Thus we may remove these summands\,---\,effectively removing $(A_r,B_r,\zeta_r)$\,---\,by computing successive theta lifts,
\[
\theta_{b_r-a_r+1}(\dots \theta_{a_r+b_r-3}(\theta_{a_r+b_r-1}(\pi))\dots).
\]

If $t(A_r,B_r,\zeta_r) > 0$, we may fall back to the above case of $t=0$ by computing a few additional lifts. Indeed, if 
\[
\psi =  \bigoplus_{i=1}^{r-1}{\chi_V\otimes S_{a_i} \otimes S_{b_i}} \, \oplus \,   \overset{t=t_r}{\chi_V\otimes S_{a_r} \otimes S_{b_r}},
\]
with $t_r > 0$, then we can compute $\theta_{-\gamma}(\pi)$ for some $\gamma$ such that $\frac{\gamma-1}{2} < B_r$ and $\frac{\gamma-1}{2} > A_i$ for $i<r$. (Recall that we can choose such $\gamma$ because we are assuming $\psi$ satisfies the conditions of Lemma \ref{lem_D_reduction1}; in particular, $B_r \gg A_i$ for $i<r$.) By Theorem \ref{theoremA} and Lemma \ref{lem_xu}, on one of the towers this lift will be parametrized by
\[
\psi_\gamma =  \bigoplus_{i=1}^{r-1}{\chi_W\otimes S_{a_i} \otimes S_{b_i}} \, \oplus \,  \chi_W\otimes S_{1} \otimes S_{\gamma} \, \oplus \,   \overset{t=t_r-1}{\chi_W\otimes S_{a_r} \otimes S_{b_r}}.
\]
Thus the extra lift allows us to go from $t_r$ to $t_r-1$. Doing this $t_r$ times will give us a representation whose highest summand has $t=0$. 

We elaborate on this, as we will need to be precise. Suppose $\pi$ is parametrized by 
\[
\psi = \text{(lower terms)}\oplus  \overset{\eta_0,\ t=t_r}{\chi_V\otimes S_{a_r} \otimes S_{b_r}}.
\]
Again, we assume $B_r \gg 0$. By Proposition \ref{prop_visoko}, the lift $\theta_{-\gamma}(\pi)$ for $\gamma > A_r$ is given by
\[
\psi_\gamma = \text{(lower terms) }  \oplus \, \overset{-\eta_1,\ t=t_r}{\chi_W\otimes S_{a_r} \otimes S_{b_r}} \, \oplus \, \overset{\eta'}{\chi_W\otimes S_{1} \otimes S_{\gamma}},
\]
where the signs $\eta_1$ and  $\eta'$ are determined using Proposition \ref{prop_visoko} and Remark \ref{rem_ort_lifts}. Note that the choice of the target tower determines the ratio $\eta'/\eta_1$. Indeed, this follows from Proposition \ref{prop_visoko} if $\pi$ is a representation of a symplectic group, and from Remark \ref{rem_ort_lifts} in case of orthogonal groups. We choose the tower for which $\eta' = (-1)^{A_r-B_r}\eta_1$. Let us see what happens when $\frac{\gamma-1}{2}$ becomes smaller than $B_r$ (but still bigger than $A_i$ for $i<r$). Then, following the Recipe and using Lemma \ref{lem_xu} to change the order of the summands, we see that $\theta_{-\gamma}(\pi)$ is parametrized by
\[
\psi_\gamma = \text{(lower terms) }  \oplus \, \overset{\eta_1}{\chi_W\otimes S_{1} \otimes S_{\gamma}} \, \oplus \, \overset{\eta_1,\ t=t_r-1}{\chi_W\otimes S_{a_r} \otimes S_{b_r}}.
\]
If we do this again, and compute $\theta_{-(\gamma+2)}(\theta_{-\gamma}(\pi))$, we get the parameter
\[
(\psi_\gamma)_{(\gamma+2)} = \text{(lower terms) }  \oplus \, \overset{-\eta_2}{\chi_V\otimes S_{1} \otimes S_{\gamma}} \, \oplus \,  \overset{\eta_2}{\chi_V\otimes S_{1} \otimes S_{\gamma+2}} \, \oplus \, \overset{\eta_2,\ t=t_r-2}{\chi_V\otimes S_{a_r} \otimes S_{b_r}}.
\]
Here $\eta_2$ is a sign that is determined by $\eta_1$. However, it is not important for us to determine it explicitly; the only thing that matters is how the signs on these last three summands relate to each other.
If we continue the process and perform a total of $t_r$ lifts, the resulting parameter is
\[
\psi_\uparrow = \text{(l.t.)}  \oplus \, 
\overset{(-1)^{t_r-1}\eta_r}{\chi\otimes S_{1} \otimes S_{\gamma}}
 \, \oplus \,  \overset{(-1)^{t_r-2}\eta_r}{\chi\otimes S_{1} \otimes S_{\gamma+2}}
 \, \oplus \, \dots  \, \oplus \, \overset{\eta_r}{\chi\otimes S_{1} \otimes S_{\gamma+2(t_r-1)}} \, \oplus 
 \overset{\eta_r,\ t=0}{\chi\otimes S_{a_r} \otimes S_{b_r}}.
\]
Here $\chi$ denotes $\chi_V\cdot(\chi_V\chi_W)^{t_r}$. The assumptions of Lemma \ref{lem_D_reduction1} (the high blocks in $\psi$ being far apart) ensure that we can choose $\gamma$ so that these added terms do not overlap with any of the existing segments in $\psi$.
\begin{rem}
\label{rem_uparrow}
The following observations are important:
\begin{enumerate}[(i)]
\item The signs of the added terms alternate. This means that each subsequent lift is the (first) going-up lift of the previous representation (cf.~Remark \ref{rem_tower}). Thus $\psi_\uparrow$ parametrizes the representation
\[
\theta_{-(\gamma+2(t-1))}^{\text{up}}(\theta_{-(\gamma+2(t-2))}^{\text{up}}(\dots \theta_{-(\gamma+2)}^{\text{up}}(\theta_{-\gamma}(\pi))\dots)).
\]
Note that the target tower for $\theta_{-\gamma}(\pi)$ is determined by $\eta$, but this doesn't tell us whether it is the going-up or the going-down lift (this turns out to be unimportant).
\item Writing the above representation is cumbersome, so we shorten the notation to $\theta_{\uparrow}(\pi)$.
\item The highest summand in $\psi_\uparrow$ now has $t=0$, so we can parametrize the same representation by replacing the highest segment in $\Jord(\psi_\uparrow)$ with the corresponding collection of singleton segments with alternating signs. Thus $\theta_{\uparrow}(\pi)$ is also parametrized by
\begin{equation}
\label{eq_psi_up}
\begin{aligned}
\psi_\uparrow' &= \text{(l.t.) }  \oplus \, 
\overset{(-1)^{t_r-1}\eta_r}{\chi\otimes S_{1} \otimes S_{\gamma}}
 \, \oplus \,  \overset{(-1)^{t_r-2}\eta_r}{\chi\otimes S_{1} \otimes S_{\gamma+2}}
 \, \oplus \, \dots  \, \oplus \, \overset{\eta_r}{\chi\otimes S_{1} \otimes S_{\gamma+2(t_r-1)}}  \\
& \oplus \overset{\eta_r}{\chi\otimes S_{1} \otimes S_{b_r-a_r+1}}  \, \oplus \, \overset{-\eta_r}{\chi\otimes S_{1} \otimes S_{b_r-a_r+3}}  \, \oplus \, \dots \, \oplus\, \overset{(-1)^{A_r-B_r}\eta_r}{\chi\otimes S_{1} \otimes S_{b_r+a_r-1}}.
\end{aligned}
\end{equation}
\item The above parametrization gives us crucial information about the representation $\theta_{\uparrow}(\pi)$. Namely, Corollary \ref{cor_removal} shows that this representation (or its twist by $\det$) can be obtained as a series of theta lifts of a smaller representation $\pi^0$ whose parameter corresponds to the “lower terms” in $\psi$: in other words, the parameter of $\pi^0$ is obtained by removing the largest term from $\psi$. To be precise, we need the following lifts to get from $\pi^0$ to $\theta_\uparrow(\pi)$:
\begin{enumerate}[1.]
\item Compute $\theta_{-\gamma}(\pi^0)$; this adds $\chi \otimes S_1 \otimes S_\gamma$ to the parameter.
\item Compute the \emph{first} going-up lift $t_r-1$ times. This will add the remaining terms in the first row of \eqref{eq_psi_up}. Call the resulting representation $\pi^1$.
\item Now, compute the \emph{going-down} lift $\theta_{-(b_r-a_r+1)}(\pi^1)$. This adds the first summand in the second row of \eqref{eq_psi_up}. We know that we are supposed to take the going-down lift because the signs on  $\chi\otimes S_{1} \otimes S_{\gamma+2(t-1)}$ and $\chi\otimes S_{1} \otimes S_{b_r-a_r+1}$ are compatible.
\item The remaining signs alternate, so it remains to compute the first going-up lift another $A_r-B_r$ times.
\end{enumerate}
To summarize: we can remove the highest term of the parameter using a specific sequence of lifts.
%
%
\end{enumerate}
\end{rem}
In constructing $\theta_\uparrow(\pi)$, the target tower for the initial lift $\theta_{-\gamma}$ and the number of subsequent going-up lifts are determined by the values of $\eta$ and $t$ on the highest block of $\psi$. This is the content of the following lemma (we set $a=a_r$ and $b=b_r$ to simplify notation):
\begin{lem}
\label{lem_76v}
Let $\psi$ and $\gamma$ be as above. Let $k \in \{0,1,\dotsc,\lfloor \min(a,b)/2\rfloor-1\}$ and let $* \in \{\text{up},\text{down}\}$. Then the representation
\[
\theta_{(k,*)}(\pi) := \theta_{-(\gamma+2k)}^{\text{up}}(\theta_{-(\gamma+2k-2)}^{\text{up}}(\dots \theta_{-(\gamma+2)}^{\text{up}}(\theta_{-\gamma}^{*}(\pi))\dots))
\]
is either $0$, or is in the expected Arthur packet parametrized by
\[
\psi_{(k,*)} = (\textnormal{l.t.})\ \oplus\ \chi \otimes S_1 \otimes S_\gamma\ \oplus\ \chi \otimes S_1 \otimes S_{\gamma+2}\ \oplus\  \dots\ \oplus\  \chi \otimes S_1 \otimes S_{\gamma+2k}\ \oplus\  \chi \otimes S_{a} \otimes S_{b}.
\]
If $t < \min(a,b)/2$, then there exists only one pair $(k,*)$ for which the parameter corresponding to $\theta_{(k,*)}(\pi)$ has $t(a,b)=0$. 

If $t = \min(a,b)/2$, then there are two possibilities: $k=t-1$ and $* \in \{\text{up},\text{down}\}$.
\end{lem}
\begin{proof} It is worthwhile to look at examples \ref{ex_prvi}--\ref{ex_zadnji}, which illustrate the Lemma. The proof is a matter of careful bookkeeping and applying the change of order formulas of Lemma \ref{lem_xu}. We first assume $a$ and $b$ are odd. The case when $a$ and $b$ are even is similar, and we address it later. (We are assuming $\psi$ is of good parity, which implies $a$ and $b$ have the same parity).

Taking $\alpha \gg 0$, Proposition 2.2 says that $\theta_{-\alpha}(\pi)$ is parametrized by
\[
(\text{lt})\ \oplus \ \overset{-\eta_1,\ t}{\chi_W \otimes S_{a} \otimes S_{b}} \ \oplus \ \overset{\eta'}{\chi_W \otimes S_{1} \otimes S_{\alpha}},
\]
where the signs $\eta_1$ and $\eta'$ are determined by Proposition 2.2 and Remark 2.3. 

We take a moment to explain this. If $\pi$ is a representation of a symplectic group, then $\eta_1=\eta$ and the choice of $\eta'$ corresponds to the choice of target tower. If $\pi$ is a representation of an orthogonal group, then $\eta'$ is determined by Equation (1), whereas $\eta_1$ reflects the choice of “target tower”, i.e.\ the choice between lifting $\pi$ and its twist, $\pi \otimes \det$. In either case, the choice of target tower controls the ratio $\eta'/\eta_1$.

We first consider the target tower on which $\eta' = (-1)^{A-B}\eta_1$. 

We begin by computing $\theta_{-\gamma}(\pi)$ on this tower. According to the Recipe, on this tower $\theta_{-\gamma}(\pi)$ is parametrized by
\[
(\text{lt})\ \oplus \ \overset{\eta_1}{\chi_W \otimes S_{1} \otimes S_{\gamma}}\ \oplus \ \overset{\eta_1,\ t-1}{\chi_W \otimes S_{a} \otimes S_{b}}.
\]
Note that we used the third bullet of Lemma \ref{lem_xu} in following the Recipe here. The next step is $\theta_{-(\gamma+2)}^\text{up}$. Following the Recipe, we see that $\theta_{-(\gamma+2)}^{\text{up}}(\theta_{-\gamma}(\pi))$ is given by
\[
\text{(lt) }  \oplus \, \overset{-\eta_2}{\chi_V\otimes S_{1} \otimes S_{\gamma}} \, \oplus \,  \overset{\eta_2}{\chi_V\otimes S_{1} \otimes S_{\gamma+2}} \, \oplus \, \overset{\eta_2,\ t-2}{\chi_V\otimes S_{a} \otimes S_{b}}.
\]
Here $\eta_2$ is a sign determined by $\eta_1$. Again, to explicate $\eta_2$, we would have to know whether $\pi$ is a representation of an orthogonal or a symplectic group. This does not concern us: all we need to know is that the signs on $\chi_V\otimes S_{1} \otimes S_{\gamma}$ and $\chi_V\otimes S_{1} \otimes S_{\gamma+2}$ differ, which is dictated by the fact that we are lifting to the going-up tower. The fact that the signs on $\chi_V\otimes S_{1} \otimes S_{\gamma+2}$ and $\chi_V\otimes S_{a} \otimes S_{b}$ are the same then follows automatically.

We may continue in this fashion. Each subsequent going-up lift will decrease the value of $t(a,b)$, as we are using the \textbf{third bullet} of Lemma \ref{lem_xu} in each step. After a total of $t$ lifts, we arrive at $\theta_{(t-1,*)}(\pi)$, parametrized by
\[
\text{(l.t.)}  \oplus \, 
\overset{(-1)^{t-1}\eta_r}{\chi\otimes S_{1} \otimes S_{\gamma}}
 \, \oplus \,  \overset{(-1)^{t-2}\eta_r}{\chi\otimes S_{1} \otimes S_{\gamma+2}}
 \, \oplus \, \dots  \, \oplus \, \overset{\eta_r}{\chi\otimes S_{1} \otimes S_{\gamma+2(t-1)}} \, \oplus 
\overset{\eta_r,\ 0}{\chi\otimes S_{a} \otimes S_{b}}.
\]
Here, once again, the precise value of $\eta_r$ depends on whether $\pi$ is a representation of an orthogonal or a symplectic group. To summarize, on this tower, $\theta_{(t-1,*)}(\pi)$ has $t(a,b)=0$.

After this, the next lift\,---\,that is, $\theta_{-(\gamma+2t)}^\text{up}$\,---\,vanishes. Indeed, by Remark \ref{rem_tower}, we have $\theta_{-(a+b-1)}^{\text{up}}(\theta_{(t-1,*)}(\pi)) \allowbreak = 0$. We are assuming $\frac{\gamma-1}{2} \ll B$ (and thus $\gamma+2t \ll a+b-1$), so $\theta_{-(\gamma+2t)}^\text{up}(\theta_{(t-1,*)}(\pi))=0$ as well. In short, on this tower, $\theta_{(k,*)}(\pi)=0$ when $k\geq t$.

It remains to see what happens on the other tower, that is, the tower for which $\eta' = -(-1)^{A-B}\eta_1$. If $t < \lfloor\min(a,b)/2\rfloor$, then the \textbf{first bullet} of Lemma \ref{lem_xu} shows that on this tower $\theta_{-\gamma}(\pi)$ is given by
\[
(\text{l.t.})\ \oplus \ \overset{-\eta_1}{\chi_W \otimes S_{1} \otimes S_{\gamma}}\ \oplus \ \overset{\eta_1,\ t+1}{\chi_W \otimes S_{a} \otimes S_{b}}.
\]
Assuming $t+1 < \lfloor\min(a,b)/2\rfloor$, the next lift is then parametrized by
\[
(\text{l.t.})\ \oplus \ \overset{\eta_2}{\chi_V \otimes S_{1} \otimes S_{\gamma}}\ \oplus \ \overset{-\eta_2}{\chi_V \otimes S_{1} \otimes S_{\gamma+2}}\ \oplus \ \overset{\eta_2,\ t+2}{\chi_V \otimes S_{a} \otimes S_{b}}.
\]
After exactly $s:=\lfloor\min(a,b)/2\rfloor-t$ lifts (using the first bullet of Lemma \ref{lem_xu} each time), we get
\[
(\text{l.t.})\ \oplus \ \overset{(-1)^s\eta_s}{\chi \otimes S_{1} \otimes S_{\gamma}}\ \oplus \  \dotsb \ \oplus \ \overset{-\eta_s}{\chi \otimes S_{1} \otimes S_{\gamma+2(s-1)}}\ \oplus \ \overset{\eta_s,\ \lfloor\min(a,b)/2\rfloor}{\chi \otimes S_{a} \otimes S_{b}}.
\]
Crucially, after this step, in computing the next lift we use the \textbf{second bullet} of Lemma \ref{lem_xu} (as opposed to the first bullet). This means that $t(a,b)$ does not change. To be precise, the next lift is $\theta_{-(\gamma+2s)}^\text{up}$, and we get
\[
(\text{l.t.})\ \oplus \ \overset{(-1)^{s}\eta_{s+1}}{\chi \otimes S_{1} \otimes S_{\gamma}}\ \oplus \  \dotsb \ \oplus \ \overset{-\eta_{s+1}}{\chi \otimes S_{1} \otimes S_{\gamma+2(s-1)}}\ \oplus \ \overset{\eta_{s+1}}{\chi \otimes S_{1} \otimes S_{\gamma+2s}}\ \oplus \ \overset{\eta_{s+1},\ \lfloor\min(a,b)/2\rfloor}{\chi \otimes S_{a} \otimes S_{b}}.
\]
Observe that the signs on the last two terms are now equal. This implies that the subsequent lifts will start decreasing the value of $t(a,b)$, because we now need to use the \textbf{third bullet} of Lemma \ref{lem_xu}. Indeed, computing $\theta_{-(\gamma+2(s+1))}^\text{up}$ will give us
\begin{gather*}
(\text{l.t.})
\ \oplus \ \overset{(-1)^{s+1}\eta_{s+2}}{\chi \otimes S_{1} \otimes S_{\gamma}}
\ \oplus \  \dotsb \ \oplus \ \overset{\eta_{s+2}}{\chi \otimes S_{1} \otimes S_{\gamma+2(s-1)}}
\ \oplus \ \overset{-\eta_{s+2}}{\chi \otimes S_{1} \otimes S_{\gamma+2s}}\\
\ \oplus \ \overset{\eta_{s+2}}{\chi \otimes S_{1} \otimes S_{\gamma+2(s+1)}}
\ \oplus \ \overset{\eta_{s+2},\ \lfloor\min(a,b)/2\rfloor-1}{\chi \otimes S_{a} \otimes S_{b}}.
\end{gather*}

With each subsequent lift, $t(a,b)$ will keep decreasing (because we are using the third bullet of Lemma \ref{lem_xu} in each step) until it reaches $0$.

Let us summarize the above discussion. We have explained how the value of $t(a,b)$ changes with every step of the construction $\pi \mapsto \theta_{-\gamma}(\pi) \mapsto \theta_{(1,*)}(\pi) \mapsto \theta_{(2,*)}(\pi) \mapsto \dots$ For one choice of $*$, the values of $t(a,b)$ are
\[
t,\ t-1,\ t-2,\ \dotsc,\ 0 \quad \text{(after which the lifts vanish)}.
\]
For the other choice of $*$, the values of $t(a,b)$ are
\[
t,\ t+1,\ \dotsc,\ \lfloor\min(a,b)/2\rfloor-1, \  \lfloor\min(a,b)/2\rfloor, \  \lfloor\min(a,b)/2\rfloor, \ \lfloor\min(a,b)/2\rfloor-1,\ \  \dotsc, \ 1 ,\ 0.
\]
Even though this sequence eventually reaches $0$, to reach $0$ we need a total of $2\lfloor\min(a,b)/2\rfloor+1-t$ steps. In particular, since $t \leq \lfloor\min(a,b)/2\rfloor$, the total number of steps is larger than $\lfloor\min(a,b)/2\rfloor$. However, in the statement of the Lemma we require the total number of steps to be at most $\lfloor\min(a,b)/2\rfloor$ (recall that $k < \lfloor\min(a,b)/2\rfloor$). In short, for this choice of $*$ we cannot reach $0$ using at most $\lfloor\min(a,b)/2\rfloor$ lifts.

This proves the Lemma when $a$ and $b$ are odd: there is exactly one pair $(k,*)$ with $k < \lfloor\min(a,b)/2\rfloor$ such that the parameter corresponding to $\theta_{(k,*)}(\pi)$ has $t(a,b)=0$. (The above discussion shows that $k=t-1$ in that case.)

It remains to explain what happens when $a$ and $b$ are even. In that case, the discussion is essentially the same, the only difference being the second bullet of Lemma \ref{lem_xu}.
We skip straight to the conclusion by listing what happens to $t(a,b)$ as we go through the steps $\pi \mapsto \theta_{-\gamma}(\pi) \mapsto \theta_{(1,*)}(\pi) \mapsto \theta_{(2,*)}(\pi) \mapsto \dots$

For one choice of $*$, we get (as before)
\[
t,\ t-1,\ t-2,\ \dotsc,\ 0 \quad \text{(after which the lifts vanish)}.
\]
For the other choice of $t$, we get
\[
t,\ t+1,\ \dotsc,\ \min(a,b)/2-1,\ \min(a,b)/2,\ \min(a,b)/2-1,\ \min(a,b)/2-2,\ \dotsc 
\]
Eventually, this sequence will reach $0$. Indeed, for $k=\min(a,b)-1-t$, we see that the parameter for $\theta_{(k,*)}(\pi)$ has $t(a,b)=0$. In particular, if we assume $t < \min(a,b)/2$, then $k \geq \min(a,b)/2$. However, in the statement of the Lemma we require $k < \min(a,b)/2$. Thus, when $t < \min(a,b)/2$, there still exists only one pair $(k,*)$ such that $t(a,b)=0$ in the parameter of $\theta_{(k,*)}(\pi)$.

Finally, the exceptional case is $t=\min(a,b)/2$ for $a,b$ even. In that case, $t(a,b)$ varies as
\[
\min(a,b)/2,\ \min(a,b)/2-1, \dotsc,\ 1,\ 0
\]
for both choices of $*$. As a result, we have two pairs $(k,*)$ such that $t(a,b)=0$ in the parameter of $\theta_{(k,*)}(\pi)$: $k=\min(a,b)/2-1$ and we may choose $*$ freely.
\end{proof}
Before moving on to \S \ref{subs_replacement}, we illustrate the above lemma with a few examples.
\begin{exmp}
\label{ex_prvi}
Let $\pi$ be a representation of $\Sp_{44}$ parametrized by
\[
\psi = \overset{+,\ t=1}{1 \otimes S_3 \otimes S_{15} }.
\]
Let us consider the lift $\theta_{-1}(\pi)$. On one tower, we get
\[
\overset{+}{1 \otimes S_1 \otimes S_{1} }\ \oplus \ \overset{+,\ t=0}{1 \otimes S_3 \otimes S_{15} }.
\]
On the other tower, we get
\[
\overset{-}{1 \otimes S_1 \otimes S_{1} }\ \oplus \ \overset{-,\ t=1}{1 \otimes S_3 \otimes S_{15} }.
\]
\end{exmp}
This example is almost trivial: in the Lemma, we consider $k$ going-up lifts after the initial lift, where $k < \lfloor \min(a,b)/2 \rfloor$. In this example, $\min(a,b)=3$, so $k < \lfloor\min(a,b)/2\rfloor$ implies $k=0$, which means that we do not consider going-up lifts after the initial lift. Yet in the above example we still get the expected result: there is only one pair $(k,*)$ for which $t(a,b)=0$ in the parameter for $\theta_{(k,*)}(\pi)$. Indeed, $k=0$, so the only choice that remains is the choice of target tower for the initial lift. On one tower we get $t(a,b)=0$, whereas $t(a,b)=1$ on the other tower.
\begin{exmp} Let $V_{76}$ be the split orthogonal space of dimension $76$, and let $\pi$ be a representation of $\text{O}(V_{76})$ parametrized by
\[
\psi =\overset{+}{1 \otimes S_1 \otimes S_{1} }\ \oplus \  \overset{-,\ t=1}{1 \otimes S_5 \otimes S_{15} }.
\]
We recall that $\pi\otimes \det$ is then parametrized by
\[
\psi =\overset{-}{1 \otimes S_1 \otimes S_{1} }\ \oplus \  \overset{+,\ t=1}{1 \otimes S_5 \otimes S_{15} }.
\]
We consider the lift $\theta_{-3}(\pi)$. On the going-down tower, we get
\[
\overset{+}{1 \otimes S_1 \otimes S_{1} }\ \oplus \ \overset{+}{1 \otimes S_1 \otimes S_{3} }\ \oplus \ \overset{+,\ t=0}{1 \otimes S_5 \otimes S_{15} }.
\]
The next step is $\theta_{-5}^{\text{up}}$. However, as shown in the Lemma, we see that $l(\theta_{-3}^{\text{down}}(\pi))= 19$, so  $\theta_{-5}^{\text{up}}(\theta_{-3}^{\text{down}}(\pi))=0$.

On the going-up tower, we get $\theta_{-3}^{\text{up}}(\pi)$ parametrized by
\[
\overset{-}{1 \otimes S_1 \otimes S_{1} }\ \oplus \ \overset{+}{1 \otimes S_1 \otimes S_{3} }\ \oplus \ \overset{-,\ t=2}{1 \otimes S_5 \otimes S_{15} }.
\]
Computing the next going-up lift will not change $t(a,b)=2$. Indeed, we get $\theta_{-5}^{\text{up}}(\theta_{-3}^{\text{up}}(\pi))$ parametrized by
\[
\overset{+}{1 \otimes S_1 \otimes S_{1} }\ \oplus \ \overset{-}{1 \otimes S_1 \otimes S_{3} }\ \oplus \ \overset{+}{1 \otimes S_1 \otimes S_{5} }\ \oplus \ \overset{+,\ t=2}{1 \otimes S_5 \otimes S_{15} }.
\]
Since we are restricting the total number of lifts to be at most $\lfloor\min(a,b)/2\rfloor = 2$, we need not compute further lifts. As we go through the steps $\pi \mapsto \theta_{-3}^*(\pi) \mapsto \theta_{-5}^{\text{up}}(\theta_{-3}^*(\pi))$, the number $t(a,b)$ exhibits the same pattern that we found in the proof of the Lemma:

When $*=\text{down}$: $1, 0$ (and $\theta_{-5}^{\text{up}}(\theta_{-3}^{\text{down}}(\pi))$ vanishes).

When $*=\text{up}$: $1, 2, 2$.
\end{exmp}

Of course, as mentioned in the proof of the Lemma, the sequence $1,2,2$ may continue further and will eventually reach $0$, but that will require strictly more than $\lfloor\min(a,b)/2\rfloor = 2$ steps. Let us see what happens in the example above:

Having applied two lifts to $\pi$, we obtained $\theta_{(1,\text{up})}(\pi)=\theta_{-5}^{\text{up}}(\theta_{-3}^{\text{up}}(\pi))$ parametrized by
\[
\overset{+}{1 \otimes S_1 \otimes S_{1} }\ \oplus \ \overset{-}{1 \otimes S_1 \otimes S_{3} }\ \oplus \ \overset{+}{1 \otimes S_1 \otimes S_{5} }\ \oplus \ \overset{+,\ t=2}{1 \otimes S_5 \otimes S_{15} }.
\]
The next going-up lift is $\theta_{-7}^{\text{up}}$, giving us $\theta_{(2,\text{up})}(\pi)$ parametrized by
\[
\overset{+}{1 \otimes S_1 \otimes S_{1} }\ \oplus \ \overset{-}{1 \otimes S_1 \otimes S_{3} }\ \oplus \ \overset{+}{1 \otimes S_1 \otimes S_{5} }\ \oplus \ 
\overset{-}{1 \otimes S_1 \otimes S_{7} }\ \oplus \ 
\overset{-,\ t=1}{1 \otimes S_5 \otimes S_{15} }.
\]
Finally, the next lift is $\theta_{-9}^{\text{up}}$, giving us $\theta_{(3,\text{up})}(\pi)$ parametrized by
\[
\overset{-}{1 \otimes S_1 \otimes S_{1} }\ \oplus \ \overset{+}{1 \otimes S_1 \otimes S_{3} }\ \oplus \ \overset{-}{1 \otimes S_1 \otimes S_{5} }\ \oplus \ 
\overset{+}{1 \otimes S_1 \otimes S_{7} }\ \oplus \ 
\overset{-}{1 \otimes S_1 \otimes S_{9} }\ \oplus \ 
\overset{-,\ t=0}{1 \otimes S_5 \otimes S_{15} }.
\]
Here is a similar example, but with $t=2$:
\begin{exmp}
\label{ex_zadnji}
Let $V_{76}$ be the split orthogonal space of dimension $76$, and let $\pi$ be a representation of $\text{O}(V_{76})$ parametrized by
\[
\psi =\overset{+}{1 \otimes S_1 \otimes S_{1} }\ \oplus \  \overset{+,\ t=2}{1 \otimes S_5 \otimes S_{15} }.
\]
We recall that $\pi\otimes \det$ is then parametrized by
\[
\psi =\overset{-}{1 \otimes S_1 \otimes S_{1} }\ \oplus \  \overset{-,\ t=2}{1 \otimes S_5 \otimes S_{15} }.
\]
We consider the lift $\theta_{-3}(\pi)$. When $*=\text{up}$, we get $\theta_{-3}^\text{up}(\pi)$:
\[
\overset{-}{1 \otimes S_1 \otimes S_{1} }\ \oplus \ \overset{+}{1 \otimes S_1 \otimes S_{3} }\ \oplus \ \overset{+,\ t=1}{1 \otimes S_5 \otimes S_{15} }.
\]
The next step is $\theta_{-5}^{\text{up}}$; we get
\[
\overset{+}{1 \otimes S_1 \otimes S_{1} }\ \oplus \ \overset{-}{1 \otimes S_1 \otimes S_{3} }\ \oplus \
\overset{+}{1 \otimes S_1 \otimes S_{5} }\ \oplus \
\overset{+,\ t=0}{1 \otimes S_5 \otimes S_{15} }.
\]

By contrast, when $*=\text{down}$, we get $\theta_{-3}^\text{down}(\pi)$:
\[
\overset{+}{1 \otimes S_1 \otimes S_{1} }\ \oplus \ \overset{+}{1 \otimes S_1 \otimes S_{3} }\ \oplus \ \overset{+,\ t=2}{1 \otimes S_5 \otimes S_{15} }.
\]
The next step is $\theta_{-5}^{\text{up}}$; we get
\[
\overset{-}{1 \otimes S_1 \otimes S_{1} }\ \oplus \ \overset{-}{1 \otimes S_1 \otimes S_{3} }\ \oplus \
\overset{+}{1 \otimes S_1 \otimes S_{5} }\ \oplus \
\overset{+,\ t=1}{1 \otimes S_5 \otimes S_{15} }.
\]
Since we are restricting the total number of lifts to be at most $\lfloor\min(a,b)/2\rfloor = 2$, we need not compute further lifts. As we go through the steps $\pi \mapsto \theta_{-3}^*(\pi) \mapsto \theta_{-5}^{\text{up}}(\theta_{-3}^*(\pi))$, the number $t(a,b)$ exhibits the familiar pattern:

When $*=\text{up}$: $2, 1, 0$.

When $*=\text{down}$: $2, 2, 1$.
\end{exmp}

Again, there is only one combination $(k,*)$ with $k < \lfloor\min(a,b)/2\rfloor$ which yields $t(a,b)=0$.

\subsection{Replacing the highest term}
\label{subs_replacement}
We now use the approach outlined above to modify the parameter we are working with. The end result of this section is Proposition \ref{prop_step} below, which enables us to replace the highest summand in $\psi$ with a singleton ($A_r = B_r$). We use the results from \cite{nas_clanak}, summarized in Appendix \ref{appendix}. At this point we advise the reader to read the introductory part of Appendix \ref{appendix}  to get acquainted with the terminology. Throughout this section, we assume that the parameter $\psi$ satisfies the conditions of Lemma \ref{lem_D_reduction1}; see \eqref{eq_parameter}. We also assume that the corresponding representation $\pi$ has $l(\pi) > 0$; this is justified by Theorem \ref{theoremB}.

Suppose $\beta < d$ is an odd integer such that $\theta_{-\beta}(\pi)$ is in the A-A packet. Although we know the parameter $\psi_\beta$, the fact that $\beta$ is below the range of our Recipe means we lose all information about $\eta$ and $t$. There is a fix, however: if we compute another going-down lift, we regain control of $\eta$ and $t$. This is the content of Lemma \ref{lem_alphabet} below.

In this section, we assume that $A_r > B_r$, i.e.~that the highest segment in $\psi$ is not a singleton. Let $\beta < b_r-a_r+1$ be an odd positive integer such that $\theta_{-\beta}^{\text{down}}(\pi)$ is in the A-A packet. Then, $\chi_W\otimes S_{a_r} \otimes S_{b_r}$ is still the highest term in $\Jord(\psi_\beta)$.
Now let $\alpha > \beta$ be another odd positive integer, which is smaller than $b_r-a_r+1$ but bigger than $\beta$ and $a_i+b_i -1$ for all $i<r$. The Recipe (along with Remark \ref{rem_descent}) then guarantees that $\theta_{-\alpha}^{\text{down}}(\theta_{-\beta}^{\text{down}}(\pi))$ is also in the A-A packet. Call this representation $\pi_{\beta,\alpha}$. The above setup is important: whenever we work with $\pi_{\beta,\alpha}$, we tacitly assume $\beta < \alpha < b_r-a_r+1$ and $\alpha > a_i+b_i-1$ for $i<r$.

The main idea of this section is to compare $\pi_{\beta,\alpha}$ to $\pi_{\beta',\alpha}$ for different  $\beta$ and $\beta'$. The key observation is that these two representations look very much alike. Theorem \ref{downlifts} describes what happens to the standard module when we go from $\pi$ to $\pi_{\beta,\alpha}$. The tempered part gets modified, but this modification is independent of $\beta$. Furthermore, we add two thin chains to the standard module:
\[
1, 2, \dots, \frac{\beta-1}{2} \quad \text{and} \quad 1, 2, \dots, \frac{\alpha-1}{2}.
\]
Thus, the only difference between $\pi_{\beta,\alpha}$ and $\pi_{\beta',\alpha}$ is in the shorter of the two added chains: one stops at $\frac{\beta-1}{2}$, while the other stops at $\frac{\beta'-1}{2}$. In both cases, we still add the same longer chain, $1, 2, \dots, \frac{\alpha-1}{2}$. We use the following convenient notation to summarize this calculus of standard module segments:
\[
\{\pi_{\beta,\alpha}\} = \{\pi\} \cup \{1, 2, \dots, \frac{\beta-1}{2}\} \cup \{1, 2, \dots, \frac{\alpha-1}{2}\}.
\]
Here, and elsewhere, we use $\{\pi\}$ to denote the collection of segments that appear in the $\GL$ part of the standard module of $\pi$; the notation forgets the tempered part of $\pi$.

We now apply the construction $\theta_\uparrow$ from \S \ref{subs_removal} to $\pi_{\beta,\alpha}$ and $\pi_{\beta',\alpha}$. Because we are assuming $B_r \gg A_i$ for $i<r$, we may also assume that $B_r \gg \frac{\alpha-1}{2}$; in particular, $\theta_\uparrow(\pi_{\beta,\alpha})$ and $\theta_\uparrow(\pi_{\beta',\alpha})$ make sense. Recall that $\theta_\uparrow$ is a sequence of lifts determined by the values $t(A_r,B_r,\zeta_r)$ and $\eta(A_r,B_r,\zeta_r)$ corresponding to the last term in the parameter $\psi_{\beta,\alpha}$. Eventually, we will show that these values do not change if we replace $\beta$ by $\beta'$. However, a priori, we do not know this. Hence, we need to make the distinction between $\theta_\uparrow$ (the sequence of lifts determined by $\psi_{\beta,\alpha}$) and $\theta_\uparrow'$ (the analogous sequence of lifts determined by $\psi_{\beta',\alpha}$). In the following lemma, we apply $\theta_\uparrow'$ to both representations:
\begin{lem}
\label{lem_compare}
Let $\alpha > \beta'> \beta$ be as above. Then $\theta_\uparrow'(\pi_{\beta,\alpha})$ is non-zero, and the standard module of $\theta_\uparrow'(\pi_{\beta',\alpha})$ is obtained from the standard module of $\theta_\uparrow'(\pi_{\beta,\alpha})$ by adding the singleton segments $\frac{\beta+1}{2}, \frac{\beta+3}{2}, \dots, \frac{\beta'-1}{2}$:
\[
\{\theta_\uparrow'(\pi_{\beta',\alpha})\} = \{\theta_\uparrow'(\pi_{\beta,\alpha})\} \cup \{\frac{\beta+1}{2}, \frac{\beta+3}{2}, \dots, \frac{\beta'-1}{2}\}.
\]
\end{lem}
\begin{proof}
By the discussion preceding this lemma, we have
\begin{align*}
\{\pi_{\beta,\alpha}\} &= \{\pi\} \cup \{1, 2, \dots, \frac{\alpha-1}{2}\} \cup \{1, 2, \dots, \frac{\beta-1}{2}\} ,\\
\{\pi_{\beta',\alpha}\} &= \{\pi\}\cup \{1, 2, \dots, \frac{\alpha-1}{2}\} \cup \{1, 2, \dots, \frac{\beta-1}{2},\frac{\beta+1}{2},\dotsc,\frac{\beta'-1}{2}\} .
\end{align*}
In particular,
\[
\{\pi_{\beta',\alpha}\} = \{\pi_{\beta,\alpha}\} \cup \{\frac{\beta+1}{2},\dotsc,\frac{\beta'-1}{2}\}.
\]
It thus suffices to show that the sequence of lifts $\theta_\uparrow'$ affects the standard modules of $\pi_{\beta,\alpha}$ and $\pi_{\beta',\alpha}$ in the same way. 

We begin by noticing that $l(\pi_{\beta,\alpha})=l(\pi_{\beta',\alpha})=\alpha$. Indeed, $l(\pi_{\beta',\alpha})=\alpha$ holds by construction (and Remark \ref{rem_tower}). On the other hand, $\pi_{\beta,\alpha}$ differs from $\pi_{\beta',\alpha}$ only in $\{\frac{\beta+1}{2},\dotsc,\frac{\beta'-1}{2}\}$. Since both representations contain a longer thin chain $\{1,\dotsc,\frac{\alpha-1}{2}\}$, Theorem \ref{lofpi} shows $l(\pi_{\beta,\alpha})=l(\pi_{\beta',\alpha})$.

The first lift we compute in $\theta_\uparrow'$ is $\theta_{-\gamma}$. We first assume that this is the lift to the going-up tower ($\pi_{\beta,\alpha}$ and $\pi_{\beta',\alpha}$ share the same going-up tower, because their standard modules have the same tempered part\,---\,see Remark \ref{towerispasseddown}). Consider $\theta_{-\gamma}(\pi_{\beta,\alpha})$. We know $l(\pi_{\beta,\alpha})=l(\pi_{\beta',\alpha})=\alpha$; in particular, since $\gamma > \alpha$, we have $\theta_{-\gamma}^\text{up}(\pi_{\beta,\alpha})\neq 0$. By Theorem \ref{uplifts}, in computing this lift, we affect a certain chain in the standard module of $\pi_{\beta,\alpha}$. The fact that $\alpha > \beta$ means that we can choose the relevant chain of segments from $\{\pi\} \cup \{1, 2, \dots, \frac{\alpha-1}{2}\}$, without using $\{1, 2, \dots, \frac{\beta-1}{2}\}$. Of course, the analogous statement holds for $\pi_{\beta',\alpha}$. Thus, in both cases, the relevant chain of segments comes from $ \{\pi\} \cup \{1, 2, \dots, \frac{\alpha-1}{2}\}$. As a result, the relevant chain is the same for both representations. This shows
\[
\{\theta_{-\gamma}(\pi_{\beta',\alpha})\} = \{\theta_{-\gamma}(\pi_{\beta,\alpha})\} \cup \{\frac{\beta+1}{2}, \frac{\beta+3}{2}, \dots, \frac{\beta'-1}{2}\}.
\]
We now claim that the next lift\,---\,that is, $\theta_{-\gamma-2}^\text{up}$\,---\,will not affect any of the singleton segments from $\{\frac{\beta+1}{2}, \frac{\beta+3}{2}, \dots, \frac{\beta'-1}{2}\}$. To show this, we use Corollary \ref{multiple_uplifts}, with $\theta_{-\gamma}(\pi_{\beta,\alpha})$ (resp.\ $\theta_{-\gamma}(\pi_{\beta',\alpha})$) playing the role of $\pi'$. Since $l(\pi_{\beta,\alpha})=l(\pi_{\beta',\alpha})=\alpha$, Corollary \ref{multiple_uplifts} shows that no singleton segments below $\frac{\alpha+3}{2}$ are affected when computing the lift $\theta_{-\gamma-2}^\text{up}$. In particular, the singletons $\{\frac{\beta+1}{2}, \frac{\beta+3}{2}, \dots, \frac{\beta'-1}{2}\}$ remain unaffected.

For the next step\,---\,that is, $\theta_{-\gamma-4}^\text{up}$\,---\,we can use the same argument (Corollary \ref{multiple_uplifts}) as soon as we show that $l(\theta_{-\gamma}(\pi_{\beta,\alpha})) = l(\theta_{-\gamma}(\pi_{\beta',\alpha}))= \gamma$. The second equality here holds by construction (and Remark \ref{rem_tower}), while $l(\theta_{-\gamma}(\pi_{\beta,\alpha})) = l(\theta_{-\gamma}(\pi_{\beta',\alpha}))$ holds by Theorem \ref{lofpi}. To see this, note that the two representations differ only in $\{\frac{\beta+1}{2}, \frac{\beta+3}{2}, \dots, \frac{\beta'-1}{2}\}$. Moreover, we are assuming that $\theta_{-\gamma}$ is a going-up lift, so Remark \ref{rem_first_lift} shows that the chain that determines $l(\theta_{-\gamma}(\pi_{\beta',\alpha}))$ cannot be thin. Hence, the same chain is present in $\theta_{-\gamma}(\pi_{\beta,\alpha})$, which shows $l(\theta_{-\gamma}(\pi_{\beta,\alpha})) = l(\theta_{-\gamma}(\pi_{\beta',\alpha}))$. 

The remaining steps are analogous: the same argument (equality of first occurrence indices combined with Corollary \ref{multiple_uplifts}) shows that all the subsequent lifts  used to compute $\theta_\uparrow'$ are non-zero, and that none of these subsequent lifts can affect the singletons $\{\frac{\beta+1}{2}, \frac{\beta+3}{2}, \dots, \frac{\beta'-1}{2}\}$. We thus get
\[
\{\theta_\uparrow'(\pi_{\beta',\alpha})\} = \{\theta_\uparrow'(\pi_{\beta,\alpha})\} \cup \{\frac{\beta+1}{2}, \frac{\beta+3}{2}, \dots, \frac{\beta'-1}{2}\},
\]
which is what we needed to show.

It remains to address the case when $\theta_{-\gamma}$ (the first lift needed for $\theta_\uparrow$) is a going-down lift. In that case, $\theta_{-\gamma}$ adds the chain $\{1,\dotsc, \frac{\gamma-1}{2}\}$ to the standard module, by Theorem \ref{downlifts}. The remaining lifts are going-up, so from that point on, we can repeat the same arguments (equality of first occurrence indices and Corollary \ref{multiple_uplifts}) as in the previous case.
\end{proof}
\begin{lem}
\label{lem_alphabet}
Suppose that $\pi$ is parametrized by
\[
\psi =  \bigoplus_{i=1}^{r-1}{\chi_V\otimes S_{a_i} \otimes S_{b_i}} \, \oplus \,   \overset{t, \eta}{\chi_V\otimes S_{a_r} \otimes S_{b_r}},
\]
Let $\alpha > \beta$ be odd positive integers as above. Assume $\theta_{-\beta}^{\text{down}}(\pi)$ and $\pi_{\beta,\alpha} = \theta_{-\alpha}^{\text{down}}(\theta_{-\beta}^{\text{down}}(\pi))$ are both in their respective A-A packets; in particular, $\pi_{\beta,\alpha}$ is parametrized by 
\[
\psi_{\beta,\alpha} = \bigoplus_{i=1}^{r-1}{\chi_V\otimes S_{a_i} \otimes S_{b_i}} \, \oplus \,  \chi_V\otimes S_{1} \otimes S_{\beta}  \, \oplus \,  \overset{\eta''}{\chi_V\otimes S_{1} \otimes S_{\alpha}}  \, \oplus \,   \overset{t', \eta'}{\chi_V\otimes S_{a_r} \otimes S_{b_r}}
\]
for some $\eta''$, $t'$, and $\eta'$.

Assume that the order on $\Jord(\psi_{\beta,\alpha})$ satisfies
$
(\chi_V,a_r,b_r) > (\chi_V, 1, \alpha)\allowbreak > (\chi_V, a_i, b_i)  $ for all $i<r$. Then $t'=t$ and, if $t < \textnormal{min}(a_r,b_r)/2$, $\eta' = \eta$. In particular, $\eta'$ and $t'$ do not depend on the choice of $\beta$. Furthermore, $\eta''$ is also independent of $\beta$.
\end{lem}

\begin{proof}
We begin by explaining why $\eta''$ is independent of $\beta$. By construction, $\pi_{\beta,\alpha}$ is equal to $\theta_{-\alpha}(\theta_{-\beta}(\pi))$. To construct the corresponding parameter, we add $\chi \otimes S_1 \otimes S_\alpha$ to $\psi_\beta$. The sign $\eta''$ on this added term is determined by the target tower. However, we know that we are lifting to the going-down tower. Now the claim follows from the fact that the going-down tower for $\theta_{-\beta}(\pi)$ is determined by the tempered part of its standard module, and this does not depend on $\beta$ (cf.~Theorem \ref{downlifts}).

We now prove that $t'=t$ and $\eta' = \eta$. We point out that we already know this when $\beta$ is big enough. Indeed, as long as $\beta > d$ (so that $\theta_{-\beta}(\pi)$ is computed using the Recipe), this is a direct consequence of the Recipe and the change of order formulas from Lemma \ref{lem_xu}: since we are computing two consecutive going-down lifts, we need to apply the Recipe twice. This means $\eta$ and $t$ on the last term get updated twice, and one checks that the two updates counteract one another.

To prove the Lemma for $\beta < d$, we compare $\pi_{\beta,\alpha}$ to $\pi_{\beta',\alpha}$, where $\beta' > d$.  To be precise, we show that the two transformations $\theta_\uparrow$ and $\theta_\uparrow'$  attached to these representations (introduced prior to Lemma \ref{lem_compare}) are, in fact, one and the same. The uniqueness from Lemma \ref{lem_76v} will then imply the result we need. 

As explained above, we know that the Lemma is true for $\pi_{\beta',\alpha}$. Consider the construction $\theta_\uparrow'(\pi_{\beta',\alpha})$. As pointed out in Lemma \ref{lem_76v}, this is the unique combination of lifts which produces a representation for which $\theta_{a_r+b_r-1}$ is non-zero, at least if $t < \text{min}(a_r,b_r)/2$. We now assume that this is the case, and deal with the case $t = \text{min}(a_r,b_r)/2$ later.

Before moving on, we examine how the standard module gets affected when we go from $\theta_\uparrow'(\pi_{\beta',\alpha})$ to $\theta_{a_r+b_r-1}(\theta_\uparrow'(\pi_{\beta',\alpha}))$: we claim that the chain which gets affected is thick. Indeed, notice that $\theta_\uparrow'(\pi_{\beta',\alpha})$ (or its $\det$ twist) can be viewed as the first going-up lift (i.e.\ $\theta_{-a_r-b_r+1}$) of some smaller representation; this follows from Remark \ref{rem_uparrow} (iv), part 4. Because it is the first going-up lift, Corollary \ref{cor_first_lift} a) shows that the standard module of $\theta_\uparrow'(\pi_{\beta',\alpha})$ does not contain the singleton segment $\frac{a_r+b_r-2}{2}$. Now Theorem \ref{downdownlifts} shows that the chain affected in computing $\theta_{a_r+b_r-1}$ will be thick. 

Now apply the same sequence of lifts (that is, $\theta_\uparrow'$) to $\pi_{\beta,\alpha}$. By Lemma \ref{lem_compare}, the standard modules of $\theta_\uparrow'(\pi_{\beta,\alpha})$ and $\theta_\uparrow'(\pi_{\beta',\alpha})$ differ only by a thin chain. In particular, the thick chain that gets affected when we go from $\theta_\uparrow'(\pi_{\beta',\alpha})$ to $\theta_{a_r+b_r-1}(\theta_\uparrow'(\pi_{\beta',\alpha}))$ is also present in $\theta_\uparrow'(\pi_{\beta,\alpha})$. Since $\theta_\uparrow'(\pi_{\beta',\alpha})$ and $\theta_\uparrow'(\pi_{\beta,\alpha})$ also share the same tempered part, Theorem \ref{lofpi} shows that $\theta_{a_r+b_r-1}$ of $\theta_\uparrow'(\pi_{\beta,\alpha})$ is also non-zero. 

Now, by Lemma \ref{lem_76v}, there is only one combination of lifts (meaning, a choice of target tower for the initial lift, and the number of subsequent going-up lifts) which results in a representation for which $\theta_{a_r+b_r-1}$ does not vanish. This combination is determined by $\eta$ and $t$ on the highest term of the parameter. Since the same combination works for both $\pi_{\beta',\alpha}$ and $\pi_{\beta,\alpha}$, we conclude that $\theta_\uparrow' = \theta_\uparrow$, i.e.\  $\pi_{\beta',\alpha}$ and $\pi_{\beta,\alpha}$ have the same values of $\eta$ and $t$ on the highest term of the parameter. This proves the Lemma in case when $t_r < \text{min}\{a_r,b_r\}/2$.

It remains to consider the case $t = \text{min}(a_r,b_r)/2$. In that case, Lemma \ref{lem_76v} says that there is not one, but two combinations of lifts that could be called $\theta_\uparrow$. However, in both combinations, exactly $t$ lifts are used. Therefore, we can apply the same reasoning as above to conclude $t'=t$. We cannot guarantee that $\eta' = \eta$, but we may assume so: the fact that $t' = \text{min}(a_r,b_r)/2$ means that we can change $\eta'$ to $-\eta'$ without changing the representation that is parametrized (see Remark \ref{rem_parni}).
\end{proof}
In the rest of this section, we keep the same assumptions as above: we assume (for the sake of contradiction) that there exists $\beta < d$ such that $\theta_{-\beta}^{\text{down}}(\pi)$ is in the A-A packet, and we choose (for the sake of comparison) $\beta'$ such that $d<\beta'< b_r-a_r+1$; we know $\theta_{-\beta'}^{\text{down}}(\pi)$ is in the A-A packet by Theorem \ref{theoremA}. We next prove a result (Lemma \ref{lem_replace}) which allows us to remove the highest term of $\psi_{\beta,\alpha}$. This will enable us to effectively replace the highest term of $\psi$ with a singleton segment.

Recall Remark \ref{rem_uparrow} (iv): the representation $\theta_\uparrow(\pi_{\beta,\alpha})$ (or its determinant twist) can be obtained by a sequence of lifts from a smaller representation $\pi_{\beta,\alpha}^0$, parametrized by 
\[
\bigoplus_{i=1}^{r-1}{\chi_V\otimes S_{a_i} \otimes S_{b_i}} \, \oplus \,  \chi_V\otimes S_{1} \otimes S_{\beta}  \, \oplus \,  \chi_V\otimes S_{1} \otimes S_{\alpha}.
\]
This parameter is obtained from $\psi_{\beta,\alpha}$ by removing the highest term. An analogous statement holds for $\theta_\uparrow(\pi_{\beta',\alpha})$. Our initial observation was that the standard modules of $\pi_{\beta',\alpha}$ and $\pi_{\beta,\alpha}$ differ by a thin chain. Lemma \ref{lem_compare} says that the same holds for $\theta_\uparrow(\pi_{\beta,\alpha})$ and $\theta_\uparrow(\pi_{\beta',\alpha})$. The following lemma allows us to extend this comparison to $\pi_{\beta,\alpha}^0$ and $\pi_{\beta',\alpha}^0$.
\begin{lem}
\label{lem_replace}
The standard modules of $\pi_{\beta,\alpha}^0$ and $\pi_{\beta',\alpha}^0$ still differ only in one thin chain: we have
\[
\{\pi_{\beta',\alpha}^0\}=\{\pi_{\beta,\alpha}^0\} 
\cup \{\frac{\beta+1}{2}, \frac{\beta+3}{2}, \dots, \frac{\beta'-1}{2}\}.
\]
\end{lem}
\begin{proof}
Going from $\theta_\uparrow(\pi_{\beta,\alpha})$ (resp.~$\theta_\uparrow(\pi_{\beta',\alpha})$) to $\pi_{\beta,\alpha}^0$ (resp.~$\pi_{\beta',\alpha}^0$) amounts to running the steps from Remark \ref{rem_uparrow} (iv) backwards. We need to show that these steps affect $\theta_\uparrow(\pi_{\beta,\alpha})$ and $\theta_\uparrow(\pi_{\beta',\alpha})$ in the same way.
\begin{itemize}
\item[4.] All the lifts in step 4 are the \emph{first} going-up lifts. By Remark \ref{rem_first_lift}, reversing these lifts will not affect any thin chains in the standard module. Because $\theta_\uparrow(\pi_{\beta,\alpha})$ and $\theta_\uparrow(\pi_{\beta',\alpha})$ have the same thick segments, Theorem \ref{downdownlifts} shows that reversing step 4 will affect them in the same way.
\item[3.] For both $\pi_{\beta,\alpha}$ and $\pi_{\beta',\alpha}$, this is a going-down lift. Reversing it, we need to remove the following thin chain from the standard module (see Theorem \ref{downlifts})
\[
1, 2, \dotsc, \frac{b_r-a_r}{2}.
\]
In doing this, we affect both representations in the same way.
\item[2.] Again, we are reversing the \emph{first} going-up lifts. Thus, we are changing the same (thick) segments in both representations, just like in step 4.
\item[1.] Here we are reversing the lift $\theta_{-\gamma}(\pi_{\beta,\alpha}^0)$ (resp.~$\theta_{-\gamma}(\pi_{\beta',\alpha}^0)$). As explained in Remark \ref{rem_uparrow} (i), we do not know whether this is a going-up lift or a going-down lift. However, this is non-important, as long as we can show that we are reversing the same type of lift for both $\pi_{\beta,\alpha}^0$ and $\pi_{\beta',\alpha}^0$ (because then we will be affecting the two standard modules in the same way).

The parameter of $\theta_{-\gamma}(\pi_{\beta,\alpha}^0)$ looks like
\[
\bigoplus_{i=1}^{r-1}{\chi\otimes S_{a_i} \otimes S_{b_i}} \, \oplus \,  \chi\otimes S_{1} \otimes S_{\beta}  \, \oplus \, {\chi\otimes S_{1} \otimes S_{\alpha}}  \, \oplus \,   {\chi\otimes S_{1} \otimes S_{\gamma}},
\]
as reversing steps 4, 3, and 2 results in removing all the higher terms. (The parameter for $\pi_{\beta',\alpha}^0$ is the same, with $\beta$ replaced by $\beta'$.) Whether this is the going-up lift or the going-down lift is determined by the values $\eta(\chi\otimes S_{1} \otimes S_{\alpha})$ and  $\eta(\chi\otimes S_{1} \otimes S_{\gamma})$: if they are the same, we are looking at the going-down lift; if they differ, then it is a going-up lift (cf.~Remark \ref{rem_tower}).

The second part of Lemma \ref{lem_alphabet} shows that $\pi_{\beta',\alpha}$ and $\pi_{\beta,\alpha}$ have the same $\eta(\chi\otimes S_{1} \otimes S_{\alpha})$. This equality remains true when we add summands to construct $\theta_\uparrow(\pi_{\beta',\alpha})$ and $\theta_\uparrow(\pi_{\beta,\alpha})$; similarly, the equality between signs holds after we remove summands to construct $\theta_{-\gamma}(\pi_{\beta,\alpha}^0), \theta_{-\gamma}(\pi_{\beta',\alpha}^0)$.

Furthermore, because $\eta(\chi\otimes S_{1} \otimes S_{\gamma})$ is determined by the sign $\eta$ on the highest summand in $\psi_{\beta,\alpha}$, it follows from the first part of Lemma \ref{lem_alphabet} that $\pi_{\beta',\alpha}$ and $\pi_{\beta,\alpha}$ also share the same sign $\eta(\chi\otimes S_{1} \otimes S_{\gamma})$.

We conclude that we are reversing the same type of lift for both representations. Therefore, the standard modules are affected in the same way. (Indeed, if $\theta_{-\gamma}$ is a going-down lift, this is trivial; in the going-up case, one verifies this using Theorem \ref{downdownlifts}\,---\,or reversing Theorem \ref{uplifts}\,---\,and the fact that $l(\pi_{\beta',\alpha}^0)=l(\pi_{\beta,\alpha}^0)=\alpha$.)
This is what we needed to show.\qedhere
\end{itemize}
\end{proof}

Let us take a moment to examine $\pi_{\beta',\alpha}^0$ and $\pi_{\beta,\alpha}^0$. Recall that $\pi_{\beta',\alpha}^0$ is parametrized by
\[
\bigoplus_{i=1}^{r-1}{\chi\otimes S_{a_i} \otimes S_{b_i}} \, \oplus \,  \chi\otimes S_{1} \otimes S_{\beta'}  \, \oplus \,  \chi\otimes S_{1} \otimes S_{\alpha}.
\]
Here we are still using $\chi$ as a placeholder for either $\chi_V$ or $\chi_W$: the number of lifts we had to perform to get to this representation depends on the size of the last segment in $\psi$, i.e. $A_r-B_r$. Recall that $\beta'$ was chosen so that $\chi \otimes S_1 \otimes S_{\beta'}$ is higher than all the other summands except $\chi \otimes S_1 \otimes S_{\alpha}$. The following observation is simple, but crucial:
\begin{lem}
\label{lem_10}
$\eta(\chi \otimes S_1 \otimes S_{\beta'}) = \eta(\chi \otimes S_1 \otimes S_{\alpha})$.
\end{lem}
\begin{proof}
The construction $\pi_{\beta',\alpha} \mapsto \theta_\uparrow(\pi_{\beta',\alpha}) \mapsto \pi_{\beta',\alpha}^0$ does not change the relation between $\eta(\chi \otimes S_1 \otimes S_{\beta'})$ and $\eta(\chi \otimes S_1 \otimes S_{\alpha})$, so it suffices to prove that the same  equality holds in the parameter $\psi_{\beta',\alpha}$ of $\pi_{\beta',\alpha}$.

The claim follows from the fact that $\pi_{\beta',\alpha} = \theta_{-\alpha}^{\text{down}}(\theta_{-\beta'}^{\text{down}}(\pi))$ is obtained by lifting to the going-down tower twice. Let us explain; we observe how the two lifts affect the parameter. We start from $\pi$, given by

\[
\psi = \bigoplus_{i=1}^{r-1}{\chi_V\otimes S_{a_i} \otimes S_{b_i}} \, \oplus \,   \overset{t_r, \eta_r}{\chi_V\otimes S_{a_r} \otimes S_{b_r}},
\]
and then use the Recipe to obtain $\pi_\beta'$, parametrized by
\[
\psi_{\beta'}= \bigoplus_{i=1}^{r-1}{\chi_W\otimes S_{a_i} \otimes S_{b_i}} \, \oplus \, \overset{\eta}{\chi_W \otimes S_1 \otimes S_{\beta'}} \, \oplus \,  \overset{t'_r, \eta'_r}{\chi_W\otimes S_{a_r} \otimes S_{b_r}},
\]
(a priori, one does not know how $\eta$ and $\eta_r'$ are related).
To get $\pi_{\beta',\alpha}$, we lift once more. We use the Recipe again: we first compute $\theta_{-\alpha}$ for $\alpha \gg 0$, and then decrease $\alpha$, using Lemma \ref{lem_xu} to skip over the highest summand. For $\alpha \gg 0$ we have $\pi_{\beta',\alpha}$, given by
\[
\psi_{\beta',\alpha} =  \bigoplus_{i=1}^{r-1}{\chi_V\otimes S_{a_i} \otimes S_{b_i}} \, \oplus \, \overset{-\eta}{\chi_V \otimes S_1 \otimes S_{\beta'}} \, \oplus \,  \overset{t'_r, -\eta'_r}{\chi_V\otimes S_{a_r} \otimes S_{b_r}} \, \oplus \, \overset{\eta''}{\chi_V \otimes S_1 \otimes S_{\alpha}}.
\]
Here $\eta''$ is determined by the requirement that this be the going-down tower. Now there are two cases:
\begin{enumerate}
\item $t_r'  > 0$. In this case, $d^{\text{up}}(\pi_{\beta'},\psi_\beta')=\beta'+2$ (see Remark \ref{rem_tower}), and the going-down lift means that the added summand can descend past $\chi_V \otimes S_1 \otimes S_{\beta'}$. By Remark \ref{rem_descent}, this means precisely that the sign $\eta''$ should be compatible with $-\eta$, so we are done.
\item $t_r'  = 0$. In this case, $d^{\text{up}}(\pi_{\beta'},\psi_\beta') = a_r+b_r+1$ (again, by Remark \ref{rem_tower}), and now the going-down lift means $\eta'' = (-1)^{A_r-B_r}\cdot(-\eta_r')$. But by Lemma \ref{lem_xu}, $t_r'  =0$ can happen only if $t_r = 1$ and $\eta = \eta_r'$.

For $\alpha \gg 0$, we thus get $\pi_{\beta',\alpha}$ parametrized by
\[
\psi_{\beta',\alpha}=\bigoplus_{i=1}^{r-1}{\chi_V\otimes S_{a_i} \otimes S_{b_i}} \, \oplus \, \overset{-\eta}{\chi_V \otimes S_1 \otimes S_{\beta'}} \, \oplus \,  \overset{t=0, -\eta}{\chi_V\otimes S_{a_r} \otimes S_{b_r}} \, \oplus \, \overset{(-1)^{A_r-B_r}\cdot(-\eta)}{\chi_V \otimes S_1 \otimes S_{\alpha}}.
\]
We still need to use the Recipe to decrease $\alpha$ and move the added term below $\chi_V\otimes S_{a_r} \otimes S_{b_r}$. Again, we use Lemma \ref{lem_xu} to see what happens once $\alpha$ drops below $b_r-a_r+1$:
\[
\psi_{\beta',\alpha} =  \bigoplus_{i=1}^{r-1}{\chi_V\otimes S_{a_i} \otimes S_{b_i}} \, \oplus \, \overset{-\eta}{\chi_V \otimes S_1 \otimes S_{\beta'}} \, \oplus \, \overset{-\eta}{\chi_V \otimes S_1 \otimes S_{\alpha}} \, \oplus \,  \overset{t=1, \eta}{\chi_V\otimes S_{a_r} \otimes S_{b_r}}.
\]
Thus in this case we also get $\eta(\chi_V \otimes S_1 \otimes S_{\beta'}) = \eta(\chi_V \otimes S_1 \otimes S_{\alpha})$ ($=-\eta$).\qedhere
\end{enumerate}
\end{proof}
Finally, we can put together the above results. The lemma we just proved shows that $\pi_{\beta',\alpha}^0$ can be viewed as $\theta_{-\beta'}^{\text{down}}(\pi_\alpha^0)$, obtained using the Recipe from the representation $\pi^0_\alpha$ parametrized by
\[
\psi_{\alpha}^0 =  \bigoplus_{i=1}^{r-1}{\chi\otimes S_{a_i} \otimes S_{b_i}} \, \oplus \, {\chi \otimes S_1 \otimes S_{\alpha}}.
\]
But Lemma \ref{lem_replace} shows that the standard module of $\pi_{\beta,\alpha}^0$ is obtained from the standard module of $\pi_{\beta',\alpha}^0=\theta_{-\beta'}^{\text{down}}(\pi_\alpha^0)$ by removing the singleton segments $\frac{\beta'-1}{2}, \dotsc, \frac{\beta+1}{2}$. By Theorem \ref{downlifts}, this means that $\pi_{\beta,\alpha}^0$ can be viewed as a lower going-down lift of the same representation:
\[
\pi_{\beta,\alpha}^0=\theta_{-\beta}^{\text{down}}(\pi_\alpha^0).
\]
In particular, $\theta_{-\beta}^{\text{down}}(\pi_\alpha^0)$ is in the A-A packet even though $\beta < d$.
It remains to prove the following Lemma (whose proof we postpone):
\begin{lem}
\label{lem_same_d}
$d^{\text{down}}(\pi_\alpha^0, \psi_\alpha^0) = d^{\text{down}}(\pi, \psi)$.
\end{lem}
Once we prove this lemma, we will have proven the following
\begin{prop}
\label{prop_step}
Let $\pi$ be a representation parametrized by
\[
\psi =  \bigoplus_{i=1}^{r-1}{\chi_V\otimes S_{a_i} \otimes S_{b_i}} \, \oplus \,   {\chi_V\otimes S_{a_r} \otimes S_{b_r}}
\]
which satisfies the conditions of Lemma \ref{lem_D_reduction1}: all the $\zeta = -1$ summands above $d=d^{\text{down}}(\pi,\psi)$ are far away from $d$ and from each other. Suppose that Theorem \ref{theoremC} fails for $\pi$: there exists a  $\beta < d$ such that $\theta_{-\beta}(\pi)$ is in the A-A packet.

Then, Theorem \ref{theoremC} also fails for the representation $\pi_\alpha^0$ whose parameter
\[
\psi_\alpha^0 = \bigoplus_{i=1}^{r-1}{\chi\otimes S_{a_i} \otimes S_{b_i}} \, \oplus \,   {\chi\otimes S_{1} \otimes S_{\alpha}}.
\]
is obtained by replacing the highest summand in $\psi$ with a singleton segment.
\end{prop}
\noindent This is the key step in an inductive proof of Theorem \ref{theoremC}, which we complete in \S \ref{subs_removing}.
Before moving on, we still need to address Lemma \ref{lem_same_d}.
\begin{proof}[Proof of Lemma \ref{lem_same_d}]
This follows essentially from Lemma \ref{lem_high_blocks_dont_matter}. Recall that $d^{\text{down}}(\pi_\alpha^0, \psi_\alpha^0)$ and $d^{\text{down}}(\pi, \psi)$ are determined by applying the Recipe to $\psi_\alpha^0$ and $\psi$, respectively. Lemma \ref{lem_10} allowed us to conclude that $\theta_{-\beta'}^{\text{down}}(\pi_\alpha^0) = \pi_{\beta',\alpha}^0$; this is parametrized by
\[
\psi_{\beta',\alpha}^0 = \bigoplus_{i=1}^{r-1}{\chi\otimes S_{a_i} \otimes S_{b_i}} \, \oplus \,   {\chi\otimes S_{1} \otimes S_{\beta'}} \, \oplus \,   {\chi\otimes S_{1} \otimes S_{\alpha}}.
\]
On the other hand, we began this section by computing $\theta_{-\beta'}^{\text{down}}(\pi)$, which is parametrized by
\[
\psi_{\beta'} =
\bigoplus_{i=1}^{r-1}{\chi_W\otimes S_{a_i} \otimes S_{b_i}} \, \oplus \,   {\chi_W\otimes S_{1} \otimes S_{\beta'}} \, \oplus \,   {\chi_W\otimes S_{a_r} \otimes S_{b_r}}.
\]
The lower parts of these parameters satisfy the assumptions of Lemma \ref{lem_high_blocks_dont_matter}: we have $t(\chi,a_i,b_i)=t(\chi_W,a_i,b_i)$ and $\eta(\chi, a_i,b_i)=\eta(\chi_W, a_i,b_i)$ for all $i < r$ (or $\eta(\chi, a_i,b_i)=-\eta(\chi_W, a_i,b_i)$ for all $i<r$). Furthermore, $\psi_{\beta',\alpha}^0$ was computed from $\psi_{\beta'}$ by computing a series of lifts: $\theta_{-\alpha}$, $\theta_\uparrow$, etc. Note that none of these lifts changed the sign $\eta(\chi \otimes S_1 \otimes S_{\beta'})$ relative to the signs on the lower terms. Thus, we also have $\eta(\chi,1,\beta')=\eta(\chi_W,1,\beta')$ (resp.\ $\eta(\chi,1,\beta')=-\eta(\chi_W,1,\beta')$).

The above discussion shows that we may apply Lemma \ref{lem_high_blocks_dont_matter} with $\pi' = \pi_\alpha^0$ (the relevant level from Lemma \ref{lem_high_blocks_dont_matter} being $\beta'$) to conclude that $d^{\text{down}}(\pi_\alpha^0, \psi_\alpha^0)=d^{\text{down}}(\pi, \psi)$.
\end{proof}
\subsection{Removing the highest term}
\label{subs_removing}

In \S \ref{subs_replacement} we worked out a strategy for turning the highest segment of $\psi$ into a singleton. It remains to discuss what happens when the highest segment is a singleton.

Recall that we are assuming the parameter $\psi$ of $\pi$ satisfies the conditions of Lemma \ref{lem_D_reduction1}: all the $\zeta = -1$ blocks above $d=d^{\text{down}}(\pi,\psi)$ are far from $d$ and from each other. Suppose that the highest segment in $\Jord(\psi)$ is a singleton:
\[
\psi =  \bigoplus_{i=1}^{r-1}{\chi_V\otimes S_{a_i} \otimes S_{b_i}} \, \oplus \,   {\chi_V\otimes S_{1} \otimes S_{b_r}}.
\]
By Corollary \ref{cor_removal}, this means that either $\pi = \theta_{-b_r}(\pi')$ or $\pi\otimes \det = \theta_{-b_r}(\pi')$, where $\pi'$ is parametrized by
\[
\psi' =  \bigoplus_{i=1}^{r-1}{\chi_W\otimes S_{a_i} \otimes S_{b_i}}.
\]
In the following discussion, we assume $\theta_{-b_r}(\pi')=\pi$ just to simplify notation, employing Convention \ref{conv_dettwist} when necessary.

We would like to know whether $\pi$ is $\theta_{-b_r}^{\text{down}}(\pi')$ or $\theta_{-b_r}^{\text{up}}(\pi')$. We can think of it this way: we may apply the Recipe to lower the term ${\chi_V\otimes S_{1} \otimes S_{b_r}}$ in $\psi$. The Recipe will stop at a certain point (which we call $d(\pi',\psi')$), and we need to determine whether the point at which the Recipe stops is $d^{\text{up}}(\pi',\psi')$ or $d^{\text{down}}(\pi',\psi')$. By Remark \ref{rem_tower}, this depends on the blocks with $t=0$ and $\zeta = -1$ above $d$:
\begin{lem}
\label{lem_up_or_down}
\begin{enumerate}[(i)]
\item If $\chi_V\otimes S_{1} \otimes S_{b_r}$ is the only high block with $t=0$, then $\pi$ is the going-up lift of $\pi'$. Moreover, $d^{\text{down}}(\pi,\psi) = d^{\text{up}}(\pi',\psi')$.
\item If there are other high blocks with $t=0$, then $\pi$ is the going-down lift of $\pi'$. Moreover, $d^{\text{down}}(\pi,\psi) = d^{\text{down}}(\pi',\psi')$.
\end{enumerate}
\end{lem}
\begin{proof} In both cases, $d^{\text{up}}(\pi,\psi) = b_r+2$. To find $d^{\text{down}}(\pi,\psi)$, we look at $\theta_{-\alpha}^{\text{down}}(\pi)$ for $\alpha \gg 0$, parametrized by
\[
\psi_\alpha =  \bigoplus_{i=1}^{r-1}{\chi_W\otimes S_{a_i} \otimes S_{b_i}} \, \oplus \,   {\chi_W\otimes S_{1} \otimes S_{b_r}}\, \oplus \,   {\chi_W\otimes S_{1} \otimes S_{\alpha}},
\]
and then use the Recipe to decrease $\alpha$. Since we are computing the going-down lift, we have $\eta({\chi_V\otimes S_{1} \otimes S_{\alpha}})\allowbreak = \eta(\chi_V\otimes S_{1} \otimes S_{b_r})$. Thus the sign $\eta({\chi_V\otimes S_{1} \otimes S_{\alpha}})$ is the same as the sign $\eta(\chi_V\otimes S_{1} \otimes S_{b_r})$ relative to the lower summands. By Lemma \ref{lem_high_blocks_dont_matter}, this implies that $d^{\text{down}}(\pi,\psi) = d(\pi',\psi')$; here $d(\pi',\psi')$ is the point at which the recipe stops when we descend in the tower on which we get $\pi$. It remains to determine whether $d(\pi',\psi')$ is $d^{\text{up}}(\pi',\psi')$ or $d^{\text{down}}(\pi',\psi')$, which brings us to the two cases from the statement of the Lemma:
\begin{enumerate}[(i)]
\item In this case, there are no terms with $t=0$ above $d(\pi',\psi')$ in $\psi'$. By Lemma \ref{lem_obstacle}  (applied to $\pi'$), this means that $d(\pi',\psi')$ cannot be $d^{\text{down}}(\pi',\psi')$. Therefore, $d(\pi',\psi') = d^{\text{up}}(\pi',\psi')$.
\item In this case, there exist terms with $t=0$ above $d(\pi',\psi')$. By Remark \ref{rem_tower}, the highest of these terms determines $d^{\text{up}}(\pi',\psi')$. Therefore, $d(\pi',\psi')$ is lower than $d^{\text{up}}(\pi',\psi')$, so we conclude $d(\pi',\psi') = d^{\text{down}}(\pi',\psi')$.\qedhere
\end{enumerate}
\end{proof}
We are now ready to end the proof, following the strategy outlined in \S \ref{subs_outline}. Of the above two cases, the second one is easy to handle:
\begin{lem}
\label{lem_step}
Let $\pi$ be parametrized with
\[
\psi =  \bigoplus_{i=1}^{r-1}{\chi_V\otimes S_{a_i} \otimes S_{b_i}} \, \oplus \,   {\chi_V\otimes S_{1} \otimes S_{b_r}}
\]
and let $\pi' = \theta_{b_r}(\pi)$ as above. Suppose that we are in case (ii) of Lemma \ref{lem_up_or_down}: $\pi = \theta_{-b_r}^{\text{down}}(\pi')$.

Assume that Theorem \ref{theoremC} fails for $\pi$: there exists an odd integer $\beta < d$ such that $\theta_{-\beta}(\pi)$ is in the A-A packet. Then Theorem \ref{theoremC} also fails for $\pi'$: $\theta_{-\beta}(\pi')$ is in the A-A packet.
\end{lem}
\begin{proof}
We are assuming that $\pi$ is the going-down lift of $\pi'$, i.e.\,$\pi = \theta_{-b_r}^{\text{down}}(\pi')$. Suppose that Theorem \ref{theoremC} fails for $\pi$, so that $\theta_{-\beta}^{\text{down}}(\pi) = \theta_{-\beta}^{\text{down}}(\theta_{-b_r}^{\text{down}}(\pi'))$ is in the packet parametrized by
\[
\bigoplus_{i=1}^{r-1}{\chi_W\otimes S_{a_i} \otimes S_{b_i}} \, \oplus \,   {\chi_W\otimes S_{1} \otimes S_{\beta}}\, \oplus \,   {\chi_W\otimes S_{1} \otimes S_{b_r}}.
\]
The highest term in this parameter is still $ {\chi_W\otimes S_{1} \otimes S_{b_r}}$, and we can take $\theta_{b_r}$ of the above representation to remove it: we get a representation $\pi'' := \theta_{b_r}(\theta_{-\beta}^{\text{down}}(\pi))$ parametrized by 
\[
\psi'_\beta = \bigoplus_{i=1}^{r-1}{\chi_V\otimes S_{a_i} \otimes S_{b_i}} \, \oplus \,   {\chi_V\otimes S_{1} \otimes S_{\beta}}.
\]
However, by Corollary \ref{downlifts_commute}, the going-down lifts commute:
\[
\theta_{-\beta}^{\text{down}}(\theta_{-b_r}^{\text{down}}(\pi')) = \theta_{-b_r}^{\text{down}}(\theta_{-\beta}^{\text{down}}(\pi')).
\]
Therefore (using Convention \ref{conv_dettwist})
\[
\pi'' = \theta_{b_r}(\theta_{-\beta}^{\text{down}}(\pi)) =  \theta_{b_r}(\theta_{-\beta}^{\text{down}}(\theta_{-b_r}^{\text{down}}(\pi')))  =  \theta_{b_r}( \theta_{-b_r}^{\text{down}}(\theta_{-\beta}^{\text{down}}(\pi'))) = \theta_{-\beta}^{\text{down}}(\pi').
\]
This shows that $\theta_{-\beta}^{\text{down}}(\pi')$ is parametrized by $\psi'_\beta$, which is precisely the A-A parameter for the lift of $\pi'$. By Lemma \ref{lem_up_or_down}, $d(\pi',\psi') = d$. Since $\beta < d$, this shows that Theorem \ref{theoremC} fails for $\pi'$ as well.
\end{proof}
Finally, we have all we need to prove Theorem \ref{theoremC}.
\begin{proof}[Proof of Theorem \ref{theoremC}]
Let $\pi$ be parametrized by
\[
\psi =  \bigoplus_{i=1}^{r}{\chi_V\otimes S_{a_i} \otimes S_{b_i}}.
\]
By Lemma \ref{lem_D_reduction1}, we may assume that $\pi$ satisfies \eqref{eq_parameter}: all the $\zeta = -1$ terms above $d = d^{\text{down}}(\pi,\psi)$ are far from $d$ and from each other.

Suppose, for the sake of contradiction, that Theorem \ref{theoremC} fails for $\pi$. If $a_r < b_r$, we may use Proposition \ref{prop_step} to replace the highest term with a singleton segment. Now we are in the situation addressed by Lemma \ref{lem_up_or_down}. If we are in case (ii) of that lemma (so that $\pi$ is the going-down lift of $\pi'$), then we use Lemma \ref{lem_step} to remove the highest term, and reduce the problem to a smaller parameter. We repeat this inductively, until we reach situation (i) from \ref{lem_up_or_down}: the highest term is a singleton, and there are no other terms with $t=0$ above $d$.

The upshot of the above discussion is the following: we may assume the highest term of $\pi$ is a singleton segment (i.e.~$a_r=1$), and moreover, that $\pi$ is the going-up lift of $\pi' = \theta_{b_r}(\pi)$. We are assuming that Theorem \ref{theoremC} fails for $\pi$: there exists a $\beta < d$ for which $\theta_{-\beta}^{\text{down}}(\pi)$ is parametrized by
\[
\bigoplus_{i=1}^{r-1}{\chi_W\otimes S_{a_i} \otimes S_{b_i}} \, \oplus \,   {\chi_W\otimes S_{1} \otimes S_{\beta}}\, \oplus \,   {\chi_W\otimes S_{1} \otimes S_{b_r}}.
\]
Once more, the highest term in this parameter shows that $\theta_{b_r}(\theta_{-\beta}^{\text{down}}(\pi)) \neq 0$. But $\pi = \theta_{-b_r}^{\text{up}}(\pi')$, so this implies
\[
\theta_{b_r}(\theta_{-\beta}^{\text{down}}(\theta_{-b_r}^{\text{up}}(\pi'))) \neq 0,
\]
with $\beta < d= d^{\text{up}}(\pi',\psi')=l(\pi')+2$. Corollary \ref{key} shows that this is impossible. We have thus reached a contradiction. This proves Theorem \ref{theoremC}.\qedhere
\end{proof}

\appendix
\refstepcounter{section}
\section*{Appendix \thesection: Chains}\label{appendix}

To prove Theorem \ref{theoremC}, we need precise information about the lifts $\theta_{-\alpha}(\pi)$ for $\alpha > 0$, on both the going-up and the going-down tower. These lifts have been completely determined (in terms of $L$-parameters) in \cite{nas_clanak}. In an attempt to keep the present paper self-contained, we use this appendix to summarize the main results of \cite{nas_clanak}. The main theorems are recorded here as \ref{lofpi}, \ref{downlifts}, \ref{downdownlifts}, and \ref{uplifts}. We also list a number of corollaries which contain statements in a ready-to-use form suitable for the proof of Theorem \ref{theoremC}. In stating the results, we do not strive for utmost precision; rather, we state them in a way which allows us to use them in our proofs, while not overburdening the reader with technical details. For the unabridged version of the results, we refer to \cite{nas_clanak}.

Any irreducible representation is completely determined by its standard module. It is convenient to state the results in terms of standard modules. Recall \S \ref{subsubs_Langlands}: a standard module is any representation of the form 
\[
\nu^{s_k}\delta_k \times \dotsb \times \nu^{s_1}\delta_1 \rtimes \tau,
\]
with $s_k \geq \dots \geq s_1 > 0$, where $\tau$ is an irreducible tempered representation of a classical group, and each $\delta_i$ is an irreducible discrete series representation of a general linear group. Recall that each $\nu^{s_i}\delta_i$ corresponds to a cuspidal segment: $\delta_i = \delta([x_i,y_i]_{\rho_i})$; see \S \ref{subsubs_segments}.
To specify a standard module, one thus needs
\begin{itemize}
\item an irreducible tempered representation; and
\item a collection (multiset) of segments $[x_i,y_i]_{\rho_i}$ with $x_i+y_i >0$.
\end{itemize}
Accordingly, to describe the theta lifts, one needs 
\begin{itemize}
\item information about lifts of tempered representations\,---\,this is supplied by Atobe and Gan \cite{Atobe_Gan};
\item combinatorics involving the segments of the standard module.
\end{itemize}
In fact, in this paper, we are only interested in segments $[x,y]_\rho$ where $\rho$ is trivial and $x,y$ are integers. (To be fully precise, we consider $\rho=\chi_V$ and $\rho = \chi_W$, but we suppress this from the notation, in accordance with Remark \ref{rem_chi}.) Thus, in this appendix, a \emph{segment} will always mean a cuspidal segment with $\rho = 1$ and $x,y \in \mathbb{Z}$. Often, we will talk about singleton segments. In that case, we write $x$ instead of $[x,x]$. By contrast, any segment $[x,y]$ for which $y-x >0$ is called thick. It is convenient to allow a segment to be empty, that is, of the form $[x,x-1]$ for $x \in \mathbb{Z}$ (the corresponding representation is trivial). The key combinatorial concept we need is that of a \emph{chain} of segments; this is a special case of what we call a RS-ladder in \cite{nas_clanak}:
\begin{defn}
A sequence $[c_1,d_1], [c_2,d_2], \dotsc, [c_k, d_k]$ of segments is called a chain if
\begin{itemize}
\item $d_{i+1} = d_i+1$, for $i=1,\dotsc, k-1$; and
\item $c_{i+1} > c_i$, for $i=1,\dotsc, k-1$.
\end{itemize}
(In particular, the segments $[c_{i+1},d_{i+1}]$ and $[c_{i},d_{i}]$ are linked.)\\
We call a chain thin if all its segments are singletons; we call it thick if none are singletons.
\end{defn}
We set up the notation needed to state our theorems. 
For a representation $\pi$, we let
\[
l(\pi)=\max \{l \geq -1: \theta_{l}(\pi)\neq 0 \text{ on the going down tower}\}.
\]
Note the edge case $l(\pi) = -1$, when $\theta_{-1}(\pi)$ is the first lift on both towers (cf.~Remark \ref{rem_updown}).
Throughout the Appendix, we let $\tau$ denote the tempered representation that appears in the standard module of $\pi$. The results of \cite{Atobe_Gan} describe the lifts of $\tau$ in terms of its Langlands parameter $\phi$. Let $S_d$ denote the unique irreducible $d$-dimensional algebraic representation of $\SL_2(\mathbb{C})$. We use $m_\phi(S_{d})$ to denote the multiplicity of $S_d$ in $\phi$. We refer the reader to \S 3 of \cite{Atobe_Gan} for details.
%
The first result is about non-vanishing:
\begin{thm}[Theorem 5.7 of \cite{nas_clanak}]
\label{lofpi}
Consider all chains $[c_1,d_1], \dotsc, [c_k,d_k]$ in the standard module of $\pi$ such that 
\begin{enumerate}[(i)]
\item $d_1 = \frac{l(\tau)+1}{2}$ and 
\item $c_1 \neq  \frac{1-l(\tau)}{2}$ if $l(\tau) >0$ and $m_\phi(S_{l(\tau)})$ is even.
\end{enumerate}
Let $t$ be the length of the longest chain satisfying these conditions. Then $l(\pi) = l(\tau)+2t$.
 \end{thm}
\begin{rem}
\label{towerispasseddown}
In particular, this shows that $l(\tau)=-1$ implies $l(\pi)=-1$. On the other hand, if $l(\tau)>0$, then the going-down (resp.\ going-up) tower of $\pi$ is inherited from $\tau$; see Proposition 5.2 of \cite{nas_clanak}.
\label{whichtower}
\end{rem}

The next theorem describes the lifts $\theta_{-\alpha}^{\text{down}}(\pi)$ for $\alpha >0$:
\begin{thm}[Theorem 6.7 of \cite{nas_clanak}]
\label{downlifts}
Let $\pi$ be an irreducible representation and let $\alpha > 0$ be an odd integer. On the going-down tower, the standard module of $\theta_{-\alpha}(\pi)$ is obtained from the standard module of $\pi$ by replacing $\tau$ with $\theta_{-1}(\tau)$ (which is tempered) and adding the singleton segments $1, 2, \dotsc, \frac{\alpha-1}{2}$.
\end{thm}
\begin{rem}
\label{multi_S1}
Roughly speaking, the Langlands parameter of $\theta_{-1}(\tau)$ is obtained by adding one copy of $S_1$ to $\phi$; see Theorem 4.3 (2) of \cite{Atobe_Gan}.
\end{rem}
\begin{cor}
\label{downlifts_commute}
On the going-down tower, the lifts commute: for any odd $\alpha, \beta >0$, we have $\theta_{-\alpha}^{\text{down}}(\theta_{-\beta}^{\text{down}}(\pi))\allowbreak = \theta_{-\beta}^{\text{down}}(\theta_{-\alpha}^{\text{down}}(\pi))$.
\end{cor}
\begin{cor}
\label{higher_lift}
Let $\alpha > 0$ and $\beta \gg 0$. Then $\theta_{-\beta}^{\text{down}}(\pi)$ is the unique irreducible quotient of $\zeta(\frac{\alpha+1}{2},\frac{\beta-1}{2}) \rtimes \theta_{-\alpha}^{\text{down}}(\pi)$.
\end{cor}
\begin{proof}
For any irreducible representation $\sigma$, the following is true: if $\beta \gg 0$, then $\zeta(\frac{\alpha+1}{2},\frac{\beta-1}{2}) \rtimes \sigma$ has a unique irreducible quotient, whose standard module is obtained by adding the singletons $\frac{\alpha+1}{2},\dotsc, \frac{\beta-1}{2}$ to the standard module of $\sigma$. 

One proves this by writing $\zeta(\frac{\alpha+1}{2},\frac{\beta-1}{2})$ as a quotient of a principal series representation; then, using irreducibility results for general linear groups, one  rearranges the segments to observe that $\zeta(\frac{\alpha+1}{2},\frac{\beta-1}{2}) \rtimes \sigma$ is a quotient of the relevant standard module.

The corollary follows from this observation (applied to $\sigma =\theta_{-\alpha}^{\text{down}}(\pi)$) and the description of lifts in Theorem \ref{downlifts}.
\end{proof}
Next, we describe the lifts $\theta_\alpha(\pi)$ for $\alpha > 0$ (this only makes sense on the going-down tower):
\begin{thm}[Theorem 6.1 of \cite{nas_clanak}]
\label{downdownlifts}
    Let $\alpha > 0$ be an odd integer such that $\theta_\alpha(\pi)\neq 0$. Let $[c_1,d_1], \dotsc, [c_t,d_t]$ be the longest chain in the standard module of $\pi$ with $d_t = \frac{\alpha-1}{2}$. If there are multiple such chains, take the one which minimizes segment widths. We set $t=0$ if such a chain does not exist. There are now two cases:
    \begin{enumerate}[(i)]
        \item $c_1 > 1-d_1$. Then $\theta_{\alpha-2t}(\tau)$ is non-zero and tempered. The standard module of $\theta_{\alpha}(\pi)$ is obtained by replacing $[c_i,d_i]$ with  $[c_i,d_i-1]$ for $i=1,\dotsc,t$ (if $d_i=c_i$, we simply eliminate the segment), and $\tau$ by $\theta_{\alpha-2t}(\tau)$.
        \item $c_1 = 1-d_1$. Let $\tau'$ be the Langlands quotient of $\delta(1-d_1,d_1) \rtimes \tau$. Then $\theta_{\alpha-2t+2}(\tau')$ is non-zero and tempered. The standard module of $\theta_{\alpha}(\pi)$ is obtained by replacing $[c_i,d_i]$ with  $[c_i,d_i-1]$ for $i=2,\dotsc,t$, removing $[c_1,d_1]$, and replacing $\tau$ by $\theta_{\alpha-2t+2}(\tau')$.
    \end{enumerate}
\end{thm}

We now describe the going-up lifts. Recall that, by the Conservation Relation, the first lift of $\tau$ on the going-up tower is $\theta_{-2-l(\tau)}(\tau)$.
\begin{thm}[Theorem 6.8 of \cite{nas_clanak}]
\label{uplifts} Let $\alpha > l(\pi)$ be an odd integer, so that $\theta_{-\alpha}^{\text{up}}(\pi) \neq 0$. We distinguish two cases:
\begin{enumerate}[(i)]
\item $m_\phi(\chi_VS_{l(\tau)})$ is odd;
\item $m_\phi(\chi_VS_{l(\tau)})$ is even.
\end{enumerate}
Let $t= \frac{\alpha-l(\tau)}{2}-1$ and let $[c_1,d_1], [c_2,d_2], \dotsc, [c_t, d_t]$ be the chain in the standard module of $\pi$ which maximizes segment widths such that $d_1 = \frac{l(\tau)+1}{2}$ and, in case (ii), $c_1 > 1-d_1$ (we allow $t=0$, in which case the chain formally contains no segments). If there is no such chain of length $t$, we take the longest available chain and complete it with empty segments $[c_i,d_i] = [d_i+1,d_i]$ to achieve length $t$.  Then
\begin{enumerate}[(i)]
\item In case (i), $\theta_{-2-l(\tau)}(\tau)$ is tempered. The standard module of $\theta_{-\alpha}(\pi)$ is obtained by replacing $\tau$ with $\theta_{-2-l(\tau)}$, and replacing each segment $[c_i,d_i]$ with $[c_i,d_i+1]$.
\item In case (ii), $\theta_{-2-l(\tau)}(\tau)$ is non-tempered; it is equal to the Langlands quotient of the standard module $\delta(\frac{1-l(\tau)}{2},\frac{l(\tau)+1}{2}) \rtimes \sigma$ for a certain tempered representation $\sigma$ for which $l(\sigma)\geq l(\tau)$. The standard module of $\theta_{-\alpha}(\pi)$ is obtained by replacing $\tau$ with $\sigma$, adding $[\frac{1-l(\tau)}{2},\frac{l(\tau)+1}{2}]$ to the collection of segments, and replacing each $[c_i,d_i]$ with $[c_i,d_i+1]$.
\end{enumerate}
\end{thm}
We now derive some consequences of this theorem. By the Conservation Relation the first going-up lift is $\theta_{-l(\pi)-2}^{\text{up}}(\pi)$. The above theorem says that the chain $[c_1,d_1], \dotsc, [c_t,d_t]$ may need to be completed with empty segments. The following remark clarifies the situation (this is Remark 6.9 from \cite{nas_clanak}).
\begin{rem}
\label{rem_first_lift}
We need empty segments to form the chain $[c_1,d_1], \dotsc, [c_t,d_t]$ any time we compute $\theta_{-\alpha}^{\text{up}}(\pi)$ with $\alpha > l(\pi)+2$ (i.e.\,in every lift except the first occurrence). To be precise, the segments with $d_i > \frac{l(\pi)-1}{2}$ will be empty (formally: $[c_i,d_i]=[d_i+1,d_i]$), and will thus turn into singleton segments when we replace them by $[c_i,d_i+1]$. In short,
\[
[c_i,d_i+1] = \begin{cases}
\text{singleton}, \quad &\text{if } d_i > \frac{l(\pi)-1}{2};\\
\text{thick}, \quad &\text{if } d_i \leq \frac{l(\pi)-1}{2}.
\end{cases}
\]
\end{rem}
We isolate some direct consequences of the above remark:
\begin{cor}\mbox{}
\label{cor_first_lift}
\begin{enumerate}[a)]
\item The standard module of $\theta_{-\alpha}^{\text{up}}(\pi)$ does not include the singleton segment $\frac{l(\pi)+1}{2}$.
\item If the standard module of $\theta_{-\alpha}(\pi)$ does not include the singleton segment $\frac{\alpha-1}{2}$, we may conclude that $\theta_{-\alpha}(\pi)$ is the first lift on the going-up tower.
\item Let $\alpha > l(\pi)+2$. Then the standard modules of $\theta_{-\alpha}^{\text{up}}(\pi)$ and $\theta_{-\alpha}^{\text{down}}(\pi)$ are the same above $\frac{\alpha-1}{2}$ ($\Sigma_{\geq\frac{\alpha-1}{2}}$, using the notation of \S \ref{subs_stdmod}).
\end{enumerate}
\end{cor}
\begin{proof}
Parts a) and b) follow directly from the above remark. For part c), we argue as follows. By theorem A.5, to obtain the standard module of $\theta_{-\alpha}^{\text{up}}(\pi)$, we alter some segments that appear in the standard module of $\pi$, and add certain singleton segments to complete the chain. By Remark \ref{rem_first_lift}, the highest existing segment that gets altered ends in $\frac{l(\pi)-1}{2}$; the added singletons are $\frac{l(\pi)+3}{2}, \frac{l(\pi)+5}{2}, \dotsc, \frac{\alpha-1}{2}$. But these are the same singletons that we add when computing the going-down lift. Thus, the standard modules are the same above $\frac{l(\pi)+3}{2}$.
\end{proof}
We now state another useful corollary of Theorem \ref{uplifts}. Let $\tau'$ denote the tempered part in the standard module of $\theta_{-\alpha}^{\text{up}}(\pi)$. Thus $\tau' = \theta_{-2-l(\tau)}(\tau)$ in case (i), and $\tau' = \sigma$ in case (ii). Let $\phi'$ denote the corresponding L-parameter.
\begin{cor}
\label{tau}
$m_{\phi'}(S_1)$ is odd whenever $l(\tau) > 0$.
\end{cor}
\begin{proof}
This result is a direct consequence of \cite[Theorems 4.1, 4.5]{Atobe_Gan}. The condition $l(\tau) > 0$ simply excludes the edge case $l(\tau)=-1$ (cf.~Remark \ref{rem_updown}). 

Indeed, if $l(\tau) >1$, then $m_\phi(S_1)$ is odd by \cite[Theorem 4.1]{Atobe_Gan}. Then \cite[Theorem 4.5 (1)]{Atobe_Gan} shows that $m_{\phi'}(S_1) = m_\phi(S_1)$; thus  $m_{\phi'}(S_1)$ is odd as well. If $l(\tau) = 1$ it could happen that $m_\phi(S_1)$ is even, but then $m_{\phi'}(S_1)$ is odd by \cite[Theorem 4.5 (2)]{Atobe_Gan}. 
\end{proof}

At a certain point in the proof of Theorem \ref{theoremC}, we need to understand what happens when we compute consecutive going-up lifts. Let us set the scene. Let $\pi' = \theta_{-\alpha}^\text{up}(\pi)$, $\alpha > l(\pi)$, be a going-up lift of $\pi$. As before, let $\tau'$ denote the tempered part of $\pi'$ (this does not depend on $\alpha$).
Here is the result we need:
\begin{cor}
\label{multiple_uplifts}
When computing another going-up lift of $\pi'$ (i.e.\ $\theta_{-\beta}^\text{up}(\pi')$ for $\beta > l(\pi')$), we do not affect any singleton segments smaller than $\frac{l(\pi)+3}{2}$ in the standard module of $\pi'$.
\end{cor}
\begin{proof}
Recall Theorem \ref{uplifts}: when computing $\theta_{-\beta}^\text{up}(\pi')$, we affect a certain chain of segments, which we denote $[c_i',d_i']$. Recall that the lowest segment $[c_1',d_1']$ has $d_1' = \frac{l(\tau')+1}{2}$. Furthermore, if there are many chains to choose from, we choose the chain which maximizes segment widths.

We claim that the standard module of $\pi'$ always contains at least one thick chain of segments whose ends range from $\frac{l(\tau')+1}{2}$ to $\frac{l(\pi)+1}{2}$. Indeed, by Theorem \ref{uplifts}, when constructing $\pi'$ we create a chain whose lowest segment ends in $\frac{l(\tau)+3}{2}$ (in case (i)) or $\frac{l(\tau)+1}{2}$ (in case (ii)), while the highest segment ends in $\frac{l(\pi')-1}{2}$. In both cases, the end of the lowest segment is less than or equal to $\frac{l(\tau')+1}{2}$. Moreover, by Remark \ref{rem_first_lift}, a segment belonging to this chain is thick if and only if its endpoint is less than or equal to $\frac{l(\pi)+1}{2}$. We conclude that $\pi'$ contains a thick chain whose ends range from $\frac{l(\tau')+1}{2}$ to $\frac{l(\pi)+1}{2}$.

Since we are maximizing segment widths when choosing the segments $[c_i',d_i']$, this implies the desired fact: all the segments with $d_i' < \frac{l(\pi)+3}{2}$ are necessarily thick.
\end{proof}

\noindent Finally, these results combine into a corollary that is critical for the proof of Theorem C.
\begin{cor}
\label{key}
Let $\alpha > l(\pi)$, so that $\theta_{-\alpha}^{\text{up}}(\pi)\neq 0$. Take $\beta \leq l(\pi)$ and consider the representation $\theta_{-\beta}^{\text{down}}(\theta_{-\alpha}^{\text{up}}(\pi))$. Then
\[
\theta_{\alpha}(\theta_{-\beta}^{\text{down}}(\theta_{-\alpha}^{\text{up}}(\pi))) = 0.
\]
\end{cor}
\begin{proof}
Theorem \ref{uplifts} tells us how to compute $\theta_{-\alpha}^{\text{up}}(\pi)$: we replace $\tau$ by $\tau'$, and replace a certain chain $[c_i,d_i]$ with $[c_i,d_i+1]$. By Corollary \ref{cor_first_lift} a), the standard module of $\theta_{-\alpha}^{\text{up}}(\pi)$ does not contain the singleton segment $\frac{l(\pi)+1}{2}$. By Corollary \ref{tau}, $m_{\phi'}(S_1)$ is odd.

Having established this, we compute $\theta_{-\beta}^{\text{down}}(\theta_{-\alpha}^{\text{up}}(\pi))$; to simplify notation, call this representation $\sigma$. Recall what computing $\theta_{-\beta}^{\text{down}}$ does to the standard module: we update the tempered part, and add the thin chain $1,\dotsc,\frac{\beta-1}{2}$ to the standard module. In particular, by Remark \ref{multi_S1}, the multiplicity of $S_1$ in the L-parameter of the tempered part is now even (because it was odd in $\phi'$).

Now suppose, for the sake of contradiction, that $\theta_{\alpha}(\sigma) \neq 0$. Call this representation $\pi'$. Then we have two cases:
\begin{enumerate}[(i)]
\item $\sigma = \theta_{-\alpha}^{down}(\pi')$. In this case, Theorem \ref{downlifts} shows that $\sigma$ would have to contain a thin chain $1, \dotsc, \frac{\alpha-1}{2}$. However, this is impossible because the standard module of $\sigma$ does not contain the singleton segment $\frac{l(\pi)+1}{2}$. Indeed, we know that $\theta_{-\alpha}^{\text{up}}(\pi)$ does not contain it. To compute $\sigma$, we simply add $1,\dotsc,\frac{\beta-1}{2}$. But we are assuming $\beta \leq l(\pi)$, so $\frac{l(\pi)+1}{2}$ is still missing. Thus, this case is impossible.
\item $\sigma =  \theta_{-\alpha}^{up}(\pi')$. In this case, Corollary \ref{tau} shows that the multiplicity of $S_1$ in the L-parameter of $\sigma$ should be odd. But we know that it is even. Therefore, this case is impossible as well.
\end{enumerate}
In either case, we have reached a contradiction. We conclude that $\theta_{\alpha}(\sigma) = 0$, which we needed to show.
\end{proof}

\bibliographystyle{siam}
\bibliography{bibliography}

\end{document}